\documentclass[11.5pt]{article}
\usepackage{appendix}
\usepackage{jmlr2e}
\usepackage{amsmath,amssymb,graphicx,mathrsfs,bbm,url}
\usepackage{amsmath}
\usepackage{hyperref}
    \hypersetup{colorlinks=true,citecolor=blue,linkcolor=blue,urlcolor=black}

\usepackage{natbib}

\usepackage[table]{xcolor}

\usepackage{float}
\usepackage{adjustbox}
\usepackage{graphicx}
\usepackage{subcaption}
\usepackage{caption}
\usepackage{lscape}
\usepackage{booktabs}
\usepackage{multirow}

\usepackage{xparse}

\usepackage{comment}

\usepackage{tikz-cd}
\newcommand{\eqdef}{\ensuremath{\stackrel{\mbox{\upshape\tiny def.}}{=}}}

\usepackage[colorinlistoftodos]{todonotes}
\usepackage{soul}

\def\1{\bm{1}}

\def\calH{{\mathcal{H}}}

\def\kkk{{\mathcal{K}}}% \def\calK{{\mathcal{K}}}

\def\uuu{{\mathcal{U}}}% \def\calU{{\mathcal{U}}}

\NewDocumentCommand{\F}{o}{
    \IfValueT{#1}{
            \mathbb{F}_{#1}
        }
    \IfValueF{#1}{
            \mathbb{F}
        }
                    }

\NewDocumentCommand{\R}{o}{
    \IfValueT{#1}{
            \mathbb{R}^{#1}
        }
    \IfValueF{#1}{
            \mathbb{R}
        }
                    }
                    
\NewDocumentCommand{\N}{o}{
    \IfValueT{#1}{
            \mathbb{N}^{#1}
        }
    \IfValueF{#1}{
            \mathbb{N}
        }
                    }

\newcommand{\E}{\mathbb{E}}

\definecolor{darkcyan}{rgb}{0.0, 0.55, 0.55}
\definecolor{MidnightBlue}{RGB}{25,25,112}
\definecolor{MidnightBlueComplementingGreen}{RGB}{25,112,25}
\definecolor{MidnightBlueComplementingPurple}{RGB}{112,25,112}
\definecolor{MidnightBlueComplementingRed}{RGB}{112,25,69}
\definecolor{WowColor}{rgb}{.75,0,.75}
\definecolor{MildlyAlarming}{rgb}{0.85,0.25,0.1}
\definecolor{SubtleColor}{rgb}{0,0,.50}
\definecolor{antiquefuchsia}{rgb}{0.57, 0.36, 0.51}
\definecolor{fashionfuchsia}{rgb}{0.96, 0.0, 0.63}
\definecolor{jade}{rgb}{0.0, 0.66, 0.42}
\definecolor{caribbeangreen}{rgb}{0.0, 0.8, 0.6}
\definecolor{aquamarine}{rgb}{0.5, 0.8, 0.85}
\definecolor{lightseagreen}{rgb}{0.13, 0.7, 0.67}
\definecolor{darkgreen}{rgb}{0.0, 0.2, 0.13}
\definecolor{darkspringgreen}{rgb}{0.09, 0.45, 0.27}
\definecolor{attentioncolor}{RGB}{152,90,81}
\definecolor{burgred}{RGB}{40,3,22}
\definecolor{AnnieGreen}{RGB}{17,123,92}
\definecolor{Turquoise}{RGB}{64,224,208}

\definecolor{darkjade}{RGB}{0,122,84}
\definecolor{Window1}{RGB}{92,150,31}%
    \definecolor{Window1dark}{RGB}{41,67,13}%
\definecolor{Window2}{RGB}{255,168,28}
    \definecolor{Window2dark}{RGB}{114,75,12}
\definecolor{Window3}{RGB}{255,96,33}
    \definecolor{Window3dark}{RGB}{97,36,12}
\definecolor{InputColor}{RGB}{20,255,177}
    \definecolor{InputColorlight}{RGB}{222,237,229}

\definecolor{RedAlizarin}{rgb}{0.82, 0.1, 0.26}

\usepackage{cleveref}
\newcounter{termcounter}
\renewcommand{\thetermcounter}{\Roman{termcounter}}
\crefname{term}{term}{terms}
\creflabelformat{term}{#2\textup{(#1)}#3}

\makeatletter
\def\term{\@ifnextchar[\term@optarg\term@noarg}%]
\def\term@optarg[#1]#2{%
  \textup{#1}%
  \def\@currentlabel{#1}%
  \def\cref@currentlabel{[][2147483647][]#1}%
  \cref@label[term]{#2}}
\def\term@noarg#1{%
  \refstepcounter{termcounter}%
  \textup{(\thetermcounter)}%
  \cref@label[term]{#1}}
\makeatother
\usepackage{marginnote}
\definecolor{MidnightBlue}{RGB}{25,25,112}
\definecolor{MidnightBlueComplementingGreen}{RGB}{25,112,25}
\definecolor{MidnightBlueComplementingPurple}{RGB}{112,25,112}
\definecolor{MidnightBlueComplementingRed}{RGB}{112,25,69}
\definecolor{coolblack}{rgb}{0.0, 0.18, 0.39}

\definecolor{deepjunglegreen}{rgb}{0.0, 0.29, 0.29}
\definecolor{applegreen}{rgb}{0.55, 0.71, 0.0}
\definecolor{WowColor}{rgb}{.75,0,.75}
\definecolor{MildlyAlarming}{rgb}{0.85,0.25,0.1}
\definecolor{SubtleColor}{rgb}{0,0,.50}
\definecolor{SubtleColor2}{rgb}{0.6,0.21,.50}

\definecolor{lasallegreen}{rgb}{0.03, 0.47, 0.19}
\newcounter{margincounter}

\NewDocumentCommand{\AK}{mo}{
    \IfValueF{#2}{
                        {{\scriptsize
                            \textcolor{deepjunglegreen}{
                            \hfill\\
                                \textbf{A:}
                                \textit{{#1}}
                            \hfill\\
                            }
                        }}
        }
    \IfValueT{#2}{
                        \marginnote{{\scriptsize
                            \textcolor{deepjunglegreen}{ 
                            \textbf{A:}
                            \textit{{#1}}
                            }
                        }}
        }
                    }

\NewDocumentCommand{\Xuwei}{mo}{
    \IfValueF{#2}{
                        {{\scriptsize
                            \textcolor{coolblack}{ 
                            \textbf{X:}
                            \textit{{#1}}
                            }
                        }}
        }
    \IfValueT{#2}{
                        \marginnote{{\scriptsize
                            \textcolor{coolblack}{ 
                            \textbf{X:}
                            \textit{{#1}}
                            }
                        }}
        }
                    }

\NewDocumentCommand{\IE}{mo}{
    \IfValueF{#2}{
                        {{\scriptsize
                            \textcolor{orange}{ 
                            \textbf{IE:}
                            \textit{{#1}}
                            }
                        }}
        }
    \IfValueT{#2}{
                        \marginnote{{\scriptsize
                            \textcolor{orange}{ 
                            \textbf{IE:}
                            \textit{{#1}}
                            }
                        }}
        }
                    }

\NewDocumentCommand{\GA}{mo}{
    \IfValueF{#2}{
                        {{\scriptsize
                            \textcolor{violet}{ 
                            \textbf{GA:}
                            \textit{{#1}}
                            }
                        }}
        }
    \IfValueT{#2}{
                        \marginnote{{\scriptsize
                            \textcolor{violet}{ 
                            \textbf{GA:}
                            \textit{{#1}}
                            }
                        }}
        }
                    }

\def\E{\mathbb{E}}

\def\R{\mathbb{R}}
\def\N{\mathbb{N}}

\def\cR{\mathcal{R}}

\begin{document}

\title{Neural Operators Can Play Dynamic Stackelberg Games}

% \maketitle


	\author{\name Guillermo A. Alvarez\thanks{Equal contribution, all authors are listed in alphabetic order. }
	\email guialv@umich.edu\\
	\addr University of Michigan\\
	Department of Mathematics\\
    2074 East Hall, 530 Church Street Ann Arbor, Michigan, USA
	\AND
	\name Ibrahim Ekren\footnotemark[1]
	\email iekren@umich.edu\\
	\addr University of Michigan\\
	Department of Mathematics\\
    2074 East Hall, 530 Church Street Ann Arbor, Michigan, USA
    \AND
    \name Anastasis Kratsios\footnotemark[1] \thanks{Corresponding author.}
	\email kratsioa@mcmaster.ca\\
    \addr McMaster University and the Vector Institute\\
	Department of Mathematics\\
    1280 Main Street West, Hamilton, Ontario, L8S 4K1, Canada
    \AND Xuwei Yang\footnotemark[1]
    \email yangx212@mcmaster.ca\\
	\addr McMaster University\\
	Department of Mathematics\\
    1280 Main Street West, Hamilton, Ontario, L8S 4K1, Canada
	}
	
	\editor{}

\maketitle

%%%%%%%%%%%%%%%%%%%%%%%%%%%
\begin{abstract} 
Dynamic Stackelberg games are a broad class of two-player games in which the leader acts first, and the follower chooses a response strategy to the leader's strategy. Unfortunately, only stylized Stackelberg games are explicitly solvable since the follower's best-response operator (as a function of the control of the leader) is typically analytically intractable.  This paper addresses this issue by showing that the \textit{follower's best-response operator} can be approximately implemented by an \textit{attention-based neural operator}, uniformly on compact subsets of adapted open-loop controls for the leader.  
We further show that the value of the Stackelberg game where the follower uses the approximate best-response operator approximates the value of the original Stackelberg game.  Our main result is obtained using our universal approximation theorem for attention-based neural operators between spaces of square-integrable adapted stochastic processes, as well as stability results for a general class of Stackelberg games.
\end{abstract}

\section{Introduction}
In the classical formulation of Stackelberg games, there are generally two players: a leader (major) who moves first and a follower (minor) who then reacts.  
One is typically interested in studying the equilibrium of these games, in which both players cannot increase their utilities by (re)acting differently.  The generic structure of these games has led their equilibria to become a powerful mathematical tool to describe the evolution of incentives in complex environments with economic applications ranging from control of disease transmission in epidemiology~\cite{aurell2022optimal,hubert2022incentives}, contract design~\cite{conitzer2006computing,elie2019tale,keppo2024dynamic,hernandez2024time,hernandez2024closed}, advertising~\cite{he2008cooperative}, mobile network planning~\cite{zheng2018stackelberg}, economic behaviour in oligopolies~\cite{carmona2021mean}, brokerage~\cite{alvarez2023optimal}, (re)insurance~\cite{cao2022stackelberg,kroell2023optimal,ghossoub2024stackelberg}, risk management~\cite{bensalem2020prevention,li2022cooperative}, algorithmic auction/mechanism design~\cite{conitzer2006computing,konrad2007generalized,dierks2022cloud}, security~\cite{kar2017trends,an2011refinement}, and green investments~\cite{zhang2023optimization}.  Most of these applications consider \textit{dynamic} Stackelberg games, where the game is played continuously in several rounds.  Even though these games provide powerfully descriptive theoretical vehicles, the highly intertwined structure of Stackelberg equilibria can be challenging both numerically and analytically.

This paper shows that deep learning can provide a viable and generic computational vehicle by which dynamic Stackelberg games can be computationally solved.  
We exhibit a class of \textit{neural operators} leveraging an attention mechanism which can approximately implement the follower's best response map to arbitrary precision, uniformly on compact subsets of the leader's strategies (defined below).  
Unlike most neural operator models, which focus on learning the solution map of PDEs~\cite{kovachki2021universal,lanthaler2022error,lee2023hyperdeeponet,lanthaler2023curse,kovachki2023neural,benitez2023out,raonic2024convolutional,bartolucci2024representation,fanaskov2024spectral} or inverse-problems~\cite{calderon1997artificial,molinaro2023neural,de2022deep,kratsios2024mixture}, our neural operators are not defined between function spaces but between spaces of stochastic controls.  
Our attention mechanism mimics that of~\cite{vaswani2017attention} used in transformers~\cite{bahdanau2014neural} while reflecting the geometry of the input and output spaces of stochastic processes, and it extends the attention mechanisms of \cite{acciaio2023designing}.
We note that, it is natural to consider compact sets of controls not only from approximation-theoretic vantage point but also from the control-theoretic perspective; this is because this guarantees the existence of a Stackelberg equilibrium under only minimal assumptions.

Here, we consider the following class of dynamic Stackelberg games, with stochastic effects, where both players (re)act in continuous time according to the following general dynamics 
\begin{align} 
d X_t = f(X_t, u^0_t, u^1_t) dt 
 + \sigma (X_t, u^0_t, u^1_t) d W_t 
 \notag 
 \end{align} 
where $W_{\cdot}\eqdef (W_t)_{t\ge 0}$ is $d$-dimensional standard Brownian motions and $u_{\cdot}^i\eqdef (u_t^i)_{t\ge 0}$ are the (re)actions/strategies of each player, where $i=0$ is the leader and $i=1$ indexes the follower. Each player seeks to optimize their respective objective functions, one in which we impose minimal continuity requirements since we are not interested in analytic expression, which would require highly stylized assumptions on the dynamics and objective functions of all involved players.  Rather, our goal is to show that a deep learning solution via neural operators is possible for a broad class of Stackelberg games lying outside the score of these classical stylized settings.  
Our first main result (Theorem~\ref{thrm:Main__BestResponse}) shows under enough strong-convexity requirements on the utility of the follower the best response maps depend continuously on each leader's actions, then there is a neural operator which can approximate the follower's best response map, uniformly over any compact set of actions of the leader, to any given precision. 

In general, it is well-known that the approximation of non-linear maps/operators between infinite-dimensional Hilbert spaces by deep learning models may be practically challenging due to necessarily slow convergence rates; see e.g.~\cite{lanthaler2023curse}, which is effectively an exacerbated version of the curse of dimensionality known in the finite-dimensional setting, see e.g.~\cite{ShenYangZhang_2022_OptimalReLU}.  Unlike the finite-dimensional setting, sufficient smoothness is insufficient to obtain reasonable convergence rates in general,~\cite{galimberti2022designing}, and typically, one has to hope for favorable structures which can be exploited by the neural operator; see e.g.~\cite{marcati2023exponential}, to obtain fast convergence rates. Fortunately, we identify a set of structures that can be exploited by our neural operator for a class of Stackelberg games that encompass analytically-solvable linear-quadratic games.  Our second main result (Theorem~\ref{thrm:Main__BestResponse___goodrates}) shows that if the compact set of controls is compatible with the best-response map, then one may guarantee efficient convergence rates for neural operator approximations of the best response map for the follower.  
We conclude that neural operators can efficiently approximate the solutions to a wide class of Stackelberg games, encompassing those games which are solvable via classical analytic means.

Additionally, we identify an \textit{unsupervised} objective function which provides a heuristic helping to detect the suitability of a neural operator approximating the best response map of the follower (Theorem~\ref{thrm:Objective}).  Importantly, this heuristic objective function does not require observations of the true optimal response (which would be an irrealistic supervised problem).

We use control theoretical arguments to prove that the leader and the follower have optimal strategies with enough continuous dependence on one another to be \textit{uniformly} approximable on compact sets of (re)actions.  We also provide an illustrative counterexample showing that the best-response map might fail to be continuous without these strong-convexity requirements.  Thus, without the strong-convexity assumption, discontinuous functions cannot be uniformly approximated by any continuous models due to the Uniform Limit Theorem; see e.g.~\cite[Theorem 21.6]{MunkresTopBook}.   

On the technical front: our main control-theoretic contributions (Lemma~\ref{lm:stabilityJ0} together with~\ref{prop:Sufficient_ContinuityBeforeDiscretization}) show that generically the optimal response of the follower is $1/2$-H\"{o}lder continuous in the leader's control if the problem of the follower is strongly convex. 
Our main approximation theoretic contribution is a quantitative universal approximation theorem (Theorem~\ref{thrm:UniversalApprox}) showing that neural operators are capable of approximating H\"{o}lder continuous non-linear operators between space of square-integrable $\mathbb{F}$-adapted processes (open-loop controls). Together, these control-theoretic and approximation-theoretic are enough to show that the best-response map is approximable by neural operators.  Moreover, for suitable sets of controls, our precise analysis both of the regularity of the best response map and the dependence of the neural operator complexity on the set of controls, allow us to conclude that polynomial approximation rates are possible (Theorem~\ref{thrm:Main__BestResponse___goodrates}).

Our neural operators leverage an attention-like decoding layer, similar to transformers~\cite{vaswani2017attention}, which allows for nonlinear decoding, unlike PCA-net~\cite{lanthaler2023operator}, the encoder-decoder models of~\cite{galimberti2022designing}, and several others.  The relationship between our infinite-dimensional analogue of the classical attention of~\cite{bahdanau2014neural} is also discussed. Attention mechanisms in operator learning, by now, have found common use in implementations; see e.g.\ the Galerkin transformers of~\cite{cao2021choose} or the Continuum Attention Mechanism of~\cite{calvello2024continuum}.

\paragraph{Organization of Paper}
Following the literature review in Section~\ref{s:Rel_Lit}, the remainder of our paper is organized as follows:
\begin{enumerate}
    \item[(i)] Section~\ref{s:Rel_Lit} and~\ref{s:prelim}, respectively, review the literature and the necessary background in stochastic analysis and in deep approximation theory to formulate our main results.
    \item[(ii)] Section~\ref{s:Main_Results} contains our main results. 
    \item[(iii)] Section~\ref{s:ProofSketch} explains why our main results work by overviewing our proof strategy, during which we showcase our supporting technical results of independent technical interest.  
    \item[(iv)] Section~\ref{s:Exmaples_Regular Optimal Response} showcases examples of Stackelberg games satisfying our convexity requirements.  
\end{enumerate}
All technical derivations are relegated to appendix~\ref{s:Proofs}.

\section{Related Literature}
\label{s:Rel_Lit}

\textbf{Stackelberg Games in Machine Learning}
In Stackelberg games, both players seek to maximize their gain while being fully rational.  This characteristic ``conditional'' sequential structure is the hallmark challenge rendering Stackelberg games analytically intractable and the reason motivating significant attention from the machine learning community~\cite{reisinger2020rectified,ito2021neural,gao2022achieving,haghtalab2023calibrated,harris2023stackelberg,gerstgrasser2023oracles,dayanikli2023machine}, and their related FBSDEs~\cite{TakAnnieFBSDE_Fase_2024}, looking for new algorithmic tools capable of solving this class of differential games.  Nevertheless, there is currently no available deep learning model which is guaranteed to solve a Stackelberg game, much less in continuous time with stochastic effects.

\paragraph{Neural Operators}
The power of deep learning to solve previously intractable high-dimensional computational problems has motivated the deep learning community to extend these tools to the infinite-dimensional setting with models such as DeepONets \cite{lu2019deeponet,lu2021learning,goswami2022physics} and a variety of \textit{neural operator} architectures; e.g. Fourier Neural Operators~\cite{li2020fourier,kovachki2021universal,li2023fourier}, graph neural operators~\cite{anandkumar2020neural}, causal neural operators~\cite{galimberti2022designing}, neural operator analogues of transformers~\cite{hao2023gnot}, convolutional neural operators~\cite{raonic2024convolutional}, encoder-decoder models such as PCA-net~\cite{lanthaler2022nonlinear}, and a myriad of other models.  The community has largely been motivated by demand in the scientific computing community, focusing on designing neural operators tailored which can learn to solve (i.e.\ learn the solution operator) to high-dimensional partial integro-differential equations (PIDEs) with applications ranging from physics and engineering \cite{de2024error} to quantitative finance \cite{acciaio2023designing}.  However, the full power of neural operators remains otherwise largely unexplored as the community has focused mainly on the approximation capacity of these models between function spaces with a view towards PIDEs.  Here, we probe the limits of neural operators beyond PIDEs by showing that they can solve a broad class of problems at the intersection of game theory and stochastic analysis.  

% \textbf{Notation of Stackelberg Equilibria}
% There exist multiple notions of Stackelberg equilibria depending on the information available to the players. In this paper we focus on the so called open-loop equilibria where the players' strategies are stochastic processes adapted to the  natural filtration generated by a common Brownian motion $W_{\cdot}$. The existing approaches to compute open-loop Stackelberg equilibria in continuous time (see ~\cite{bensoussan2015maximum}) are based on the stochastic maximum principle (Theorems 3.2, 5.2, \cite{JiongminYong_control1999}). The previous results allow to characterize, via the solution of a FBSDE (Forward Backward Stochastic Differential Equation), the follower's optimal response given any leader's strategy $u_0$. As a result, the leader's problem can be reduced to a stochastic control of FBSDE which can be solved by an extended version of the stochastic maximum principle (Jionmig paper's). 

% The previous approach presents multiple challenges and, in practice, very few problems beyond a linear-quadratic setting can be solved following the above method.

\section*{Notation}
We use the following notations.  
For any $C\in \mathbb{N}_+$, we denote the $C$-simplex by $\Delta_C \eqdef \{u\in [0,1]^C:\, \sum_{c=1}^C\, u_c = 1\}$
and we define the associated softmax function by $\operatorname{softmax}_C:\mathbb{R}^C\ni w
                        \mapsto 
                    \big(
                        e^{w_c}/\sum_{c=1}^C e^{w_c}
                    \big)_{c=1}^C \in \Delta_C
$.
Both in $\Delta_C$ and $\operatorname{softmax}_C$, the subscript $C$ will be suppressed when clear from the context.
We will write $f\in \tilde{\mathcal{O}}(g)$ if $f\in \mathcal{O}(g\,\log^k(g))$ for some $k\in \mathbb{N}_+$.

\section{Preliminaries}
\label{s:prelim}

We now overview the background required to formulate the main results of our paper and formalize our neural operator model.  We first overview the notion of a square-integrable predictable process from stochastic analysis and then the definition of a multilayer perceptron (MLP) from deep learning.  Additional details are included in Appendix~\ref{app:Background}.

\subsection{Predictable Processes}
\label{s:predictable_processes}
Fix a time horizon $T>0$ and let $(\Omega,\mathcal{F},\mathbb{F},\mathbb{P})$ be a filtered probability space whose filtration $\mathbb{F}\eqdef (\mathcal{F}_t)_{0\le t\le T}$ is generated by a $d$-dimensional Brownian motion $W_{\cdot}\eqdef (W_t)_{0\le t\le T}$; for some $d\in \mathbb{N}_+$.  We consider the space $\mathcal{H}^2_T$ of all square-integrable $\mathbb{F}$-\textit{predictable processes} on $(\Omega,\mathcal{F},\mathbb{F},\mathbb{P})$ whose elements $H_{\cdot}\eqdef (H_t)_{0\le t\le T}\in \mathcal{H}^2_T$ consist of $d$-dimensional $\mathbb{F}$-predictable processes for which the norm
\[
        \|H\|_{\mathcal{H}^2_T}^2
    \eqdef 
        \mathbb{E}\biggl[
                \int_0^T \, 
                    |H_s|^2
                    d\,s
        \biggr]
,
\]
is finite.  Furthermore, $\mathcal{H}^2_T$ is a separable infinite-dimensional Hilbert space whose inner product $\langle \cdot,\cdot\rangle_{\mathcal{H}^2_T}$, is given for each $H_{\cdot},\tilde{H}_{\cdot}\in \mathcal{H}^2_T$ by
\[
        \langle H_{\cdot},\tilde{H}_{\cdot}\rangle_{\mathcal{H}^2_T}^2
    \eqdef 
        \mathbb{E}\biggl[
                \int_0^T \, 
                    H_s^\top\tilde{H}_s
                    \,ds
        \biggr]
.
\]

Similarly to the Fourier neural operator and the neural operators of, our approximation results will rely on an orthogonal basis%
\footnote{Some authors emphasize that ``orthonormal basis'' is rather a complete orthonormal system since it is not a linear algebraic basis.  However, we follow the common abuse of terminology standard in functional analysis.}~%
of $\mathcal{H}^2_T$.  We begin by constructing a relatively computationally convenient basis of $L^2(\mathcal{F}_t)$ based on the Wiener Chaos decomposition.  A key feature of the following orthonormal basis is that no iterated stochastic integrals need to be explicitly computed, as is the case with general Wiener Chaos decompositions; see Appendix~\ref{app:Background} for additional details.

For any $i\in \mathbb{N}$, the Hermite polynomials $(h_i)_{i\in \mathbb{N}}$ are the eigenfunctions of the generator $\frac{d^2}{dx^2}-x\frac{d}{dx}$ of the Ornstein-Uhlenbeck process $dX_t = -X_t +\sqrt{2}dW_t$.  For each $i\in \mathbb{N}_+$, the $i^{th}$ Hermite polynomial $h_i$ is given recursively by Rodrigues' formula as
\[
h_i(x) = \frac{(-1)^i}{i!} e^{x^2/2} \frac{d^i}{dx^i}e^{-x^2/2} , 
\qquad 
h_0(x) = 1.
\]
The Hermite polynomials allow us to define a family of random variables $\big\{
u^t_{i,j,k}
:\,
j\in \mathbb{N},\,i,k\in \mathbb{N},\,\frac{k+1}{2^i}\le 
% {1}
1
\big\}\subset L^2(\mathcal{F}_t)$ where each $u_{i,j,k}^t$ is defined by
\begin{equation}
    u^t_{i,j,k}
\eqdef 
    \prod_{\tilde{j}=1}^{j}
    h_{{\tilde{j}}}\big(
    2^i \,W_{\frac{tk}{2^i}}
    -2^{i+1}\,W_{\frac{t(1+2k)}{2^{i+1}}}
    +2^i\,W_{\frac{t(k+1)}{2^i}}
    \big) .
    \label{base_l2_ft}
\end{equation}

Using~\eqref{base_l2_ft}, we construct a predictable and dynamic version of the Wiener chaos decomposition.  
First recall the \textit{Haar (wavelet) system} on the larger space $L^2([0,T])$ where 
 $(\psi_{i,k})_{i,k\in \mathbb{N};0\le k,\,\frac{k+1}{2^i}\le 1}$ and  
 \[
    \psi_{i,k}({t}) \eqdef 2^i\big(
            I_{[T\frac{k}{2^i},T\frac{1+2k}{2^{i+1}})}({t})
        -
            I_{[T\frac{1+2k}{2^{i+1}},T\frac{k+1}{2^i})}({t})
    \big)
\]
which is a complete orthonormal basis of $L^2([0,T])$, see \citep[Chapter 3]{MeyerWaveletBook_Vol1_1990}.  The Haar wavelet system will allow us to activate/deactivate the random variables in Wiener Chaos, in~\eqref{base_l2_ft}, as a function of time.  We focus on the collection of simple processes 
obtained as linear combinations of
\begin{equation}
\label{eq:simple_processes}
\begin{aligned}
\mathcal{S}\eqdef & \big\{
u_{i,j,k}^{s_1,s_2}
:\,
i,j,k,s_1,s_2\in \mathbb{N},\, s_2+1\leq 2^{s_1},\frac{k+1}{2^i}\le \frac{s_2}{{2^{s_1}}}
\big\}
\\
    u_{i,j,k}^{s_1,s_2}(t,{\omega})\eqdef & \psi_{s_1,s_2}(t)\cdot u^T_{i,j,k}
(\omega)
.
\end{aligned}
\end{equation}
It can be shown, see Lemma~\ref{lem:Orthonormal} in the proofs section, that $\mathcal{S}$ is an orthonormal basis of $\mathcal{H}^2_T$.  

Furthermore, $\mathcal{S}$ has an elegant interpretation as a Haar wavelet expansion in time and the iterated It\^{o} stochastic integral of a Haar wavelet expansion in space.  The ability to explicitly compute the iterated It\^{o} stochastic integrals of the Haar wavelet system in space makes this closed-form expansion particularly favourable, especially in higher dimensions where the computation of iterated stochastic integrals can be computationally intensive. 

\subsubsection{Examples of Compact Sets of Square-Integrable Controls}
\label{s:predictable_processes__ss:CompactControls}
We will often be considering compact subsets of $\mathcal{H}_T^2$ whereon we frame the existence of Stackelberg equilibria exist and uniform approximation is possible.  
There are several other examples of a compact subset of $\mathcal{H}_T^2$ routinely encountered in the literature, for example, sets of processes which are Malliavin differentiable with uniformly bounded Malliavin Derivative, see~\citep[Corollary C.3.]{BenosDuedahlMeyerBrandisProske_2018_CompactnessMaillivain_AIHP}, or perturbations of closed-loop controls which are efficiently approximable (see Section~\ref{s:EfficientRates}).  Two illustrative, but broad classes, of examples are now constructed; building on compactness results in classical function spaces.

\begin{example}[Compactness Via Regularity of the Malliavin Derivative]
\label{ex:CompactnessviaMalliavin}
    For given $C\geq 0$, and $\alpha\in (0,1)$, define $\kkk_{\alpha,C}\subset\mathcal{H}^2_T$ as the set of $H\in \mathcal{H}^2_T$ so that 
    $$\sup_{0\leq s\leq t\leq T}\E[|D_sH_t|^2+|H_t|^2]\leq C,\,\sup_{0\leq s<t\leq T}\frac{\E[|H_t-H_s|^2]}{|t-s|^\alpha}\leq C,\mbox{ and }\sup_{r,0\leq s<t\leq T}\frac{\E[|D_tH_r-D_sH_r|^2]}{|t-s|^\alpha}\leq C.$$
For $H\in \kkk_{\alpha,C}$, thanks to \cite[Proposition 1.3.8]{Nualart_MalCRT_2006} we can compute the Malliavin derivative
    $D_t\int_0^T H_r dW_r=H_t+\int_t^T D_t H_r dW_r$ for $0\leq s<t\leq T$ so that we have the following estimate 
    \begin{align*}
        &\frac{\E[|D_t\int_0^T H_r dW_r-D_s\int_0^T H_r dW_r|^2]}{|t-s|^\alpha}\\
        &\leq 2\frac{\E[|H_t-H_s|^2]+\E[|\int_s^t D_s H_r dW_r|^2]+\E[|\int_t^T D_t H_r-D_s H_r dW_r|^2]}{|t-s|^\alpha}\\
        &\leq 2\frac{\E[|H_t-H_s|^2]}{|t-s|^\alpha}+2\frac{\int_s^t \E[|D_s H_r|^2]dr}{|t-s|^\alpha}+2\int_t^T\frac{\E[ |D_t H_r-D_s H_r |^2]}{|t-s|^\alpha}dr\leq 2C(1+T^{1-\alpha}+T).
    \end{align*}
    Thanks to \cite[Corollary C3]{BenosDuedahlMeyerBrandisProske_2018_CompactnessMaillivain_AIHP}, this estimate impiles that the set of random variables $\{\int_0^T H_s dW_s: H\in \kkk_{\alpha,C}\} $ is relatively compact in the set of square integrable random variables. This implies by Ito's isometry that $\kkk_{\alpha,C}$ is relatively compact in $\mathcal{H}^2_T$.
\end{example}

In a very special case of Example~\ref{ex:CompactnessviaMalliavin}, one may require the predictable process $H_{\cdot}$ to be a (non-random) smooth function.  This next example shows precisely this, and it elucidates the link between classical families of smooth functions and compact sets of open-loop controls.  

\begin{example}[Compactness Via Continuously Differentiable Martingale Controls]
\label{ex:Sobolev_Copy}
Fix $T>0$.
Let $W^{1,2}([0,1])$ denote the Sobolev space on the unit interval $[0,1]$; and denote its norm by $\|\cdot\|_{W^{1,2}}$.
By the Rellich-Kondrashov Theorem, see e.g.~\cite[Theorem 5.1]{EvansPDE_Book_2010}, the set of function $\varsigma:[0,1]\to \mathbb{R}$ satisfying 
\begin{equation}
\label{eq:Sobolev_Unit_Ball}
\|\varsigma\|_{W^{1,2}} \le 1
\end{equation}
is compact in $L^2([0,1])$.  By the Ito isometry the set of $\mathcal{F}_T$-measurable random variables $\int_0^T\, \varsigma(t)\,dW_t$ is therefore compact in $L^2(\mathcal{F}_T)$.  Since conditional expectations are non-expansive ($1$-Lipschitz) then the set $\mathcal{K}\subset\mathcal{H}_T^2$ of processes/open-loop controls $u\eqdef (u_t)_{0\le t\le T}$ of the form
\[
        u_t 
    = 
        \mathbb{E}\biggr[\int_0^T\, \varsigma(s)\,dW_s \Big| \mathcal{F}_t\biggl]
    = 
        \int_0^t\, \varsigma(s)\,dW_s 
\]
where $\varsigma$ satisfies~\eqref{eq:Sobolev_Unit_Ball}, is compact in $\mathcal{H}_T^2$, and the last inequality held by the Martingale property of the It\^{o} (stochastic) integral.
\end{example}
\begin{example}[Deterministic H\"{o}lder Continuous Controls]
\label{ex:bounded_Holder}
Fix $T>0$, $0<\alpha \le 1$, and consider the set $\mathcal{K}_{\alpha}$ of deterministic controls $u_{\cdot}=(u_t)_{t\ge 0}$ in $\mathcal{H}_T^2$  where $t\mapsto u_t$ is an $\alpha$-H\"{o}lder function mapping $[0,T]$ to $[-1,1]^d$.  
In this case, for each $u_{\cdot},v_{\cdot}\in \mathcal{K}_{\alpha}$ we have
\[
        \mathbb{E}\biggl[\int_0^T\, |u_t-v_t|^2\biggr]^{1/2}
    =
        \biggl(
            \int_0^T\, |u_t-v_t|^2
        \biggr)^{1/2}
    \le 
        \max_{0\le t\le T}\, |u_t-v_t|
.
\]
Therefore, the map $C([0,1],\mathbb{R}^d)\to \mathcal{H}_T^2$ is a $1$-Lipschitz embedding when the domain is equipped with the uniform norm.  By the Arzel\'{a}–Ascoli Theorem, we have that any set of uniformly bounded $\alpha$-H\"{o}lder functions is relatively compact in $C([0,1],\mathbb{R}^d)$; thus, $\mathcal{K}_{\alpha}$ is relatively compact in $\mathcal{H}_T^2$.
\end{example}

One can easily extend the construction in Example~\ref{ex:Sobolev_Copy} to non-Martingale controls iterated integrals using the isometries between the space of symmetric functions in $L^2([0,1]^q)$, for any $q\in \mathbb{N}_+$, and the $q^{th}$ Wiener Chaos (see e.g.~\cite[Theorem 1.1.1]{Nualart_MalCRT_2006}); however, do not do so for simplicity of presentation.
% \newpage
\begin{example}[Conditioned Lipschitz Perturbations of Random Variables at Terminal Time]
\label{ex:Compacta_ByParameterized_Maps}
Fix $X\in L^2(\mathcal{F}_T)$, and fix $T>0$.  Consider the set $\mathcal{X}$ of $1$-Lipschitz
function $f:\mathbb{R}^d\to [-1,1]^d$ which are supported on the hypercube $[-1,1]^d$; i.e. $f(x)=0$ if $x\not\in [-1,1]^d$.    
By the Arzela-Ascoli Theorem, $\mathcal{X}$ is relatively compact in $C(\mathbb{R}^d,\mathbb{R}^d)$.  
Since the map sending any $f\in C(\mathbb{R}^d,\mathbb{R}^d)$ to $f(X)\in L^2(\mathcal{F}_T)$ is $1$-Lipschitz then the set of random variables 
$\{f(X)\in L^2(\mathcal{F}_T):\, f\in \mathcal{X}\}$ is compact in $L^2(\mathcal{F}_T)$.  As in Example~\eqref{ex:Sobolev_Copy}, since conditional expectations are $1$-Lipschitz then the set of controls $u_{\cdot}\eqdef (u_t)_{t\ge 0}\in \mathcal{K}\subset \mathcal{H}_T^2$ of the form
\[
        u_t 
    = 
        \mathbb{E}\big[
                f(X)
            \big| 
                \mathcal{F}_t
        \big]
\]
where $f\in \mathcal{X}$, is relatively compact in $\mathcal{H}_T^2$.
As a concrete example, one may take $f$ to belong to the set of $1$-Lipschitz ReLU Neural Networks with output restricted to belong to $[-1,1]^d$; see e.g.~\cite[Theorem 1.1]{hong2024bridging}.
\end{example}
\begin{remark}[{Alternative Proofs of Compactness Directly Via Example~\ref{ex:CompactnessviaMalliavin}}]
Example~\ref{ex:Compacta_ByParameterized_Maps} can also be obtained from Example~\ref{ex:CompactnessviaMalliavin} upon adding Malliavin differentiability requirements on $X$ and additional regularity.  Examples~\ref{ex:Sobolev_Copy} and~\ref{ex:bounded_Holder} can have alternatively be obtained as straightforward consequences of Example~\ref{ex:CompactnessviaMalliavin}.  We opted for self-contained presentations for each example to illustrate various construction methods for compacta in $\mathcal{H}_T^2$.
\end{remark}
In practice, one often discretizes their space when implementing it on a digital machine.  In these cases, the set of controls is finite and, therefore, compact.
\begin{example}[Finite Sets of Controls]
\label{ex:finite_controls}
Let $I\in \mathbb{N}$ and $\mathcal{K}\eqdef \{u_i\}_{i=1}^I\subset \mathcal{H}_T^2$.  Then, $\mathcal{K}$ is compact.
\end{example}

\subsection{The Dynamic Stackelberg Game} 
\label{s:FormulationModel__ss:SGame}
In the previously introduced probability space $(\Omega, \mathcal{F}, \{\mathcal{F}\}_{0\leq t\leq T}, \mathbb{P})$, we consider a Stackelberg game with a leader indexed with $i=0$ and a follower indexed with $i=1$. The state process of the game is described by the stochastic differential equation 
\begin{align} 
d X_t = f(X_t, u^0_t, u^1_t) dt 
 + \sigma (X_t, u^0_t, u^1_t) d W_t , 
 \label{dX} 
 \end{align} 
%where $X^0_t \in \mathbb{R}^{d_0}$ and $X^1\in \mathbb{R}^{d_1}$ are the states of the leader and the follower,  
and $u^0_t\in \mathbb{R}^{d_0}$ and $u^1_t \in \mathbb{R}^{d_1}$ are the controls of the leader and the follower, respectively. The exact set of admissible controls will be provided below. We assume that a deterministic initial $X_0\in \R^d$ is fixed and is known to both agents. Thus, we omit the dependence of various parameters on $X_0$.  
The cost functionals of the two players are given by 
\begin{align} 
 & J_0(u^0, u^1) = \mathbb{E} \Big[ \int_0^T L_0(X_t, u^0_t , u^1_t ) dt 
 + g_0(X_T) \Big] , \label{J0} \\ 
& J_1(u^0, u^1) = \mathbb{E} \Big[ \int_0^T L_1(X_t, u^0_t , u^1_t ) dt 
 + g_1(X_T) \Big] ,  \label{J1} 
\end{align} 
where 
$L_i: \mathbb{R}^d \times \mathbb{R}^{d_0}\times \mathbb{R}^{d_1} \mapsto [0, \infty)$ and 
$g_i: \mathbb{R}^d \mapsto [0, \infty)$, 
$i=0,1$. 
We require the following regularity conditions of the involved functions.
\begin{assumption}[Regularity Conditions]
\label{assm:LipfsigLg}
There exists a constant $K>0$ such that for $h(x, u^0, u^1) = f(x, u^0, u^1)$, $\sigma(x, u^0, u^1)$,  $L_i(x, u^0, u^1)$, and $g_i(x)$, $i=0$, $1$,   
\begin{align} 
  | h (x, u^0, u^1) - h (\tilde{x}, \tilde{u}^0, \tilde{u}^1) | \leq 
  K ( |x - \tilde{x}|  + | u^0 - \tilde{u}^0| + |u^1 - \tilde{u}^1 |  ) .  
\label{LipfLg} 
\end{align} 
\end{assumption}
%%%% Controls Defined
\noindent We define 
\begin{align*} 
\uuu_i & = \Big\{ u: [0, T] \times \Omega \to \mathbb{R}^{d_i} | 
\, \text{$u(\cdot)$ is $\{\mathcal{F}\}_{t}$-adapted}, \ 
 \mathbb{E} \int_0^T |u_t|^2dt < \infty \Big\} \notag 
\end{align*}  
and fix $\kkk_0\subset \uuu_0$ so that $\kkk_0$ is the set of possible controls of the leader, $\uuu_1$ is the set of possible controls of the follower. The introduction of $\kkk_0$ is needed due to the fact that our operators in Subsection~\ref{s:prelim__ss:NNs} will only perform optimization relative to a compact subset $\kkk_0$ of $\uuu_0$.

For each $u^0\in \uuu_0$, the set of best responses for the follower is
\begin{align} 
 \mathcal{R}(u^0) = \big\{ u\in \uuu_1: \, 
 J_1(u^0, u) \leq J_1(u^0, u^1), \, \forall u^1 \in \uuu_1 \big\} . 
\label{R(u0)}  
\end{align} 
Following the definitions in \cite{bensoussan2015maximum}, we define {adapted} open-loop (AOL) responses of the follower to the controls of the leader by
\begin{align}
    \bar \uuu_1 = \Big\{ u: [0, T] \times \Omega \times  \kkk_0\to \mathbb{R}^{d_1} | 
\, \forall u_0\in \kkk_0,\, u(\cdot,u^0)\in \uuu_1\}.
\end{align}
We recall that implicitly we make the assumption that the initial point $X_0$ of $X$ is fixed throughout the paper. If this initial condition is not fixed, one has to allow the elements of $ \bar \kkk_1$ to also depend on this initial condition as it is the case in \cite{bensoussan2015maximum}.

We study the Stackelberg equilibria for the leader-follower problem; that is, a set of leader-follower strategies wherein the follower optimally responds to the preemptive optimal action of the leader in such a way that
neither player can gain utility by perturbing their strategy.  Formally, a Stackelberg equilibrium is defined as follows.
\begin{definition}[Stackelberg Equilibrium] 
\label{def:stackelberg}
A Stackelberg equilibrium (relative to $\kkk_0\subset \uuu_0$) of the leader-follower game \eqref{dX}-\eqref{J0}-\eqref{J1} is a pair $(u^{0, \star}, U^{\star})\in \kkk_0\times \bar \uuu_1 $ such that $U^{\star}(u^{0}) \in \mathcal{R}(u^{0})$ for all $u^0 \in \mathcal{U}_0$ and 
\begin{align} 
 J_0(u^{0, \star},U^{\star}(u^{0, \star})) \leq J_0(u^0,U^{\star}(u^0))  \mbox{ for all }u^0\in \kkk_0. \notag  
\end{align}
If it exists, a map $U^{\star}$ is called \textit{a} best response map of the follower.

\end{definition}

In general, $\mathcal{R}(u^0)$ can be empty for some $u^0\in \kkk_0$. If this happens, the equilibrium will not exist. However, if the problem of the follower is convex enough,  $\mathcal{R}(u^0)$ is reduced to a point $ U^{\star}(u^0)\in\uuu_1 $ that can be described by an FBSDE. Thus, the existence of a Stackelberg equilibrium is reduced to the optimization of $u^0\in \kkk_0\mapsto J_0(u^0,U^{\star}(u^0)) $ that we call the effective criterion of the leader. This can be done if $\kkk_0$ is compact or under additional structural assumption on the data as a form of control of solutions of FBSDEs. 
We will assume that the optimal response operator of the follower exists and possesses a minimal level of regularity.  
\begin{assumption}[H\"{o}lder Continuity of The Follower's Best Response]
\label{assm:continuity}
There exists a H\"{o}lder continuous mapping $u^0\in \kkk_0\mapsto U^{\star}(u^0)\in \uuu_1$ so that $U^{\star}(u^0) \in \mathcal{R}(u^0)$ for all $u^0\in \kkk_0$.
\end{assumption}
\begin{remark}
Though several of our universal approximation results (Theorem~\ref{thrm:UniversalApprox}) can be applied only while assuming continuity of the following best response, our quantitative estimates fundamentally rely on H\"{o}lder continuity since we use properties of doubling metrics, building on the method of~\cite{kratsiosuniversal}, and metric snowflakes; see~\cite[page 66]{NikWeaversFamousBook_2018} for definitions.
\end{remark}

In Section~\ref{s:ProofSketch} and~\ref{s:Exmaples_Regular Optimal Response}, below, we show this Holder dependence estimates of the optimal response if the optimization problem of the follower is strongly convex. 
In fact, we provide an example of a static game in Section \ref{s:CounterEx} showing that even if the problem of the follower is only convex but not strongly convex, the optimal response of the follower will lack continuous dependence on the control of the leader. In such cases, our neural operator cannot approximate the optimal response. Thus, Assumption \ref{assm:continuity} is not merely an often satisfied and purely technical assumption. 

\subsection{Neural Operators}
\label{s:prelim__ss:NNs}

Fix an activation function $\sigma:\mathbb{R}\times \mathbb{R}\to \mathbb{R}$ and, for any $N\in \mathbb{N}_+$,  define its componentwise composition with any vector $x\in \mathbb{R}^N$ with trainable parameter $\boldsymbol{\alpha} \in \mathbb{R}^{N}$, by $
        \sigma_{\boldsymbol{\alpha}}\bullet x 
    \eqdef 
        \big(
            \sigma_{\boldsymbol{\alpha}_i}(x_i)
        \big)_{i=1}^d
$.
For the majority of our paper, we consider neural operators either with 
an unattainable activation function satisfying the condition of~\citep{KidgerLyons} or a trainable variant $\sigma\in C(\mathbb{R}^2)$ of the super-expressive activation function of~\cite{zhang2022deep}; defined shortly. 
In the former case, we consider the following activation functions. 
\begin{example}[{\cite{KidgerLyons}-Type ``Standard'' \textit{Trainable} Functions}]
\label{ex:Activation_Standard}
There is a non-affine $\sigma_0\in C(\mathbb{R})$ such that: there exists some $t_0\in \mathbb{R}$ at which $\sigma_0$ is differentiable and such that $\sigma_0(t_0)^{\prime}\neq 0$.  Define $\sigma\in C(\mathbb{R}^2)$ by $(\alpha,t)\mapsto \sigma_{\alpha}(t)$.
% {\color{red}sure about this?}\\
Observe that the $\operatorname{ReLU}$ activation function falls into this class.
\end{example}
\begin{example}[{Super-Expressive Activation with Neuron-Specific Skip-Connection}]
\label{ex:Activation_Superexpressive}
We define the trainable variant of the activation function $\sigma:\mathbb{R}^2\to \mathbb{R}$ mapping any $(\alpha,t)\in \R^2$ to
\begin{equation}
\label{eq:supreexpressive}
    % \sigma(\alpha,t)\eqdef 
    \sigma_{\alpha}(t) \eqdef 
    \alpha t + (1-\alpha)\,
    \begin{cases}
        | t \pmod{2}| & \mbox{ if } t \ge 0\\
        \frac{t}{|t|+1} & \mbox{ if } t<0
    \end{cases}
    % \max\{\alpha_1t,t\} + \alpha_2\sin(t) + \alpha_3 2^t
    .
\end{equation}
When $\alpha=0$, then $\sigma_0$ coincides with the super-expressive activation function of~\cite{zhang2022deep}.
The parameter $\alpha$ allows us to apply a skip connection at any given specific neuron.   
\end{example}

We now define the deep learning backbone of our neural operator model, namely the multilayer perceptrons (MLP). 
Fix input and output dimensions $n$ and $m$.
An MLP with (trainable) activation function $\sigma$, \textit{depth} $J\in\mathbb{N}_+$, and \textit{width} $W\in \mathbb{N}_+$  
is a map $f:\mathbb{R}^{n}\to \mathbb{R}^{m}$ with iterative representation,
\begin{equation}
    \label{eq_definition_ffNNrepresentation_function}
    \begin{aligned}
    % Function
    \hat{f}_{\theta}(x) 
        & \eqdef
    x^{(J)}
   +c,
    \\
    %% Loop
    x^{(j+1)} &\eqdef 
    A^{(j)} \sigma_{\alpha^{(j)}}
    %_{\bar{\alpha}^{(j)}}
    \bullet(
        x^{(j)}
            +
        b^{(j)}),
    %% Init
    \\
    x^{(0)} &\eqdef  x.
    \end{aligned}
\end{equation}
where for $j=0,\dots,J-1$: $A^{(j)}\in \mathbb{R}^{d_{j+1}\times d_j}$, $\alpha^{(j)},b^{(j)}\in \mathbb{R}^{d_{j+1}}$, and $n=d_0,\dots,m=d_J\le W$.
Denote the set of MLPs mapping $\mathbb{R}^n$ to $\mathbb{R}^m$ with depth at-most $J$ and width at-most $W$ by $\mathcal{NN}_{J,W:n,m}$.  

\paragraph{The Attentional Neural Operator Model}
% \label{s:prelim__ss:NOs}

Neural operators (NOs) are natural, infinite dimensional extensions of classical (finite-dimensional) neural networks.  We follow the general encoder-processor-decoder NO paradigm considered, e.g.\ in PCA-nets \cite{lanthaler2023operator},~\cite{castro2023kolmogorov}, or~\cite{kratsiosuniversal_GDL_MaxPower}.  Our neural operators model, illustrated in Figure~\ref{fig:Attentional_NO}, maps inputs and output in the space structure $\mathcal{H}_T^2$, as opposed to standard neural operators which are maps functions on a Euclidean domain to functions in another Euclidean domain.

\begin{figure}[htp!]%[H]
    \centering
    \includegraphics[width=\linewidth]{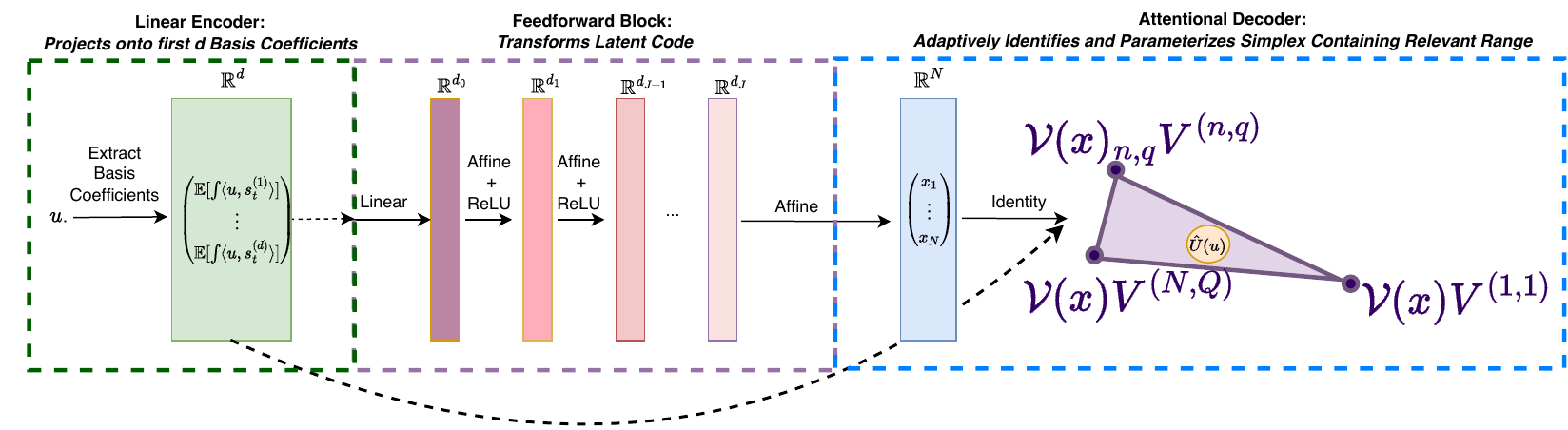}
    \caption{\textbf{Attentional Neural Operator Workflow:} 
    Our \textit{attentional} \textbf{n}eural \textbf{o}perator model maps controls $u_{\cdot}$ to square-integrable $\mathbb{F}$-adapted processes $\hat{U}(u_{\cdot})$ in three phases.  First, the (input) control is linearly projected onto the wavelet-like (in time) Wiener Chaos-like (in space) orthonormal basis of $\mathcal{H}_T^2$.  Next, the basis coefficients are transformed by a feedforward neural network (MLP).  Lastly, the basis coefficients are used to identify extremal points in a simplex in $\mathcal{H}_T^2$ and the outputs of the MLP are used to parameterize a prediction in its relative interior.}
    \label{fig:Attentional_NO}
\end{figure}

The neural operators in Figure~\ref{fig:Attentional_NO} can be formalized as follows.

\begin{definition}[Attentional Neural Operator]
\label{defn:NO}
Fix a trainable activation function $\sigma_{\cdot}\in C(\mathbb{R}^2)$, an encoding dimension $d\in \mathbb{N}_+$, a queries dimension $Q\in \mathbb{N}_+$, a values dimension $N\in \mathbb{N}_+$, depth and width parameters $J,W\in \mathbb{N}_+$, and an orthonormal basis $\mathcal{S}$ of $\mathcal{H}_T^2$.
\\
The set $\mathcal{NO}_{N,Q,d,J,W:\mathcal{S}}^{\sigma}$ consists of all (non-linear) operators $U:\mathcal{H}_T^2\to \mathcal{H}_T^2$ admitting the following representation: for each $u_{\cdot}\in \mathcal{H}_T^2$
\begin{align}
\nonumber
        U (u_{\cdot})
    \eqdef 
        \mathcal{D}\big(f\circ \mathcal{E}(u_{\cdot}),\mathcal{E}(u_{\cdot}) \big)
\\
\tag{encoder}
\label{eq:encoder}
        \mathcal{E}(u_{\cdot})
    \eqdef 
        \Biggl(
            \mathbb{E}\biggl[
                \int_0^T\,
                    \langle u_{t},s^{(i)}_t \rangle
                dt
            \biggr]
        \Biggr)_{i=1}^d
\\
\tag{$\mathcal{H}_T^2$-attentional decoder}
\label{eq:decoder}
        \mathcal{D}(w,x)
    \eqdef 
        \sum_{n=1}^N\, \operatorname{softmax}(w)_n \, \sum_{q=1}^Q\, 
        % K({u_{\cdot}})_{n,q}
        \mathcal{V}_{n,q}(x)
        \, V^{(n,q)}
,
\end{align}
where $\{s^{(i)}\}_{i=1}^{d},\{V^{(n,q)}\}_{n,q=1}^{N,Q}\subset \mathcal{S}$, $f\in \mathcal{NN}_{J,W:d,N}^{\sigma}$, $\mathcal{V}\in \mathcal{NN}_{J,W:d,N\times Q}$.
\hfill\\
The number of (non-zero trainable) parameters defining the neural operators $U$ is at most 
\[
    \underbrace{
        JW^2 
    }_{\text{MLP }(f)}
    + 
    \underbrace{
        NQ
    .
    }_{\text{MLP } (\mathcal{D})}
% .
\]
The map $\mathcal{E}$ is called an encoder, $\mathcal{D}$ is called an attention-based decoder, $(\mathcal{V},\{V^{(n,q)}\}_{n,q=1}^{N,Q})$ are called values, and $U$ is called a attentional neural operator.
\end{definition}
We henceforth take $\mathcal{S}$ to be~\eqref{eq:simple_processes} and denote its elements $\{s^{(i)}\}_{i=1}^{\infty}$.  Thus, we write $\mathcal{NO}_{N,Q,d,J,W}^{\sigma}$ in place of $\mathcal{NO}_{N,Q,d,J,W:\mathcal{S}}^{\sigma}$.

\paragraph{Link Between Our Decoder, Attention, and Transformers}
Before moving on, we discuss the relationship between the decoder of our neural operator and the attention layer in \cite{bahdanau2014neural} in the standard transformer networks of \cite{vaswani2017attention}.  For simplicity, we focus on a single attention head, not the multi-head attention mechanism.

The standard attention mechanism $\operatorname{attention}$ maps $N$, $d$-dimensional vectors $x_1,\dots,x_N$, seen as a $N\times d$ matrix $x=(x_n)_{n=1}^d$ to another $N\times \tilde{d}$ matrix $\operatorname{attention}(x)$ for some $\tilde{d}\in \mathbb{N}$.  
This attention mechanism operates in two phases, first, it extracts \textit{contextual weights} $w^x$ in the $N$-simplex via
\begin{equation}
\label{eq:attention_weights}
        w^x
    \eqdef 
        \operatorname{softmax}\Big(
                \big(
                    \langle
                        x_nQ, x_jK
                    \rangle
                    /\sqrt{d}
                \big)_{j = 1}^{N}
            \Big)_{n=1}^N
\end{equation}
where the key and query matrices $K$ and $Q$ are $\tilde{d}\times d$ matrices.  
Using these weights, one defines a set of $N$ \textit{values} $v_1,\dots,v_N$, depending on the given input $x$, by 
\begin{equation}
\label{eq:attention_values}
    v_n^x\eqdef Vx_n
\end{equation}
for $n=1,\dots,N$.
The attention mechanism then uses the weights $w^x$ to tweak the contextual importance of each value $v_1^x,\dots,v_N^x$ when generating a weighted prediction by
\begin{equation}
\label{eq:attention_definition}
        \operatorname{attention}(x)
    \eqdef 
        \sum_{n=1}^N\,
            w^x_n
            v^x_n
.
\end{equation}
In this way, the standard attention mechanism parameterizes the interior of the convex hull of the values $v_1^x,\dots,v_N^x$ using the softmax weights of any input.

If we streamline the contextual weight extraction step in~\eqref{eq:attention_weights}, by simply an input weight $w\in \mathbb{R}^N$ which is mapped to contextual importance weights $w$ by only using the softmax function itself.  One can relax the dependence on the contextual values $v_1^x,\dots,v_N^x$ in~\eqref{eq:attention_values} to be vectors in the target space, i.e.\ $\mathbb{R}^{N\times \tilde{d}}$ for classical transformers and $\mathcal{H}_T^2$ for our neural operator, depending non-linearly on the given input.  Here, we parameterize this non-linear dependence using a values neural network $\mathcal{V}$ depending on the encoding instead of a values matrix $V$; doing so, one arrives at the following construction, which is precisely our \textit{$\mathcal{H}_T^2$-attention}
\begin{equation*}
        \mathcal{D}(w)
    \eqdef 
        \sum_{n=1}^N\, 
            \underbrace{
                \operatorname{softmax}(w)_n 
            }_{\text{Contextual Weights}~\eqref{eq:attention_weights}}
        \, 
        \underbrace{
            \sum_{q=1}^Q\, 
            \mathcal{V}}_{n,q}(u_{\cdot})\, 
            V^{(n,q)
        }_{\text{Contextual Values}~\eqref{eq:attention_values}}
.
\end{equation*}
Thus, one can interpret our $\mathcal{H}_T^2$-attention as an infinite-dimensional analogue of the standard attention mechanism of \cite{bahdanau2014neural}.  Note that, since $\mathcal{D}$ receives inputs from an MLP and since MLPs can approximately implement continuous functions then there is no need to rely on a vector of inner-products such as $(\langle Q x_n, Kx_j/\sqrt{d})_{j=1}^N$, in~\eqref{eq:attention_weights}, since that can be approximately implemented by the MLP in principle; by the universal approximation theorem.

Instead, we use the inner products to obtain low-dimensional approximate representations of any given input (control) in our encoding layer~\eqref{eq:encoder}.  
Thus, one can view our neural operator as an infinite-dimensional take on the transformer network model.

\section{Main Results}
\label{s:Main_Results}

Our first result shows that that neural operators are rich enough to approximately minimize the response functional of the follower, to arbitrary precision on any given compact set of actions which the follower can take.

\begin{theorem}[$\varepsilon$-Optimal Response Operators]
\label{thrm:Main__BestResponse}
Under Assumptions~\ref{assm:LipfsigLg} and \ref{assm:continuity}, 
%{\color{red}this is not true we have to need Assimptoin \ref{assm:continuity}}, 
for each compact $\kkk_0\subseteq \uuu_0$, and each $\varepsilon>0$ there is an encoding dimension $d\in \mathbb{N}_+$ and a neural operator $\hat{U}\in \mathcal{NO}:\uuu_0\to\uuu_1$ satisfying 
\begin{equation}
\label{eq:thrm:Main__ResponseEquilibria}
    \sup_{u^0\in \kkk_0}\,
        \|
                U^{\star}(u^0)
            -
                \hat{U}(u^0)
        \|_{\mathcal{H}_T^2}
    \le 
        \varepsilon
    .
\end{equation}
%{\color{red}since there is a corrolary below all this is to be removed.}
%Furthermore, there is a $u^0=\sum_{i=1}^{\infty}\, \beta_i\,u_i\in K$ such that the pair $(\hat{u}_d^0,\hat{U}(\hat{u}_d^0))$, % where, $\hat{u}_d^0\eqdef \sum_{i=1}^d\,\beta_i\,u_i$ is an $\varepsilon$-Stackelberg equilibrium.
%Additionally, denoting $u^\star$ an $\e$-optimizer of the continuous function $u_0\in K\mapsto J_0\big(u_0,\hat{U}(u_0)\big) $, %the couple $(u^\star,\hat U)$ is an $\e$-Stackelberg equilibrium. 
\end{theorem}

\subsection{An Unsupervised Objective Function - Bypassing Computing \texorpdfstring{$U^{\star}$}{U}}
Theorem~\ref{thrm:Main__BestResponse}
alone does not guarantee that approximations produce Stackelberg equilibria.  Our second main result guarantees that approximately playing any such Stackelberg games is enough to yield approximate Stackelberg equilibria. 

Additionally, theorem~\ref{thrm:Main__BestResponse}
% and Corollary~\ref{cor:Main__Equilibria} 
guarantees that the solution operator to the Stackelberg game can be approximately implemented on compact sets, to arbitrary precision.  
However, it does not describe how one could train such a network in practice.  Indeed it seems most natural to minimize the leader's loss $J_0$ for any given $u^0$ in the relevant compact set $\kkk_0$.  However, $u^0$ may not be \textit{exactly implementable} in practice, instead one could consider minimizing $J_0$ where $u^0$ is replaced by a finite-dimensional (e.g.\ linear) approximation $\hat{u}^0\eqdef p_d(u^0)$ where $p_d:\mathcal{H}_T^2\to \operatorname{span}\{s_i\}_{i=1}^d$ is the orthogonal projection; i.e.\ $p_d\big(
\sum_{i=1}^{\infty}\,\beta_is_i \big)= \sum_{i=1}^{d}\,\beta_is_i$ for all $u=\sum_{i=1}^{\infty}\,\beta_i\,s_i\in \mathcal{H}_T^2$.  
Thus, a natural and tractable objective would be to minimize
\begin{equation}
\label{eq:loss_to_minimize__U}
    {
    % \hat{U}\in 
    \underset{U}{
    \operatorname{min}
    % \operatorname{argmin}
    }
    }
    \min_{\hat{u}^0_d \in p_d(\kkk_0)}
        J_{{0}}\big(
            \hat{u}^0_d
        ,
            U(
            \hat{u}^0_d
            )
        \big)
\end{equation}
when training the attentional neural operator
{where $\hat{U}$ is minimized over a compact class of neural operators.  Once a minimizer $\hat{U}$ of~\eqref{eq:loss_to_minimize__U} is identified, it can then be used to approximate the optimal action of the leader by minimizing the following objective 
\begin{equation}
\label{eq:loss_to_minimize}
    \min_{\hat{u}^0_d \in p_d(\kkk_0)}
        J_{{0}}\big(
            \hat{u}^0_d
        ,
            \hat{U}(
            \hat{u}^0_d
            )
        \big)
\end{equation}
}.

Our following result suggests that~\eqref{eq:loss_to_minimize} can be used, {after training $\hat{U}$}, as an \textit{unsupervised} objective function, which is small only when $\hat{U}$ is has correctly approximated to optimal response $U^{\star}$.  Unlike the supersized criterion in~\eqref{eq:thrm:Main__ResponseEquilibria}, which requires us to knowing pairs of $u^0$ and best responses $U^{\star}(u^0)$,~\eqref{eq:thrm:Main__ResponseEquilibria} can be minimized without having to first compute $U^{\star}$.  Moreover, $\hat{U}$ is only optimized on a finite-dimensional subspace of our space of controls.

\begin{theorem}[The Unsupervised Objective Function~\eqref{eq:loss_to_minimize} Detects Optimality]
\label{thrm:Objective}
The following hold in the setting of Theorem~\ref{thrm:Main__BestResponse}.
\begin{enumerate}
    \item[(i)] For every $\delta>0$ there exist $\epsilon>0$, and $d\in \mathbb{N}_+$ such that if $\hat{U}$ satisfies~\eqref{eq:thrm:Main__ResponseEquilibria}, then the following hold
\begin{equation}
\label{eq:thrm:Objective}
    \sup_{u^0\in \kkk_0}
    \,
        \big|
                J_{{0}}(u^0,U^{\star}(u^0))
            -
                J_{{0}}\big(
                    \hat{u}^0_d
                ,
                    \hat{U}(
                    \hat{u}^0_d
                    )
                \big)
        \big|
    <
        \delta
.
\end{equation}

    \item[(ii)] Moreover, there is a $\hat{u}^0_d = \sum_{i=1}^d \, \beta_i \, s_i\in \kkk_0$ such that the pair $(\hat{u}_d^0,\hat{U}(\hat{u}_d^0))$ is an $\varepsilon$-Stackelberg equilibrium
    in the sense that the pair $(\hat{u}^0_d , \hat U)\in \kkk_0\times \bar \uuu_1 $ satisfies for all $(u^0,u^1)\in\kkk_0\times \uuu_1$ the inequalities
\begin{align*}
    J_1(u^0,\hat U(u^0))&\leq J_1(u^0,u^1)+\epsilon\\
    J_0(\hat{u}^0_d,\hat U(\hat{u}^0_d))&\leq  J_0(u_0,\hat U(u_0))+\epsilon.
\end{align*}
\end{enumerate}
% (iii). Denoting $u^\star$ an $\e$-optimizer of the continuous function $u_0\in \kkk_0\mapsto J_0\big(u_0,\hat{U}(u_0)\big) $, the couple $(u^\star,\hat{U}(u^{\star}))$ is an $\e$-Stackelberg equilibrium. 

\end{theorem}

As with classical uniform approximation results in deep learning, we worked on a compact in the space of square-integrable $\mathbb{F}$-predictable processes.  Though this is standard in the approximation theory literature, it is not so in game theory.  
%{\color{red}no next result. We can merge the theorem above with the corrolarry} The next result 
% {\color{red}
However, reducing the problem to compact $\kkk_0$ is asymptotically coherent in the sense that if we take compact sets $\kkk_n\subset \uuu_0$ so that $\cup_n \kkk_n=\uuu_0$. Then, for any $\delta_n\downarrow0$ and $( \hat{u}^n_d,\hat{U}^n)$ satisfying \eqref{eq:thrm:Objective} with $\delta_n$, 
we have that $$ \inf_{u^0\in \uuu_0}J_{{0}}(u^0,U^{\star}(u^0))=\lim_{n\to \infty}J_{{0}}\big(
                    \hat{u}^n_d
                ,
                    \hat{U}^n(
                    \hat{u}^n_d
                    )).$$
Thus, by taking the compact set $\kkk_n$ and the NO larger, we approximate the optimum of the effective value of the leader $\inf_{u^0\in \uuu_0}J_{{0}}(u^0,U^{\star}(u^0))$. Note that if $\uuu_0$ is not compact or if the problem of the leader does not have additional properties such as convexity {or coercivity of $u^0\mapsto J_{{0}}(u^0,U^{\star}(u^0))$, see e.g.~\cite[Theorem 7.12]{dal2012introduction}}, one cannot easily claim the existence of the optimizer of $\inf_{u^0\in \uuu_0}J_{{0}}(u^0,U^{\star}(u^0))$ or the convergence of the family $\big(
                    \hat{u}^n_d
                ,
                    \hat{U}^n).$                  
However, these additional structural assumptions are not needed for the purposes of computationally approximating the optimal value of the leader. 
% 
% remove the next paragraph
% }
% \Xuwei{removed.}

%Theorem~\ref{thrm:Main__BestResponse___goodrates} confirms that there is no loss of generality in doing so as if asymptotically %expands the ``size'' of the compact set until the entire space is recovered then sequences of approximate Stackelberg %equilibria on the compact converge to the \textit{true Stackelberg equilibrium} over the entire space of square-integrable %$\mathbb{F}$-predictable processes.  Thus, both the conditions in the deep learning and game theory literature are %asymptotically coherent.

More insight can be gained into the inner workings of these results by inspecting the steps in their derivations, at least at a high level.  The next section shows the overarching structure of an argument which one can use to derive any such result.

\subsection{Rates For Perturbations of Closed-Loop Solutions to Linearized Game}
\label{s:EfficientRates}

This section shows that, contrary to the general approximation rates for neural operators (see e.g.~\cite{lanthaler2022error} or~\cite{galimberti2022designing}), which may be highly inefficient due to the infinite-dimensional nature of the involved spaces, there are stylized conditions which allow for efficient approximation rates of the optimal response map by our attentional neural operator. 

Building on the strategy of \cite{shi2023deep}, 
we focus on the case where the leader begins by considering a proxy/ansatz/pre-trained  control, $\bar{u}
$ 
Suppose that $\bar{u}_0 \in \mathcal{H}_T^2$ is a control for the \textit{leader}.  Suppose that the leader is comfortable deviating from their strategy to an open-loop control, but only marginally through ``small residual perturbations''.  
If those deviations are not too large, efficient approximation rates are possible.  
More precisely, consider the following situation.
\begin{assumption}[Perturbations of an Ansatz]
\label{ass:efficient_rates}
Let $C\ge 0$, $r> 0$, $\bar{u}\in \mathcal{H}_T^2$, and suppose that Assumption~\ref{assm:continuity} holds.  
% \\
Let $\kkk_0$ consist of all $u\in \mathcal{H}_T^2$ for which: 
\begin{enumerate}
    \item[(i)] $|\langle u-\bar{u},s_i\rangle | \le C \, e^{-ri}$,
    \item[(ii)] $|\langle U^{\star}(u-\bar{u}),s_i\rangle | \le C \, e^{-ri}$.
\end{enumerate}
% {\color{red}there is an inconsistency with \eqref{eq:open_loop_perturbations}}
\end{assumption}

\begin{remark}
\label{rem:non-vaccousness}
Note that, for every $C,r\ge 0$, $\bar{u}\in \mathcal{H}_T^2$ the set $\kkk_0$ is non-empty since $\bar{u}\in \kkk_0$.
\end{remark}
In both of the following examples, we consider the sets of ``residual controls''.
Let $C,r>0$ and define $B_{C,r}(u)$ to be the ``exponentially ellipsoidal'' set of \textit{open-loop} ``residual controls'' relative to a control $\bar{u} \in \mathcal{H}_T^2$ to consist of all {$u\in \mathcal{H}_T^2$}
% {\color{red} sure but in one equation you evaluate Ustar at v and in \eqref{eq:open_loop_perturbations} at u. THere has to be a problem}
% \AK{non the set consists of all $u$, but I think my }
such that
$v \eqdef u-\bar{u}
=
\sum_{i=1}^{\infty}\, \beta_i\, s_i\in \mathcal{H}_T^2$ satisfies
\begin{equation}
\label{eq:open_loop_perturbations}
\begin{aligned}
        \max\{|\langle v,s_i\rangle |,|\langle U^*(
        % \bar{u}_{\cdot}+
        {v}
        ),s_i\rangle |\}
        % &
        \le C\, e^{-ri}
        \,\,(\forall i \in \mathbb{N}_+)
.
\end{aligned}
\end{equation}

Illustrated by Figure~\ref{fig:perturbations__lin_prob_control}, our primary example of a compact set satisfying the perturbation of the Ansatz assumption above is given by first solving a linear-quadratic proxy of the general (non-linear) Stackelberg game.  Then, once an optimal \textit{feedback} control for the leader is derived, for the linear-quadratic proxy of our (non-linear) Stackelberg game, we build $\kkk_0$ by adding residual open-loop controls in $B_{C,r}$. 
Intuitively, these residual loops provide added freedom to the closed-loop control optimizing the linear-quadratic proxy required when playing the Stackelberg game.

\begin{figure}[H]
\begin{minipage}{0.35\textwidth}
%%%
\centering
\includegraphics[width=1\linewidth]{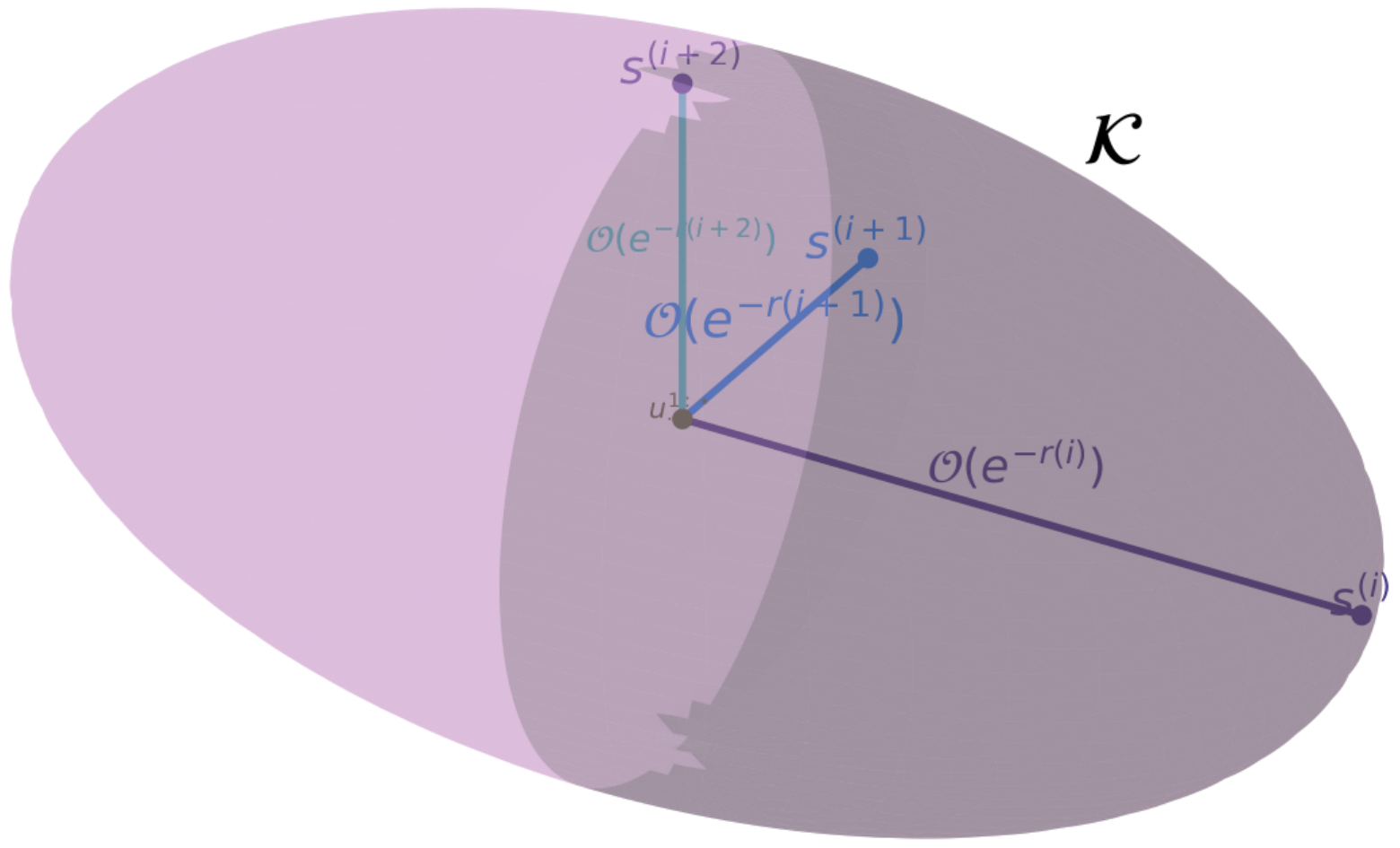}
\end{minipage} \hfill
\begin{minipage}{0.65\textwidth}
\textbf{Explanation of Figure:}
The set $\kkk_0$ in Example~\ref{ex:perturbations__lin_prob_control} is a (compact) ellipsoidal region in $\mathcal{H}_T^2$ consisting of open-loop controls $u$ which are a small perturbation of a base strategy/control $u^{1:\star}$.  The base strategy $u^{1:\star}$ is the solution to a linearization, see~\eqref{eq:linearized_Dynamics}-\eqref{eq:linearized_Cost}, of the dynamic Stackelberg game. The perturbations $u$ of $u^{1:\star}$ in $\kkk_0$ are built by adding a small residual term in each of the basic directions $s^{(i)}\in \mathcal{S}$ where we allow for possibly large perturbations of the linearized strategy $u^{1:\star}$ for the ``low-frequency directions'' (i.e.\ for small values of $i$) and much smaller perturbations of the linearized strategy for ``high-frequency directions'' (i.e.\ for small $i$).  The key subtlety is that the size of the perturbations must decay exponentially in $i$.
\end{minipage}
\caption{The ellipsoidal compact set $K$ of Example~\ref{ex:perturbations__lin_prob_control}.}
\label{fig:perturbations__lin_prob_control}
\end{figure}

%%%

%%%
\begin{example}[Perturbations of Feedback Control For Linearized Problem - Pt.\ I]
\label{ex:perturbations__lin_prob_control}
Consider a finite subset $\{(x_n,v_n^0,v_n^1)\}_{n=1}^N\subset \mathbb{R}^d$ and let $A,B_1,B_2,C,D_1,D_2$ be matrices minimizing the following MSE problem over all matrices of compatible dimension
\[
\sum_{n=1}^N\, 
    \|
        f(x_n,v_n^0,v^1_n)
        -
        \underbrace{
            (Ax_n + B_1v_n^0 + B_2v^1)
        }_{\text{lin. approx. drift}}
    \|^2
    +
    \|
        \sigma(x_n,v_n^0,v_n^1)
        -
        \underbrace{
            Cx_n + D_1v_n^0 + D_2v_n^1
        }_{\text{lin. approx. diff.}}
    \|^2
.
\]
Consider the (controlled) ``linearized'' state-space process $X_{\cdot}^{lin}\eqdef (X_t^{lin})_{t\ge 0}$ given by
\begin{equation}
\label{eq:linearized_Dynamics}
      dX_t^{lin}
    =
        (AX_t^{lin} + B_1u_t^0 + B_2u_t^1)dt
    +
        (CX_t^{lin} + D_1u_t^0 + D_2u_t^1)dW_t
\end{equation}
Further, assume that $Q_1,Q_2,R_1,R_2,G_1,G_2$ are 
matrices minimizing the following MSE problem over all matrices of compatible dimension
\[
\sum_{n=1}^N\, 
\sum_{i=1}^2\,
    \|
        L_i(x_n,v^0_n,v^1_n)
        -
        \underbrace{
        \big(
            (Q_ix_n)^{\top}x_n
            +
            (R_iv_n^i)^{\top}v_n^i
        \big)
        }_{\text{lin. approx. running cost}}
    \|^2
    +
    \|
        \underbrace{
            g_i(x_n)
        -
            (G_ix_n)^{\top}x_n
        }_{\text{lin. approx. terminal}}
    \|^2
.
\]
Under \citep[Assumptions (DI) and (H1)1035]{yong2002leader} the computations on~\citep[ages 1034 and 1035]{yong2002leader} together with~\citep[Proposition 2.2 and Theorem 2.]{yong2002leader}, they imply that there is an optimal control $u^{0:\star}$ for the leader minimizing the approximate objective function
\begin{equation}
\label{eq:linearized_Cost}
    \mathbb{E}\biggl[
        \int_0^T
        \,
            L_0(X_t,u_t^0,u_t^1)
        \,dt 
        + 
            (G_iX_T)^{\top}X_T
    \biggr]
\end{equation}
across all controls in $\mathcal{H}_T^2$.  Moreover, as shown in~\citep[Equastion (5.12)]{yong2002leader}, $u^{1:\star}$ is given by
\[
        u_t^{1:\star}
    = 
        -R_1^{-1}B_1^{\top}\,
            p_t
    ,
\]
where $(p_t,q_t)_{t\ge 0}$ solve the FBSDE
\[
\begin{aligned}
dp_t & = p
(A^{\top}\,p_ +Q_0X_t)
dt
+ 
q_t\,dW_t
\\
p_T & = G_0X_T.
\end{aligned}
\]
Let $\kkk_0$ be the set of controls for the leader $u\in \mathcal{H}_T^2$ for which there exists some open-loop ``residual perturbations'' $v \in B_{C,r}(u^{1:\star})$ such that
\[
    u = u^{1:\star} + v
.
\]
By construction, $\kkk_0$ satisfies Assumption~\ref{ass:efficient_rates}.
\end{example}

The set of open-loop perturbations of the closed-loop optimal control for the linearized game, just described in Example~\ref{ex:perturbations__lin_prob_control}, can be efficiently represented using few dimensions, and likewise for its image under the optimal response map $U^{\star}$.  
Nevertheless, the set need not be coverable by few controls; that is, it may still have high metric entropy (see e.g.~\cite{LorentzEntropApprox} for an expos\'{e}e).  We, therefore, adopt a technique used in statistical learning/empirical process theory to construct classes of VC-dimension proportional to the metric entropy of high-dimensional balls (see e.g.~\cite[Theorem 2.7.11]{varderVaartWellner_EmpiricalProcessesBook_1996}).  Namely, we postulate the existence of a latent unknown Lipschitz parameterization of a latent low-dimensional ``manifold'' of the set of open-loop controls $\kkk_0$ constructed in Example~\ref{ex:perturbations__lin_prob_control}.

\begin{example}[Perturbations of Feedback Control For Linearized Problem - Pt.\ II]
\label{ex:perturbations__lin_prob_control__Pt2}
Fix a latent dimension $d\in \mathbb{N}_+$, a latent parameter space $B_d\eqdef \{x\in \mathbb{R}^d:\,\|x\|\le 1\}$, and a latent parameterization given by a $1$-Lipschitz map $\pi:\mathbb{R}^d\to \kkk_0$.
with the property that $\pi(0)=u^{1:\star}$; i.e.\ parameterizing a latent low-dimensional structure which perturbs the optimal control for the linear-quadratic approximation of the general Stackelberg games.  
Let $K_{d,\pi}\eqdef \pi(B_d)$.
By construction $K_{d,\pi}\subseteq \kkk_0$; thus, $K_{d,\pi}$ also satisfies Assumption~\ref{ass:efficient_rates}.
Finally, by~\citep[Theorem 2.7.11]{varderVaartWellner_EmpiricalProcessesBook_1996}\footnote{We have used the upper-bound of the $\varepsilon$-bracketing numbers on the $\varepsilon/2$-covering numbers (see e.g.~\cite[page 84]{varderVaartWellner_EmpiricalProcessesBook_1996})} and \cite[Proposition 15.1.3]{lorentz1996constructive}, we have that: for each $\varepsilon>0$ there exists at-most $N\le 3^d\,\big(\sqrt{d}2/\varepsilon \big)^d$ controls $\{u^{(n)}\}_{n=1}^N$ forming an $\varepsilon$-cover of $K_{d,\pi}$; that is, 
\begin{equation}
\label{eq:metric_entropy_bound}
        \underset{u_{\cdot}\in K_{d,\pi}}{\sup}
        \,
        \underset{n=1,\dots,N}{\min}
        \,
            \mathbb{E}\biggl[
                \int_0^T\, 
                    \|
                        v_t - u_t
                    \|^2
                dt
            \biggr]
        \le 
            \varepsilon
.
\end{equation}
In other words, if the latent dimension $d$ is ``small'', then $K_{d,\pi}$ is small in metric entropy.  In turn, the parameter $N$ in our transformer can be taken to also be proportionally small (see Theorem~\ref{thrm:Main__BestResponse___goodrates}).
\end{example}

When the conditions of 
Assumption~\ref{ass:efficient_rates} are met, we are able to guarantee that the approximating attention neural operator $\hat{U}$ to the best-response map $U^{\star}$, given by Theorem~\ref{thrm:Main__BestResponse}, is determined by relatively few parameters.  This is the content of our last main result.

\begin{theorem}[Efficient Approximation of the Best-Response Map]
\label{thrm:Main__BestResponse___goodrates}
Consider the setting of Theorem~\ref{thrm:Main__BestResponse} and additionally suppose that $\kkk_0$ satisfies Assumption~\ref{ass:efficient_rates}.
For every $\varepsilon>0$, the conclusion of Theorem~\ref{thrm:Main__BestResponse} holds and there exists a neural operator $\hat{U}\in \mathcal{NO}:\uuu_0\to\uuu_1$ satisfying the uniform estimate in~\eqref{eq:thrm:Main__ResponseEquilibria} whose depth, width, decoding dimension, encoding dimension, and attentional complexity are bound above by the estimates in Table~\ref{tab:Complexity_Estimates_for__perturbations__lin_prob_control__Pt2}.
\end{theorem}
\begin{table}[H]
    \centering
    \begin{adjustbox}{width=\textwidth}
    \begin{tabular}{ccccc}
    \toprule
         \textbf{Depth} & \textbf{Width} & \textbf{Decoding Dim.} $(Q)$ & \textbf{Encoding Dim.} & \textbf{Att.\ Complexity} ($N$)
         \\
        \midrule
         $
            \mathcal{O}\Big(
                \ln(\varepsilon
                % _D
                ^{-1/r})
                \,
                \varepsilon
                % _D
                ^{-\ln(C)/r}
            \Big)
         $
         & 
         $
            \mathcal{O}\big(
                \varepsilon
                % _D
                ^{
                % 1
                -\ln(C)/r}
            \big)
         $
         &
         $
         \mathcal{O}(\ln(\varepsilon
         % _D
         ^{-1/r}))
         $
         &
         $
         \mathcal{O}(\varepsilon
         % _D
         ^{1/(1-r)})
         $
         &
         $
         \big(\tilde{c}\,\varepsilon
         % _A
         ^{-1}\,       
            \ln(\varepsilon
            % _D
            ^{-1/r})^{1/2}   
        \big)^{c(\ln(\varepsilon
        % _D
        ^{-1/r})}
         $
    \\
    \bottomrule
    \end{tabular}
    \end{adjustbox}
    \caption{Parametric Complexity of Attentional Neural Operator Approximation of the Best Response Map $U^{\star}$ over the compact set $\kkk_0$ of Assumption~\ref{ass:efficient_rates}, with trainable activation function $\sigma$ of~\eqref{ex:Activation_Superexpressive}.
    \hfill\\
    Here, $C,r>0$ are the constant defined in Assumption~\ref{ass:efficient_rates} and $c>0$ is an absolute constant.
    }
    \label{tab:Complexity_Estimates_for__perturbations__lin_prob_control__Pt2}
\end{table}

\section{Overview of Proof for \texorpdfstring{Theorems~\ref{thrm:Main__BestResponse} and~\ref{thrm:Objective}}{Main Results}}
\label{s:ProofSketch}

The derivation of Theorem~\ref{thrm:Main__BestResponse} is undertaken in two steps.  First, one must show that, under mild conditions, the best response operator set $\mathcal{R}$ is single-valued and there is a continuous selection therein, i.e., the best response operator $U^{\star}$ is a well-defined continuous non-linear operator.
This is key since our neural operator architectures are continuous, and classes of continuous functions cannot uniformly approximate discontinuous functions by the Uniform Limit theorem, see e.g.~\cite[Theorem 21.6]{MunkresTopBook}.
% {\color{red}problem with the last sentence}

Since we now know that the best response operator is well-defined and continuous, we need to show that our attentional neural operator class has the power to approximate it or, more generally, functions of the same regularity.  This is guaranteed by the following Universal Approximation theorem, guaranteeing that our class of neural operators can approximate, uniformly on compacta, any continuous non-linear operator with square-integrable $\mathbb{F}_{\cdot}$-adapted processes as inputs and outputs.

\begin{theorem}[Universal Approximation for Operators Between $\mathbb{F}$-Adapted $2$-Integrable Processes]
\label{thrm:UniversalApprox}
For every (non-empty) compact subset $\kkk_0\subset \mathcal{H}_T^2$, each continuous ``target'' function $f:\kkk_0\to\mathcal{H}_T^2$, and every ``approximation error'' $\varepsilon>0$ there exists an attentional neural operator $\hat{F}:\calH_T^2\to\calH_T^2$ satisfying
\begin{equation*}
        \max_{u\in \kkk_0}\, 
            \|f(u)-\hat{F}(u)\|_{\mathcal{H}_T^2} 
    \leq
        \varepsilon.
\end{equation*}
\end{theorem} 

Lemma~\ref{lm:stabilityJ0} shows that $U^\star$ is $1/2$-H\"{o}lder continuous. Together, Lemma~\ref{lm:stabilityJ0} and Theorem~\ref{thrm:UniversalApprox} imply that the best response map can be approximated.  The main step in using these results to deduce Theorem~\ref{thrm:Objective} lies in the continuous dependence of the leader's response on the best response. 
% Setting $\delta$, in Theorem~\ref{thrm:Objective}, to be proportional to $\tilde{\omega}(\varepsilon)$ allows us to deduce the result from Lemma~\ref{lm:stabilityJ0}.

\begin{comment} 
\Xuwei{The lemma here was moved to Lemma~\ref{lm:stabilityJ0} }
\begin{lemma}[H\"{o}lder-Like Regularity of the Leader's Utility on Optimal Response]
\label{lm:stabilityJ0}
Under Assumptions~\ref{assm:LipfsigLg} and~\ref{assm:continuity}, there exists $C,\alpha>0$ depending only on $K,T$ and the Holder norm in Assumption \ref{assm:continuity} so that for each $u^0_{\cdot},\tilde{u}^0 \in \mathcal{U}_0\cap \mathcal{H}^2_T$,  
we have 
\[
    \big| J_0(u^0, U^{\star}(u^0)) - J_0(\widetilde{u}^0, U^{\star}(\widetilde{u}^0) ) \big|  
\leq   
    \tilde{\omega}\big(
        \|u^0-\widetilde{u}^0\|_{\mathcal{H}_T^2}
    \big)
\]
where $\tilde{\omega}(t) \eqdef C\max\{|t|^{\alpha},|t|\}$ for each $t\ge 0$.
\end{lemma} 
\end{comment}

\section{Examples of Stackelberg Games Satisfying \texorpdfstring{Assumption~\ref{assm:continuity}}{The Basic Regularity of the Follower's Response Map}}
\label{s:Exmaples_Regular Optimal Response}

An easily verifiable sufficient condition guarantees that Assumption \ref{assm:continuity} holds.  The condition requires that the Hamiltonian of the follower satisfies a basic level of strong convexity; where the Hamiltonian of the follower is given by
\begin{align*}
        H_1(x,u^0,u^1,p_1,q_1) 
    \eqdef 
        p_1^\top f(x,u^0,u^1)+Tr(q_1^\top \sigma (x,u^0,u^1))+L_1(x,u^0,u^1)
.
\end{align*}
\begin{proposition}[Hamiltonian Strong Convexity Implies H\"{o}lder Continuity]
\label{prop:Sufficient_ContinuityBeforeDiscretization}
Assume Assumption \ref{assm:LipfsigLg} and that there exists $\kappa>0$ such that the functions
    \begin{align*}
        (x,u^1)\in \R^d&\mapsto H_1(x,u^0,u^1,p_1,q_1)-\frac{\kappa}{2}|u^1|^2\\
        x\in \R^d&\mapsto g_1(x)
    \end{align*}
    are convex for all values of other variables. 
    Assume also that $\nabla_x L_1$ and $\nabla_{x} g_1$ are Lipschitz continuous in their variables. 
    Then, 
    $ \mathcal{R}(u^0)$ is single valued and its unique element 
    $\{U^{\star}_t(u^0)\}=\mathcal{R}(u^0)$ 
    have the following continuous dependence
   \begin{align}\label{eq:contd}
       &\frac{\kappa}{2}\E\Big[\int_0^T |U^{\star}_t(u^0)-U^{\star}_t(\tilde u^0)|^2dt\Big]\notag\\
       &\leq J_1(\tilde u^0,U^{\star}_t(\tilde u^0))-J_1( u^0,U^{\star}_t(u^0))+J_1( u^0, U^{\star}_t(\tilde u^0))-J_1( \tilde u^0,U^{\star}_t(\tilde u^0)) 
   \end{align}
    and the Assumption \ref{assm:continuity} holds. 
    
\end{proposition}

\begin{remark}[Deriving Alternative Sufficient Conditions]
   1) Another case where one can check the Assumption \ref{assm:continuity} would be to directly rely on \cite[Section 3]{peng1999fully} and check their Assumption H3.1 uniformly in $u^0$ and rely on the existence of solution to a fully coupled FBSDE. Note that the case in Proposition \ref{prop:Sufficient_ContinuityBeforeDiscretization} deviates from \cite[Section 3]{peng1999fully}. Indeed, our proof is self-contained and only requires the existence of a solution for a BSDE (and not FBSDE) \eqref{bsde} for a given $u^0,u^1,\tilde u^1$ which is a trivial problem. Although it is out of scope of this work, one can then check that the optimally controlled state together with the solution to \eqref{bsde} leads to a solution to the same FBSDE. 
   
\end{remark}

The following is a broad class of functions satisfying the assumption of Proposition~\ref{prop:Sufficient_ContinuityBeforeDiscretization}.
\begin{example}
\label{ex:convex}
Let $A$, $A^{\sigma}$ be matrices of dimensions $d\times d$, $B$, $B^{\sigma}$ be matrices of dimensions $d\times d^1$, 
$(C,C^{\sigma},C^L):\mathbb{R}^{d_0}\to \mathbb{R}\times  \mathbb{R}\times  \mathbb{R}_+$ and $(D,D^{\sigma}):\mathbb{R}^{d_0}\to (\mathbb{R}^d)^2$ be Lipschitz functions. Let $L_1^1:\mathbb{R}^{d\times d_1}\to \mathbb{R}$ convex in $x$ and strongly convex in $u^1$ and $L_1^2:\mathbb{R}^{d_0}\to \mathbb{R}_+$ and $g:\mathbb{R}^d\to \mathbb{R}$ have Lipschitz gradients (in $x$) %with $L_1^2\ge \tilde{\kappa}\, \|u^0\|$ {\color{red}here why?}\Xuwei{it should be $\|u^0\|$ rather than $\|x\|$}for some $\tilde{\kappa}>0$.
.
%%%%%
Define
\begin{align*}
            f(x,u^0,u^1) 
        & \eqdef 
            C(u^0)\,
            \big( A\, x + B u^1 \big) + D(u^0)
    \\
            \sigma(x,u^0,u^1) 
        &\eqdef
            C^{\sigma}(u^0)\, 
            \big( A^{\sigma} \, x + B^{\sigma} u^1 \big)  + D^\sigma (u^0)
    \\
            L_1(x,u^0,u^1) 
        &\eqdef 
            C^L(u^0)\,
            L_1^1(x,u^1) + L_1^2(u^0)
.
\end{align*}
The Hamiltonian associated to these $f$, $\sigma$, and $L_1$ and the map $g$ satisfy the assumption of Proposition~\ref{prop:Sufficient_ContinuityBeforeDiscretization}. 
%Assumption~\ref{assm:LipfsigLg}.
\end{example} 

In linear-quadratic Stackelberg games, e.g.,~\cite{yong2002leader}, the coefficients $f(x, u^0, u^1)$ and $\sigma(x, u^0, u^1)$ are linear in $(x, u^0, u^1)$, and the costs $L_i(x, u^0, u^1)$ and $g_i(x)$ $i=0, 1$ are quadratic forms of $(x, u^0, u^1)$. 
With an additional assumption that the weight matrix of the quadratic form for $u^1$ in $L_1$ is positive definite %{\color{red}what do you mean}\Xuwei{I think it is clear. It means if $L_1(x, u^0, u^1)$ is quadratic, and the weight matrix for %$u^1$ is positive definite, then ...}, 
the assumption of Proposition~\ref{prop:Sufficient_ContinuityBeforeDiscretization} is satisfied.

\section{Conclusion}
\label{s:Conclusion}
We show that, the best response map for a broad class of dynamic Stackelberg games with random effects, can be approximated by neural operators (Theorem~\ref{thrm:Main__BestResponse}).  
This implies that Stackelberg games may be solved (approximate computation of equilibria in Theorem~\ref{thrm:Objective}), and we can explicitly describe how they are played (approximate representation of best response map) without making highly stylized conditions needed in classical results needed to derive analytic expressions for the best response and the equilibirum itself (see e.g.\ Example~\ref{ex:perturbations__lin_prob_control}).  
Thus, our approach allows one to specify realistic dynamics and action sets, and obtain an approximate solution.
%{\color{red}change this}Thus, with our approach, one can specify realistic dynamics and sets of actions for real-life games and %still obtain an approximate solution.

We further showed that if the space of actions for the leader consists of perturbations of the optimal solution for a linearized version of the Stackelberg game, then one can approximate the best response map (for the general game) for the follower much more efficiently (Theorem~\ref{thrm:Main__BestResponse___goodrates}).
% {\color{red}not the correct theorem cited}

\subsection*{Future Research}
\label{s:Future_Research}

Our universal approximation theorem for non-linear operators between spaces of $\mathbb{F}$-adapted square-integrable processes, namely Theorem~\ref{thrm:UniversalApprox}, is quantitative.  Though, as with most approximation theorems between infinite-dimensional Hilbert spaces, the approximation rates are particularly insightful.  Nevertheless, we find that if one uses a trainable variant of the ``super-expressive'' activation function of \cite{zhang2022deep} (see~\eqref{eq:supreexpressive}) when defining our neural operator and if the compact subset on which the non-linear operator is sufficiently close to being finite-dimensional (see Definition~\ref{defn:exp_width}) then exponential approximation rates can be achieved (see Table~\ref{tab:complexity__lemma}).  The main challenge is then to verify that the non-linear operator being approximated and the relevant compact set of adapted open-loop controls satisfy the required compatibility conditions.  Constructing examples of Stackelberg games with these properties, i.e.\ games whose best-response operator is efficiently approximable, is a highly non-trivial but interesting project and is the objective of our future research.  

Nevertheless, we include the relevant quantitative universal approximation theorem and conditions for exponential approximation rates in the appendix of this paper so that these results may be used in other operator learning problem in stochastic analysis and its applications; e.g. to economics and finance.

\appendix

%\Xuwei{Background on the Wiener chaos was moved to Section~\ref{app:Background}}

\section{Counterexamples}
\label{s:CounterEx}
\textbf{The effective optimization problem of the leader can be discontinuous}
\hfill\\
\noindent In Proposition \ref{prop:Sufficient_ContinuityBeforeDiscretization}, we show that a sufficient condition for the continuous dependence of the optimal response function $U^{\star}$ is the strong convexity of the problem of the follower. We now provide the example of a deterministic game that shows that both $U^{\star}$ and $u^0\mapsto J_0(u^0, U^{\star}(u^0))$ might not be continuous without the strong convexity assumption. 
Consider the deterministic single period game where the controls are $u^i\in [0,1]$. 
Fix the loss functions
$l_1(u^0,u^1)=u^0 u^1$ and $l_0(u^0,u^1)=-u^1$ so that the leader's problem is 
\begin{align}\label{eq:value1}
\inf_{u^0\in [0,1]} \inf_{u^1\in \cR(u^0) }l_0(u^0,u^1)
\end{align}
and $\cR(u^0) \eqdef \{u^1\in [0,1]: l_1(u^0,u^1)\leq \inf l_1(u^0,\cdot)\}$. Note that $l_1$ is convex in $u_1$ but it lacks the strong convexity in $u_1$ needed in Proposition \ref{prop:Sufficient_ContinuityBeforeDiscretization}. Clearly
$\cR(u^0)=\{0\}$ if $u^0>0$ and $\cR(0)=[0,1]$. Thus, 
$$\inf_{u^1\in \cR(u^0) }l_0(u^0,u^1)=-1_{\{u^0=0\}}$$
which is discontinuous in $u_0$.

\section{Proofs}
\label{s:Proofs}
\subsection{Continuous Dependence of \texorpdfstring{$J_i$}{Ji}}

\noindent For the next result, it is convenient to recall that the $\mathcal{H}_T^{\infty}$ norm of a process $X_{\cdot}$ in $\mathcal{H}_T^2$ is given by
\[
        \|X_{\cdot}\|_{\mathcal{H}_T^{\infty}} 
    \eqdef 
        \mathbb{E}\biggl[
            \sup_{0\le t\le T}\,
                |X_t|
        \biggr]
.
\]
\begin{lemma} 
\label{lm:sup|Xu1-Xtildeu1|2}
Under Assumption~\ref{assm:LipfsigLg}, we have the following uniform continuity guarantees
\begin{align} 
& \Big\|
        \big| 
                X^{u^0, u^1}_{\cdot}
            - 
                X^{u^0, \tilde{u}^1}_{\cdot} 
        \big|^2  
    \Big\|_{\mathcal{H}_T^{\infty}}
 \leq  
    C \,
        \big\|
            u^1
            -
            \widetilde{u}^1
        \big\|_{\mathcal{H}_T^2}^2 , 
\label{estESup|Xu1-Xtildeu1|2}
\\
& \Big\|
        \big| 
                X^{u^0, u^1}_\cdot 
            - 
                X^{\tilde{u}^0, u^1}_\cdot   
        \big|^2   
    \Big\|_{\mathcal{H}_T^{\infty}}
 \leq  
    C \,
        \big\|
            u^0
            -
            \widetilde{u}^0
        \big\|_{\mathcal{H}_T^2}^2
,
\label{estESup|Xu0-Xtildeu0|2}
\end{align}
where $C\ge 0$ is a constant that depends on $T$ and the Lipschitz constant $K$ in Assumption~\ref{assm:LipfsigLg};   
moreover, 
\begin{align} 
& \Big\|
                X^{u^0, u^1}_{\cdot}
            - 
                X^{u^0, \tilde{u}^1}_{\cdot}   
    \Big\|_{\mathcal{H}_T^{\infty}}
 \leq  
    C \,
        \big\|
            u^1
            -
            \widetilde{u}^1
        \big\|_{\mathcal{H}_T^2} , 
\label{estESup|Xu1-Xtildeu1|}
\\
& \Big\|
                X^{u^0, u^1}_\cdot 
            - 
                X^{\tilde{u}^0, u^1}_\cdot   
    \Big\|_{\mathcal{H}_T^{\infty}}
 \leq  
    C \,
        \big\|
            u^0
            -
            \widetilde{u}^0
        \big\|_{\mathcal{H}_T^2}
. 
\label{estESup|Xu0-Xtildeu0|}
\end{align} 
%{\color{red}these two should not be needed.}
%\Xuwei{agree}
\end{lemma} 
\begin{proof} 
We prove \eqref{estESup|Xu1-Xtildeu1|2}, noting that the derivation of~\eqref{estESup|Xu0-Xtildeu0|2} is carried out in a nearly identical similar manner. 
Let $X$ and $\widetilde{X}$ be the strong solution of \eqref{dX} under $(u^0, u^1)$ and $(u^0, \widetilde{u}^1)$, respectively, 
i.e., $X = X^{u^0, u^1}$, $\widetilde{X} = X^{u^0, \tilde{u}^1}$.  
Denote $h(s)=h(X_s, u^0_s, u^1_s)$ and 
$\widetilde{h}(s)=h(\widetilde{X}_s, u^0_s, \widetilde{u}^1_s)$, where $h$ is as defined in Assumption~\ref{assm:LipfsigLg}.  

By Jensen's inequality, we have 
\begin{align} 
  |X_r-\widetilde{X}_r|^2   
= & \Big| \int_0^r f(s) - \widetilde{f}(s) ds 
 + \int_0^r \sigma(s) - \widetilde{\sigma}(s) d W_s \Big|^2  
 \notag \\ 
\leq & 2 \Big| \int_0^r f(s) - \widetilde{f}(s) ds \Big|^2 
 + 2 \Big| \int_0^r \sigma(s) - \widetilde{\sigma}(s) d W_s \Big|^2  
\notag \\ 
\leq & 2  \int_0^r | f(s) - \widetilde{f}(s) |^2 ds     
+ 2 \Big| \int_0^r \sigma(s) - \widetilde{\sigma}(s) d W_s \Big|^2  
\notag 
\end{align} 
We deduce that
\begin{align} 
  \sup_{0\leq r \leq t}  |X_r-\widetilde{X}_r|^2    
\leq   
 2  \int_0^t | f(s) - \widetilde{f}(s) |^2 ds   
 + 2 \sup_{0\leq r \leq t} \Big| \int_0^r \sigma(s) - \widetilde{\sigma}(s) d W_s \Big|^2 . 
 \label{est|Xu1-Xtildeu1|2int}   
\end{align}
By Burkholder-Davis-Gundy's inequality \cite[Theorem 11.5.5]{cohenelliotbook}, Tonelli's theorem \citep[Theorem 1.4.6]{cohenelliotbook}, and Assumption \ref{assm:LipfsigLg}, we have 
\begin{align} 
  \mathbb{E} \sup_{0\leq r \leq t} \Big| \int_0^r \sigma(s) - \widetilde{\sigma}(s) d W_s \Big|^2        
 \leq & C \cdot \mathbb{E}\int_0^t | \sigma(s) - \widetilde{\sigma}(s) |^2   ds  \notag \\ 
 %\leq & \int_0^t K \cdot \mathbb{E} \sup_{0\leq r \leq s} \big( | \sigma_0(r) - % \widetilde{\sigma}_0(r) |^2 + | \sigma_1(r) - \widetilde{\sigma}_1(r) |^2 \big) \, ds \notag \\
\leq & C \cdot \int_0^t  \mathbb{E} \sup_{0\leq r \leq s} 
  | X_r - \widetilde{X}_r |^2   
 +  |u^1_s - \widetilde{u}^1_s|^2 \, ds. 
 \label{est|Xu1-Xtildeu1|2diffusion} 
\end{align}  
By Tonelli's Theorem and Assumption~\ref{assm:LipfsigLg}, we have  
\begin{align}
  \mathbb{E} \int_0^t  | f(s) - \widetilde{f}(s) |^2    \, ds      
 \leq & C \cdot \int_0^t \big| X_s - \widetilde{X}_s \big|^2  
  + |u^1_s - \widetilde{u}^1_s |^2  \, ds     \notag \\ 
 \leq &  
 C \cdot \int_0^t  \mathbb{E} \sup_{0\leq r \leq s} 
  | X_r - \widetilde{X}_r |^2   
 + |u^1_s - \widetilde{u}^1_s|^2  \, ds 
  \label{est|Xu1-Xtildeu1|2drift} 
\end{align} 
Combining \eqref{est|Xu1-Xtildeu1|2int}, \eqref{est|Xu1-Xtildeu1|2diffusion}, and \eqref{est|Xu1-Xtildeu1|2drift}, we obtain 
\begin{align} 
  \mathbb{E} \sup_{0\leq r \leq t}  |X_r-\widetilde{X}_r|^2    
  \leq &  
 C \cdot \int_0^t  \mathbb{E} \sup_{0\leq r \leq s} 
  | X_s - \widetilde{X}_s |^2  ds  
  + C \cdot \mathbb{E} \int_0^t |u^1_s - \widetilde{u}^1_s|^2  \, ds . \notag   
\end{align} 
From the above inequality and Gr\"{o}nwall's inequality we obtain
\begin{align} 
& \mathbb{E} \Big[
    \sup_{0\leq t \leq T}\,
        \big| 
                X^{u^0, u^1}_t 
            - 
                X^{u^0, \tilde{u}^1}_t  
        \big|^2  
    \Big]
 \leq  
    C \cdot  \mathbb{E} \int_0^T  |u^1_t  - \tilde{u}^1_t|^2  dt 
\label{estESup|Xu1-Xtildeu1|2___v1}
.
\end{align} 
Arguing nearly identically, we may also obtain
\begin{align}
& \mathbb{E} \Big[
    \sup_{0\leq t \leq T}\,
            \big| 
                    X^{u^0, u^1}_t 
                - 
                    X^{\tilde{u}^0, u^1}_t  
            \big|^2   
    \Big]
 \leq  
    C \cdot  \mathbb{E} \int_0^T  |u^0_t  - \tilde{u}^0_t|^2 dt 
. 
\label{estESup|Xu0-Xtildeu0|2___v1}
\end{align}
Note that the right-hand side of~\eqref{estESup|Xu1-Xtildeu1|2___v1} (resp.\ \eqref{estESup|Xu0-Xtildeu0|2___v1}) is simply a scalar multiple (by a factor of $C\ge 0$) of the squared norm between $u^1$ and $\widetilde{u}^1$ (resp.\ $u^0$ and $\widetilde{u}^0$).  
%%%
By definition of the $\|\cdot\|_{\mathcal{H}_T^{\infty}}$ norm applied to the ``squared difference processes'' $
\big| 
        X^{u^0, u^1}_{\cdot}
    - 
        X^{u^0, \tilde{u}^1}_{\cdot}
\big|^2   
$ and $
\big| 
        X^{u^0, u^1}_{\cdot}
    - 
        X^{\tilde{u}^0, u^1}_{\cdot}
\big|^2   
$,~\eqref{estESup|Xu1-Xtildeu1|2___v1} and~\eqref{estESup|Xu0-Xtildeu0|2___v1} can be re-expressed as~\eqref{estESup|Xu1-Xtildeu1|2} and~\eqref{estESup|Xu0-Xtildeu0|2}. 
\eqref{estESup|Xu1-Xtildeu1|} and~\eqref{estESup|Xu0-Xtildeu0|} follows from 
\eqref{estESup|Xu1-Xtildeu1|2} and~\eqref{estESup|Xu0-Xtildeu0|2} by Cauchy-Schwarz inequality; 
thus concluding our proof.
\end{proof} 

\begin{lemma} 
\label{lm:|Ji(u0u1)-Ji(tu0tu1)|<=|u0-tu0|+|u1-tu1|}
The costs $J_i$, $i=0, 1$ are Lipschitz in $u^0$ and $u^1$ such that 
for each $u^i,\tilde{u}^i \in \mathcal{U}_i \cap \mathcal{H}^2_T$, $i=0, 1$, 
\begin{align} 
\big| J_i(u^0, u^1) - J_i(\tilde{u}^0, \tilde{u}^1) \big| 
\leq 
C \, \big( 
        \big\|
            u^0
            -
            \widetilde{u}^0
        \big\|_{\mathcal{H}_T^2} 
        + 
        \big\|
            u^1
            -
            \widetilde{u}^1
        \big\|_{\mathcal{H}_T^2} 
        \big) , 
\label{|Ji(u0u1)-Ji(tu0tu1)|<=|u0-tu0|+|u1-tu1|}
\end{align} 
where $C$ is a constant depending on $T$ and the Lipschitz constant $K$ in Assumption~\ref{assm:LipfsigLg}. 
\end{lemma} 

\begin{proof} 
For each $u^i,\tilde{u}^i \in \mathcal{U}_i \cap \mathcal{H}^2_T$, $i=0, 1$, 
it follows from Assumption~\ref{assm:LipfsigLg} that 
\begin{align} 
 & | J_i(u^0, u^1) - J_i(\widetilde{u}^0, \widetilde{u}^1) | \notag \\  
 \leq & 
 \mathbb{E} \Big[ \int_0^T  \big| L_i( X^{u^0, u^1}_t, u^0_t, u^1_t ) 
 - L_i( X^{\tilde{u}^0, \tilde{u}^1}_t,  \widetilde{u}^0_t, \widetilde{u}^1_t  ) \big| \, dt  
 + \big| g_i(X^{u^0, u^1}_T) 
 - g_i( X^{\tilde{u}^0, \tilde{u}^1}_T ) \big| \Big] \notag \\ 
 \leq & 
  \mathbb{E} \Big[ \int_0^T   K ( |X^{u^0, u^1}_t - X^{\tilde{u}^0, \tilde{u}^1}_t | 
  + |u^0_t - \widetilde{u}^0_t | + |u^1_t - \widetilde{u}^1_t| )  \, dt  
 + K | X^{u^0, u^1}_T - X^{\tilde{u}^0, \tilde{u}^1}_T |  \Big] . 
 \notag \\ 
 \leq & 
 K(1+T) \big\|  X^{u^0, u^1} - X^{\tilde{u}^0, \tilde{u}^1}   \big\|_{\mathcal{H}^\infty_T} 
  + K \big( \big\| u^0 - \widetilde{u}^0 \big\|_{\mathcal{H}_T^2} 
  + \big\| u^0 - \widetilde{u}^0 \big\|_{\mathcal{H}_T^2}  \big) , \notag 
\end{align} 
where the last inequality is due to the Cauchy-Schwarz inequality. 
By the triangle inequality and Lemma~\ref{lm:sup|Xu1-Xtildeu1|2}, we have 
\begin{align} 
 \Big\|
             X^{u^0, u^1}_{\cdot}
            - 
                X^{\tilde{u}^0, \tilde{u}^1}_{\cdot}   
    \Big\|_{\mathcal{H}_T^{\infty}} 
\leq &  
2 \Big( \Big\|
                X^{u^0, u^1}_{\cdot}
            - 
                X^{\tilde{u}^0, u^1}_{\cdot}   
    \Big\|_{\mathcal{H}_T^{\infty}} 
    + 
  \Big\|
                X^{ \tilde{u}^0, u^1}_{\cdot}
            - 
                X^{\tilde{u}^0, \tilde{u}^1}_{\cdot}   
    \Big\|_{\mathcal{H}_T^{\infty}} 
    \Big) 
\notag \\ 
 \leq &  
    C \, \Big( 
        \big\|
            u^0
            -
            \widetilde{u}^0
        \big\|_{\mathcal{H}_T^2} 
        + 
        \big\|
            u^1
            -
            \widetilde{u}^1
        \big\|_{\mathcal{H}_T^2} 
        \Big)  . 
\notag 
\end{align} 
It then follows that 
\begin{align} 
\big| J_1(u^0, u^1) - J_1(\tilde{u}^0, \tilde{u}^1) \big| 
\leq 
C \, \Big( 
        \big\|
            u^0
            -
            \widetilde{u}^0
        \big\|_{\mathcal{H}_T^2} 
        + 
        \big\|
            u^1
            -
            \widetilde{u}^1
        \big\|_{\mathcal{H}_T^2} 
        \Big). 
\notag 
\end{align} 

\end{proof}

\begin{lemma} 
\label{lm:contJ1u0R(u0)} Assume  Assumption~\ref{assm:LipfsigLg} and that for all $u^0\in \kkk_0$, $\mathcal{R}(u^0)$ is not empty and choose $U^{\star}(u^0) \in \mathcal{R}(u^0)$. Then,
the map $u^0\in \kkk_0\mapsto J_1(u^0,U^{\star}(u^0))$ is continuous;  
particularly, for each $u^0,\tilde{u}^0 \in \mathcal{U}_0\cap \mathcal{H}^2_T$, 
it satisfies for the follower that
\begin{align}
    &    \big| J_1(u^0, U^{\star}(u^0)) - J_1(\widetilde{u}^0, U^{\star}(\widetilde{u}^0) ) \big|  
    \le
        C 
        \,
        \big\|u^0 - \widetilde{u}^0 \big\|_{\mathcal{H}^2_T} . 
\label{|J1(u0U(u0))-J1(tu0U(tu0))|<=C|u0-tu0|}
\end{align} 
The constant $C$ is depend on $T$ and the Lipschitz constant $K$ in Assumption~\ref{assm:LipfsigLg}.  
\end{lemma}
{\begin{remark}
With the assumption of Lemma \ref{lm:contJ1u0R(u0)}, $u^0\in \kkk_0\mapsto J_1(u^0,U^{\star}(u^0))$ is continuous but as shown in counterexample in Section \ref{s:CounterEx}, $u^0\in \kkk_0\mapsto J_0(u^0,U^{\star}(u^0))$ might fail to be continuous. 
\end{remark}}
\begin{proof} 

By \eqref{|Ji(u0u1)-Ji(tu0tu1)|<=|u0-tu0|+|u1-tu1|}, we have for any $u \in \mathcal{U}_0 \cap \mathcal{H}_T^2$,  
\begin{align} 
 J_1(u^0, u) 
 \leq & J_1(\tilde{u}^0, u) + |J_1(\tilde{u}^0, u) - J_1(u^0, u) |    \notag \\ 
  \leq & J_1(\tilde{u}^0, u) + C \| u^0 - \tilde{u}^0 \|_{\mathcal{H}^2_T} , \notag 
\end{align} 
and furthermore 
\begin{align}  
  J_1(u^0, U^{\star}(u^0))=\inf_ u J_1(u^0,u) 
  \leq & \inf_ u J_1(\tilde u^0,u) +  C \| u^0 - \tilde{u}^0 \|_{\mathcal{H}^2_T} \notag \\ 
 = & J_1(\tilde u^0, U^{1,\star}(\tilde u^0) ) +  C \| u^0 - \tilde{u}^0 \|_{\mathcal{H}^2_T} . 
 \label{J1(u0u1)<=J1(tu0tu1)+K|u0-tu0|} 
\end{align} 
Exchanging the roles of $(u^0, U^{1,\star}(u^0))$ and $(\tilde{u}^0, U^{1,\star}(\tilde u^0))$ in \eqref{J1(u0u1)<=J1(tu0tu1)+K|u0-tu0|}, we have 
\begin{align} 
 J_1( \tilde{u}^0, U^{\star}( \tilde{u}^0) )  \leq J_1( u^0, U^{\star}(u^0)) +  C \| u^0 - \tilde{u}^0 \|_{\mathcal{H}^2_T} .  
 \label{J1(tu0tu1)<=J1(u0u1)+K|tu0-u0|} 
\end{align} 
Combining \eqref{J1(u0u1)<=J1(tu0tu1)+K|u0-tu0|} and \eqref{J1(tu0tu1)<=J1(u0u1)+K|tu0-u0|}, we obtain \eqref{|J1(u0U(u0))-J1(tu0U(tu0))|<=C|u0-tu0|}.  
\end{proof}

\begin{lemma}[H\"{o}lder Regularity of the Leader's Utility on Optimal Response]
\label{lm:stabilityJ0}
Under Assumptions~\ref{assm:LipfsigLg} and~\ref{assm:continuity}, there exists $C,\alpha>0$ depending only on $K,T$ and the Holder norm in Assumption \ref{assm:continuity} so that for each $u^0,\tilde{u}^0 \in \mathcal{U}_0\cap \mathcal{H}^2_T$,  
we have 
\[
    \big| J_0(u^0, U^{\star}(u^0)) - J_0(\widetilde{u}^0, U^{\star}(\widetilde{u}^0) ) \big|  
\leq   
    \tilde{\omega}\big(
        \|u^0-\widetilde{u}^0\|_{\mathcal{H}_T^2}
    \big)
\]
where $\tilde{\omega}(t) \eqdef C\max\{|t|^{\alpha},|t|\}$ for each $t\ge 0$.
\end{lemma} 

\begin{proof}[{Proof of Lemma~\ref{lm:stabilityJ0}}]
By \eqref{J1(u0u1)<=J1(tu0tu1)+K|u0-tu0|} and Assumption \ref{assm:continuity}, we have 
\begin{align} 
\label{eq:PRF_lm:stabilityJ0__detailed_nearlycomplete}
 \big| J_0(u^0, U^{\star}(u^0)) - J_0(\widetilde{u}^0, U^{\star}(\widetilde{u}^0) ) \big| 
 \leq & 
  C  \|u^0-\widetilde{u}^0\|_{\mathcal{H}_T^2}
 + C   \big\| U^{\star}(u^0) - U^{\star}(\widetilde{u}^0) \big\|_{\mathcal{H}_T^2} \notag \\ 
  \leq & 
  C  \|u^0-\widetilde{u}^0\|_{\mathcal{H}_T^2}
 + (C\widetilde{C})   \|u^0-\widetilde{u}^0\|_{\mathcal{H}_T^2}^{\alpha} , 
\end{align} 
for some constant $C\ge 0$ depending only on $T,K$ and $\tilde C,\alpha\geq 0$ the Holder continuity parameters of the best response function given by Assumption \ref{assm:continuity}.  Define $C^{\prime}\eqdef C\max\{1,\tilde{C}\}/2$
.  Note that $a+b \le 2\max\{a,b\}$ for every $a,b\ge 0$ then~\eqref{eq:PRF_lm:stabilityJ0__detailed_nearlycomplete} implies that
\begin{equation}
\label{eq:PRF_lm:stabilityJ0__detailed}
% \begin{aligned} 
    \big| J_0(u^0, U^{\star}(u^0)) - J_0(\widetilde{u}^0, U^{\star}(\widetilde{u}^0) ) \big|  
\leq %&
    C^{\prime} \cdot 
    \max\Big\{ 
        \|u^0-\widetilde{u}^0\|_{\mathcal{H}_T^2}
    ,
        \|u^0-\widetilde{u}^0\|_{\mathcal{H}_T^2}^{\alpha}
    \Big\} 
.
% \end{aligned}
\end{equation}
Define the modulus of continuity $\tilde{\omega}(t)\eqdef C^{\prime}\,\max\{|t|^{\alpha},|t|\}$.  
Since $t \leq t^{\alpha}$ when $t\in [0,1)$ and $t\geq t^{\alpha}$ when $t\in [1,\infty)$ then,~\eqref{eq:PRF_lm:stabilityJ0__detailed} cleans up as
\begin{equation}
    \big| J_0(u^0, U^{\star}(u^0)) - J_0(\widetilde{u}^0, U^{\star}(\widetilde{u}^0) ) \big|  
\leq   
    \tilde{\omega}
        \big(
            \|u^0-\widetilde{u}^0\|_{\mathcal{H}_T^2}
        \big)
.
\end{equation}
Relabelling $C$ as $C^{\prime}$ concludes our proof.
\end{proof}

\subsection{Proof of The Sufficient conditions for Assumption \ref{assm:continuity}\texorpdfstring{ In Proposition~\ref{prop:Sufficient_ContinuityBeforeDiscretization}}{}}

\begin{proof}[{Proof of Proposition~\ref{prop:Sufficient_ContinuityBeforeDiscretization}}]
    We fix $u^0\in \uuu_0$ and choose $u^1,\tilde u^1\in \uuu_1,\,\lambda\in [0,1]$. Denote $X^\lambda_t$ the state controlled by the pair $(u^0,u^\lambda)=(u^0,\lambda u^1+(1-\lambda)\tilde u^1)$ respectively.
Thus, 
\begin{align*}
    &J_1(u^0,u^\lambda)-\lambda J_1(u^0, u^1)-(1-\lambda)J_1(u^0,\tilde  u^1)\\
    &=\E \Big[ \int_0^T L_1( X^\lambda_t,  u^0_t ,  u^\lambda_t ) -\lambda L_1( X^1_t,  u^0_t ,   u^1_t )-(1-\lambda) L_1( X^0_t,  u^0_t , \tilde  u^1_t )dt \Big] \\
     &+\E \Big[ g_1(X^\lambda_T)-\lambda g_1( X^1_T)-(1-\lambda) g_1( X^0_T) \Big] \\         
    &=\E \Big[ \int_0^T L_1( X^\lambda_t,  u^0_t ,  u^\lambda_t ) -\lambda L_1( X^1_t,  u^0_t ,   u^1_t )-(1-\lambda) L_1( X^0_t,  u^0_t , \tilde  u^1_t )dt \Big] \\
     &-\E \Big[ \lambda (g_1( X^1_T)-g_1(X^\lambda_T))+(1-\lambda) (g_1( X^0_T)-g_1(X^\lambda_T)) \Big] \\
    &\leq- \lambda \E \Big[ \int_0^T L_1( X^1_t,  u^0_t ,   u^1_t ) -L_1( X^\lambda_t,  u^0_t ,  u^\lambda_t )dt + \nabla_x g_1( X^\lambda_T)^\top(X^1_T-X^\lambda_T)\Big] \\
     &-(1-\lambda) \E \Big[\int L_1( X^0_t,  u^0_t , \tilde  u^1_t )- L_1( X^\lambda_t,  u^0_t ,  u^\lambda_t )dt+\nabla g_1( X^\lambda_T)^\top(X^0_T-X^\lambda_T)\Big]
\end{align*}
    where we used the the convexity of $g_1$ to obtain the last line. We now provide an upper bound for the last two terms. For fixed $\lambda$, by the definition of $X^\lambda, u^\lambda$ and our assumptions on $H_1$, the function $(y,z)\mapsto \nabla_x H_1(X^\lambda_t,u^0_t,u^\lambda_t,y,z)$ is uniformly Lipschitz continuous and $\nabla g_1(X^\lambda_T)$ is square integrable. Thus, there exists a unique solution $(Y^\lambda_t,Z^\lambda_t)$ to the BSDE
\begin{align}\label{bsde}
    dY^\lambda_t&= -\nabla_x H_1(X^\lambda_t,u^0_t,u^\lambda_t,Y^\lambda_t,Z^\lambda_t)dt+Z^\lambda_tdW_t\\
    Y^\lambda_T&=\nabla g_1(X^\lambda_T)
\end{align}
so that by Ito's formula we have
\begin{align*}
    &\E \Big[\nabla g_1( X^\lambda_T)^\top(X^0_T-X^\lambda_T)\Big]=\E \Big[  {Y^\lambda_T}^\top(X^0_T-X^\lambda_T)\Big]\\
    &=\E \Big[  \int_0^T{Y^\lambda_t}^\top (f(X^0_t ,u^0_t,\tilde u^1_t)-f(X^\lambda_t ,u^0_t, u^\lambda_t)) dt\Big]\\
    &+\E \Big[ Tr\left({Z^\lambda_t}^\top (\sigma(X^0_t ,u^0_t,\tilde u^1_t)-\sigma(X^\lambda_t ,u^0_t, u^\lambda_t))\right)dt\Big]\\
    &-\E \Big[\nabla_x H_1(X^\lambda_t,u^0_t,u^\lambda_t,Y^\lambda_t,Z^\lambda_t) (X^0_t-X^\lambda_t)dt \Big]\\
    &=\E \Big[  \int_0^TH_1(X^0_t ,u^0_t,\tilde u^1_t,Y^\lambda_t,Z^\lambda_t)-H_1(X^\lambda_t,u^0_t ,u^\lambda_t,Y^\lambda_t,Z^\lambda_t) dt\Big]\\
    &-\E \Big[\nabla_x H_1(X^\lambda_t,u^0_t,u^\lambda_t,Y^\lambda_t,Z^\lambda_t) (X^0_t-X^\lambda_t)dt \Big]\\
    &-\E \Big[  \int_0^TL_1(t,X^0_t ,u^0_t,\tilde u^1_t)-L_1(t,X^\lambda_t ,u^0_t,u^\lambda_t) dt\Big].
\end{align*}
Thus, we obtain 
\begin{align*}
    &\E \Big[\nabla g_1( X^\lambda_T)^\top(X^0_T-X^\lambda_T)\Big]+\E \Big[  \int_0^TL_1(X^0_t ,u^0_t,\tilde u^1_t)-L_1(X^\lambda_t ,u^0_t,u^\lambda_t)dt\Big]\\
    &=\E \Big[  \int_0^TH_1(X^0_t ,u^0_t,\tilde u^1_t,Y^\lambda_t,Z^\lambda_t)-H_1(X^\lambda_t,u^0_t ,u^\lambda_t,Y^\lambda_t,Z^\lambda_t) dt\Big]\\
    &-\E \Big[\nabla_x H_1(X^\lambda_t,u^0_t,u^\lambda_t,Y^\lambda_t,Z^\lambda_t) (X^0_t-X^\lambda_t)dt \Big]
\end{align*}
and 
\begin{align*}
    &\E \Big[\nabla g_1( X^\lambda_T)^\top(X^1_T-X^\lambda_T)\Big]+\E \Big[  \int_0^TL_1(X^1_t ,u^0_t, u^1_t)-L_1(X^\lambda_t ,u^0_t,u^\lambda_t)dt\Big]\\
    &=\E \Big[  \int_0^TH_1(X^1_t ,u^0_t, u^1_t,Y^\lambda_t,Z^\lambda_t)-H_1(X^\lambda_t,u^0_t ,u^\lambda_t,Y^\lambda_t,Z^\lambda_t) dt\Big]\\
    &-\E \Big[\nabla_x H_1(X^\lambda_t,u^0_t,u^\lambda_t,Y^\lambda_t,Z^\lambda_t) (X^1_t-X^\lambda_t)dt \Big].
\end{align*}
Combining these equalities, we obtain 
\allowdisplaybreaks
\begin{align*}
&- \lambda \E \Big[ \int_0^T L_1( X^1_t,  u^0_t ,   u^1_t ) -L_1( X^\lambda_t,  u^0_t ,  u^\lambda_t )dt + \nabla_x g_1( X^\lambda_T)^\top(X^1_T-X^\lambda_T)\Big] \\
     &-(1-\lambda) \E \Big[\int L_1( X^0_t,  u^0_t , \tilde  u^1_t )- L_1( X^\lambda_t,  u^0_t ,  u^\lambda_t )dt+\nabla_x g_1( X^\lambda_T)^\top(X^0_T-X^\lambda_T)\Big]\\
     &=\E \Big[  \int_0^TH_1(\lambda X^0_t+(1-\lambda)X^1_t,u^0_t ,u^\lambda_t,Y^\lambda_t,Z^\lambda_t)-\lambda H_1(X^1_t ,u^0_t, u^1_t,Y^\lambda_t,Z^\lambda_t)dt\Big]\\
     &- \E \Big[ \int_0^T(1-\lambda)H_1(X^0_t ,u^0_t,\tilde u^1_t,Y^\lambda_t,Z^\lambda_t)dt\Big]\\
    &+\E \Big[  \int_0^TH_1(X^\lambda_t,u^0_t ,u^\lambda_t,Y^\lambda_t,Z^\lambda_t)-H_1(\lambda X^0_t+(1-\lambda)X^1_t,u^0_t ,u^\lambda_t,Y^\lambda_t,Z^\lambda_t)dt\Big]\\
    &+\E \Big[\int_0^T\nabla_x H_1(X^\lambda_t,u^0_t,u^\lambda_t,Y^\lambda_t,Z^\lambda_t) (\lambda X^1_t+(1-\lambda)X^0_t-X^\lambda_t)dt \Big].
\end{align*}
By the convexity of $H_1$ in $x$ and strong convexity in $u^1$ we have that 
\begin{align*}
         &\E \Big[  \int_0^TH_1(\lambda X^0_t+(1-\lambda)X^1_t,u^0_t ,u^\lambda_t,Y^\lambda_t,Z^\lambda_t)-\lambda H_1(X^1_t ,u^0_t, u^1_t,Y^\lambda_t,Z^\lambda_t)\Big]\\
     &- \E \Big[ \int_0^T(1-\lambda)H_1(X^0_t ,u^0_t,\tilde u^1_t,Y^\lambda_t,Z^\lambda_t)dt\Big]\leq -\frac{\kappa\lambda(1-\lambda)}{2}\E \Big[  \int_0^T|u^1_t-\tilde  u^1_t|^2dt\Big]
\end{align*}
and 
\begin{align*}
    H_1(x,u^0 ,u^1,y,z)-H_1(\tilde x,u^0 ,u^1,y,z)+ \nabla_x H_1(x,u^0 ,u^1,y,z)(\tilde x-x)\leq 0.
\end{align*}
Thus, 
\begin{align*}
    &J_1(u^0,u^\lambda)-\lambda J_1(u^0, u^1)-(1-\lambda)J_1(u^0,\tilde  u^1)\leq-\frac{\kappa\lambda(1-\lambda)}{2}\E \Big[  \int_0^T|u^1_t-\tilde  u^1_t|^2dt\Big]
\end{align*}
which is the strong convexity in $u^1$.

Due to this strong convexity, for all $u^0$, there exists a unique minimizer for the optimal response of the follower that we denote $U^{\star}( u^0)$.
To prove \eqref{eq:contd}, we use the first order optimality condition for $U^{1,*}( u^0)$ which reads
$$J_1(u^0,u^1)-J_1(u^0,U^{\star}( u^0))\geq \frac{\kappa}{2}\E \Big[  \int_0^T|u^1_t-U^{\star}_t( u^0)|^2dt\Big]$$
which is \eqref{eq:contd} for $u^1=U^{\star}(\tilde u^0)$.

To conclude the proof of the Proposition it remains to prove the $1/2$ Holder continuity of the best response function which allows us to verify Assumption \ref{assm:continuity}.
By \eqref{eq:contd}, we have 
\allowdisplaybreaks 
\begin{align}
& \big\| U^{\star}(u^0) - U^{\star}(\tilde u^0) \big\|_{\mathcal{H}_T^2}  \notag\\ 
& \le \frac{2}{\kappa^{1/2}} \Big\{  \big| J_1(\tilde u^0,U^{\star}(\tilde u^0))-J_1( u^0,U^{\star}(u^0)) \big|^{1/2} 
+ \big| J_1( u^0, U^{\star}(\tilde u^0))-J_1( \tilde u^0,U^{\star}(\tilde u^0)) \big|^{1/2} 
\Big\} . 
\label{|U(u0)-U(tu0)|}
   \end{align}
   
By~\eqref{|Ji(u0u1)-Ji(tu0tu1)|<=|u0-tu0|+|u1-tu1|} and \eqref{|J1(u0U(u0))-J1(tu0U(tu0))|<=C|u0-tu0|}, we have   
\begin{align} 
& \big| J_1( u^0, U^{\star}(\tilde u^0))-J_1( \tilde u^0,U^{\star}(\tilde u^0)) \big| 
\leq C \cdot  \|u^0 - \widetilde{u}^0 \|_{\mathcal{H}^2_T} ,  
\label{|J1(u0tildeu1)-J1(tildeu0-tildeu1)|Lip} \\ 
&  \big| J_1(\tilde u^0,U^{\star}(\tilde u^0))-J_1( u^0,U^{\star}(u^0)) \big|  
    < C \cdot \,
     \|u^0 - \widetilde{u}^0 \|_{\mathcal{H}^2_T}   . 
        \label{|J1(u0u1)-J1(tildeu0u1)|Lip}
\end{align} 

Then the desired result follows from \eqref{|U(u0)-U(tu0)|}, \eqref{|J1(u0tildeu1)-J1(tildeu0-tildeu1)|Lip} 
and \eqref{|J1(u0u1)-J1(tildeu0u1)|Lip}.

\end{proof}

\subsection{Additional Background on the Wiener Chaos}
\label{app:Background}

For any time $0\le t\le T$, since the $\sigma$-algebra $\mathcal{F}_t$ is generated by $\{W_{s}\}_{0\le s\le t}$ then, any $u\in L^2(\mathcal{F}_t)$ admits the following \textit{Wiener chaos expansion}
\begin{equation}
\label{eq:WienerChaos}
        u 
    = 
        \mathbb{E}[u] 
        + 
        \sum_{i=0}^{\infty}\, 
            \int_0^T
            \int_0^{s_n}
            \dots
            \int_0^{s_2}
            \,
                f_i(s_n,\dots,s_1)
            \,
                dW_{s_1}\dots dW_{s_n}
\end{equation}
where, for each $i\in \mathbb{N}_+$, the deterministic functions $f_i$ belongs to $L^2(\mathcal{S}_{i,T})$ where $\mathcal{S}_{i,T}\eqdef \{(s_j)_{j=1}^i\in [0,T]^i:\, 0<s_1<\dots<s_i<T\}$.

We consider an alternative description of the Wiener chaos expansion of any random variable $u\in L^2(\mathcal{F}_t)$ for a given $0\le t\le T$, which both extends more easily to multiple dimensions and does not require to lengthy computation of multiple iterated stochastic integrals.
For any $i\in \mathbb{N}$, the Hermite polynomials $(h_i)_{i\in \mathbb{N}}$ are the eigenfunctions of the generator $\frac{d^2}{dx^2}-x\frac{d}{dx}$ of the Ornstein-Uhlenbeck process $dX_t = -X_t +\sqrt{2}dW_t$.  For each $i\in \mathbb{N}_+$, the $i^{th}$ Hermite polynomial $h_i$ is given recursively by Rodrigues' formula as
\[
h_i(x) = \frac{(-1)^i}{i!} e^{x^2/2} \frac{d^i}{dx^i}e^{-x^2/2}
\qquad 
h_0(x) = 1.
\]
The multi-dimensional version of the $i^{th}$ iterated integral in~\eqref{eq:WienerChaos} is given by
\begin{equation}
\label{eq:ith_WienerChaos}
    \sum_{|\alpha|=i}
    \,
        \beta_{i,\alpha_j}
        \prod_{j=1}^{J_i}
        \,
        h_{\alpha_j}\Big(
            \int_0^T\,
                \psi_{i,k}(s)
            \,
            dB_s
        \Big)
\end{equation}
where $\alpha\eqdef (\alpha_1,\dots,\alpha_{J_i})$ is a multi-index consisting of positive integers with $i=|\alpha|\eqdef \sum_{j=1}^{J_i}\, \alpha_i$, $\beta_{i,\alpha_j}\in \mathbb{R}$, and where $(\psi_{i,k})_{i,k\in \mathbb{Z};0\le k,\,\frac{k+1}{2^n}\le t}$ is an orthonormal basis of $L^2([0,t])$.  
Here, we will consider the \textit{Haar (wavelet) system} given by
\[
    \psi_{i,k}(s) \eqdef 2^i\big(
            I_{[t\frac{k}{2^i},t\frac{1+2k}{2^{i+1}})}(s)
        -
            I_{[t\frac{1+2k}{2^{i+1}},t\frac{k+1}{2^i})}(s)
    \big)
.
\]
Since $L^2([0,t])\subset L^2([0,T]) $ we can replace the $t$ in indicator functions with $T$ and we can consider the \textit{Haar (wavelet) system} on the larger space $L^2([0,T])$ where 
 $(\psi_{i,k})_{i,k\in \mathbb{N};0\le k,\,\frac{k+1}{2^i}\le 1}$ and  
 \[
    \psi_{i,k}(s) \eqdef 2^i\big(
            I_{[T\frac{k}{2^i},T\frac{1+2k}{2^{i+1}})}(s)
        -
            I_{[T\frac{1+2k}{2^{i+1}},T\frac{k+1}{2^i})}(s)
    \big)
.
\]
By elementary functional analytic considerations, for each $0\le t\le T$, $L^2(\mathcal{F}_t)$ is closure of the span on the of orthogonal $L^2(\mathcal{F}_t)$ random variables 
\begin{equation}
\label{eq:truncated_WienerChaos}
        \prod_{\tilde{j}=1}^{j}
        \,
        h_{\alpha_{\tilde{j}}}\Big(
            \int_0^{{t}}\,
                \psi_{i,k}(s)
            \,
            dB_s
        \Big)
\end{equation}
where $j\in \mathbb{N}$, and $i,k\in \mathbb{N};0\le k,\,\frac{k+1}{2^i}\le \frac{t}{T}$,
for some $I\in \mathbb{N}$, $\beta_0,\beta_{1,\alpha_1},\dots,\beta_{I,\alpha_{J_I}}\in \mathbb{R}$, and $(f_i)_{i\in \mathbb{N}}$ is an orthonormal basis of $L^2([0,t])$.  
Since $(\psi_{i,k})_{i,k\in \mathbb{N};0\le k,\,\frac{k+1}{2^i}\le \frac{t}{T}}$ is piecewise constant, then the stochastic integrals in~\eqref{eq:truncated_WienerChaos} simplify to 
\[
\int_0^{{t}}\,\psi_{i,k}(s)\,dB_s = 
2^i \,W_{\frac{
{t}%T
k}{2^i}}
-2^{i+1}\,W_{\frac{
{t}%T
(1+2k)}{2^{i+1}}}
+2^i\,W_{\frac{
{t}%T
(k+1)}{2^i}}
\]
where $T\frac{k+1}{2^i}\le t$. 
Thus, the random variables $\big\{
u_{i,j,k}
:\,
j\in \mathbb{N},\,i,k\in \mathbb{N},\,\frac{k+1}{2^i}\le \frac{t}{T}
\big\}\subset L^2(\mathcal{F}_t)$ where 
\begin{equation}
u^{(t)}_{i,j,k}\eqdef 
    \prod_{\tilde{j}=1}^{j}
    h_{\alpha_{\tilde{j}}}\big(
    2^i \,W_{\frac{tk}{2^i}}
    -2^{i+1}\,W_{\frac{t(1+2k)}{2^{i+1}}}
    +2^i\,W_{\frac{t(k+1)}{2^i}}
    \big)
\label{base_l2_ft__1}  
\end{equation}
form an \textit{orthonormal} basis of $ L^2(\mathcal{F}_t)$.  Conveniently, each of the $u_{i,j,k} $ can be computed without any explicit stochastic integration.  We refer to \cite{Nualart_MalCRT_2006} for more details on Wiener Chaos.

Next, we will derive our universal approximation guarantees for our transformer model.

\subsection{Proof of Universal Approximation Theorem\texorpdfstring{~\ref{thrm:UniversalApprox}}{}}
\label{s:Proofs__ss:UAT}

Our main universal approximation theorem relies on the following orthonormal basis of simple processes in $\calH_T^2$, defined by linear combinations of the elementary processes in~\eqref{eq:simple_processes}.

\begin{lemma}[{Orthonormal Basis of $\mathcal{H}^2_T$}]
\label{lem:Orthonormal}
The collection of simple processes \\
$
\mathcal{S}\eqdef \big\{
u_{i,j,k}^{s_1,s_2}
:\,
i,j,k,s_1,s_2\in \mathbb{N},\, s_2+1\leq 2^{s_1},\frac{k+1}{2^i}\le \frac{s_2}{{2^{s_1}}}
\big\}$ where
% each $u_{i,j,k}^{s_1,s_2}$ is defined by
\[
u_{i,j,k}^{s_1,s_2}(t,\omega)\eqdef \psi_{s_1,s_2}(t)\cdot u_{i,j,k}^{{T}}(\omega)
\]
is an orthonormal basis of $\mathcal{H}^2_T$.
\end{lemma}

\begin{proof}[Proof of Lemma~\ref{lem:Orthonormal}]
We denote the set of indices 
\begin{align*}
\mathcal{I}& \eqdef  \{(i,j,k,s_1,s_2)\in \mathbb{N}^5 : \frac{k+1}{2^i}\le \frac{s_2}{2^{s_1}}\le 1,
\,
\frac{s_2+1}{2^{s_1}}\le 1\}, \\
\mathcal{I}(s_1,s_2)& \eqdef  \{(i,j,k)\in \mathbb{N}^3 : (i,j,k,s_1,s_2) \in \mathcal{I}\}, \quad (s_1,s_2) \in \mathbb{N}^2.\\
\end{align*}
and observe the following equality
$$\mathcal{S} = \left\{u^{s_1,s_2}_{i,j,k} \right\}_{(i,j,k,s_1,s_2)\in \mathcal{I}}\subset \mathcal{H}^2([0,T]).$$
It is easy to check that $\mathcal{S}$ is orthonormal in $\mathcal{H}^2([0,T])$ as the set $\{\psi_{s_1,s_2} \}_{(s_1,s_2)\in \mathbb{N}^2}$ is orthornormal in $L^2([0,T])$, and for all $(s_1,s_2)\in \N^2$, $\left(u^{T}_{i,j,k}\right)_{(i,j,k)\in \mathcal{I}(s_1,s_2)}$ is orthonormal in $ L^2(\mathcal{F}_{T})$. It remains to show that $\mathcal{H}^2([0,T])$ is the closure of the span of $\mathcal{S}$.

We denote $\hat{\mathcal{H}}_0$ to the set of simple processes $Z\in \mathcal{H}^2([0,T])$ satisfying
 \begin{align}\label{simple.process}
Z \eqdef  \sum_{l=1}^{m} \xi_{l} \mathbbm{1}_{[t_{2l-1},t_{2l})},
\end{align} 
where $m \in \mathbb{N}$, $(t_i)_{i=1}^{m+1}$ is a strictly increasing sequence of real numbers satisfying $t_1=0, t_{2m}<T$, and $\xi_{l} \in L^2(\mathcal{F}_{t_{2l-1}})$, for $1\leq l \leq m$.

We notice that $\bar{\hat{\mathcal{H}}}_0=\mathcal{H}^2([0,T])$. Indeed, we observe that for every simple process $Z:=\sum_{l=1}^m \xi_l\mathbbm{1}_{[t_l,t_{l+1})} $, $t_1 =0, t_{m+1}=T$ can be written as the $\mathcal{H}^2([0,T])$-limit of the sequence $(Z^n)_{n\in\mathbb{N}} \subset \hat{\mathcal{H}}_0$:
\begin{equation}
    Z^n_t = \sum_{l=1}^m \xi_l \mathbbm{1}_{[t_l,t_{l+1}-\frac{1}{n})}. 
\end{equation}
The previous shows that $\bar{\hat{\mathcal{H}}}_0= \bar{\mathcal{H}}_0 = \mathcal{H}^2([0,T])$.
\\
Next, assume that $Z\in \hat{\mathcal{H}}_0$ is such that: for all $ U \in \mathcal{S}$ 
\begin{equation}\label{complete.property}
\int_0^T\mathbb{E}\left[U_tZ_t \right]dt=0.
\end{equation}
 We notice that for any $\tilde{s}_1\in \mathbb{N}$ sufficiently large there exists $\tilde{s}_2\in\mathbb{N}$   satisfying
\begin{align}\label{eq.s1.s2}
&1+\tilde{s}_2\leq 2^{\tilde{s}_1}, \nonumber\\
&t_1 <\frac{\tilde{s}_2T}{2^{\tilde{s}_1}} \leq t_2 < \frac{(1+2\tilde{s}_2)T}{2^{\tilde{s}_1+1}} < \frac{(1+\tilde{s}_2)T}{2^{\tilde{s}_1}} <t_3 . 
\end{align}

Then, for any $(i,j,k) \in \mathcal{I}(\tilde{s}_1,\tilde{s}_2)$, we set
\begin{equation}\label{process.test}
    {U}  \eqdef  u_{i,j,k}^{\tilde{s}_1, \tilde{s}_2}\in \mathcal{S}.
\end{equation}
\\
Plugging in the process \eqref{process.test} into \eqref{complete.property}, we obtain
\begin{equation}\label{result.process.test}
\int_0^T\mathbb{E}\left[{U}_tZ_t \right]dt = 2^{\tilde{s}_1}\left(t_2-\frac{\tilde{s}_2T}{2^{\tilde{s}_1}}\right)\E\left[u^{T}_{i,j,k}\xi_{1} \right] = 0, \quad \forall (i,j,k) \in \mathcal{I}(\tilde{s}_1,\tilde{s}_2).
\end{equation}
Finally, we will prove that the closure of the span of 
$\left(u_{i,j,k}^{T}\right)_{(i,j,k)\in \mathcal{I}(\tilde{s}_1,\tilde{s}_2)}$ contains $L^2\left(\mathcal{F}_{t_1}\right)$. Using \eqref{eq.s1.s2}, we observe that for all $(i,k) \in \mathbb{N}^2$, $1+k\leq 2^{i}$ the following inequality holds:
\begin{equation*}
    0\leq \frac{t_1(k+1)}{T2^i} < \frac{\tilde{s}_2}{2^{\tilde{s}_1}} \leq 1.
\end{equation*}
Using that the set of the dyadic rationals is dense in $[0,1]$, there exists a sequence $(i_n,k_n)\in \mathbb{N}^2$, $ k_n+1\leq 2^{i_n},$ satisfying: 
\begin{align*} 
    \lim_{n\rightarrow \infty}\frac{\left(k_n+1\right)}{2^{i_n}} &= \frac{t_1(k+1)}{T2^i}, \\
    \frac{t_1(k+1)}{T2^i}&\leq \frac{\left(k_n+1\right)}{2^{i_n}} \leq \frac{\tilde{s}_2}{2^{\tilde{s}_1}} , \quad n \in \mathbb{N}.
\end{align*}
Hence, for all $j \in \mathbb{N}$, $\left(u^T_{i_n,j,k_n}\right)_{n\in \mathbb{N}} \subset \mathcal{S}$. 
By continuity of the paths of the Brownian motion and the Hermite polynomials, we obtain that $(u^{T}_{i_n,j,k_n})_{n\in \mathbb{N}}$ converges  $\mathbb{P}$-a.s. to $u^{(t_1)}_{i,j,k}$. Moreover, we observe that  $\left(\xi_1u^T_{i_n,j,k_n} \right)_{n\in \mathbb{N}}$ is uniformly integrable. Indeed, for $0 <\epsilon_0$ small enough, applying H\"{o}lder's inequality, there exist constants $p_1:=\frac{3}{2(1+\epsilon_0)} $, and $p_2:= \frac{1}{1-\frac{1}{p_1}}$ such that
\allowdisplaybreaks
\begin{align*}    
& \sup_{n\in \mathbb{N}}\mathbb{E}\left[\left|\xi_1u^{T}_{i_n,j,k_n}\right|^{1+{\epsilon_0}}\right] \\
&\leq \mathbb{E}\left[|\xi_1|^{(1+\epsilon_0)p_1}\right] \sup_{n\in \mathbb{N}} \mathbb{E}\left[\prod_{\tilde{j}=1}^{j}\left|h_{\tilde{j}}\left(
    2^{i_n} \,W_{\frac{Tk_n}{2^{i_n}}}
    -2^{i_n+1}\,W_{\frac{T(1+2k_n)}{2^{i_n+1}}}
    +2^{i_n}\,W_{\frac{T(k_n+1)}{2^{i_n}}}\right)\right|^{p_2}\right] \\
    &<\infty. 
\end{align*} 
We therefore deduce that
\begin{equation}
   \mathbb{E} \left[\xi_1u^{(t_1)}_{i,j,k} \right] =  \lim_{n\rightarrow \infty} \mathbb{E} \left[\xi_1u^T_{i_n,j,k_n} \right] =  0.
\end{equation}
Using completeness of the basis $\{u^{(t_1)}_{i,j,k}\}_{(i,j,k) \in \mathbb{N}^3}$ in $L^2(\mathcal{F}_{t_1})$ we obtain that 
$$\xi_1 = 0, \quad \mathbb{P} \text{-}a.s.$$ 

Finally, we repeat the same argument in \eqref{eq.s1.s2} for every addend in \eqref{simple.process}, obtaining 
$$Z = 0, \quad dt\otimes \mathbb{P} -a.e.$$
Hence, 
\begin{equation}\label{simple.cap.h0}
    \overline{\operatorname{span}(\mathcal{S})}^\perp \cap \hat{\mathcal{H}}_{0} = \{ 0 \},
\end{equation}
where $\mathcal{H}_0$ denotes the set of simple processes in $\mathcal{H}^2([0,T])$.
Finally, using the decomposition 
$$\mathcal{H}^2([0,T])= \overline{\operatorname{span}(\mathcal{S})}\oplus \overline{\operatorname{span}(\mathcal{S})}^{\perp},$$
and \eqref{simple.cap.h0}, we have that $\hat{\mathcal{H}}_0\subset \overline{\operatorname{span}(\mathcal{S})}$. The later implies that $\overline{\operatorname{span}(\mathcal{S})}= \mathcal{H}^2([0,T])$.
\end{proof}

Generally, deep learning faces the curse of dimensionality in finite-dimensions; see e.g.~\cite{lanthaler2023curse}.  Nevertheless, the impact of infinite-dimensionality on the parametric complexity of deep learning models can be reduced by considering the following \textit{trainable} version of the ``super-expressive'' activation function of \cite{zhang2008improved} designed to exploit the bit-extraction mechanism of \cite{bartlett2019nearly}. 
The next lemma provides quantitative rates for attentional neural operators with the activation function~\eqref{eq:supreexpressive}; since, in that case, they are not overwhelmingly large (as is the case approximation of general Lipschitz non-linear operators).

\begin{lemma}[Approximation Of Lipschitz Operators with By Attentional Neural Operator]
\label{lem:UAT_1}
Let $\kkk_0\subseteq \mathcal{H}_T^2$ be compact, $F:\calH_T^2\to\calH_T^2$ be an $L$-Lipschitz (non-linear) operator, and consider respective ``dimension reduction'' and ``approximation'' errors $\varepsilon_D,\varepsilon_A>0$.  There exists an attentional neural operator $\hat{F}:\calH_T^2\to\calH_T^2$ satisfying
\[
    \sup_{u\in \kkk_0}\,
        \|
           F(u)
            -
            \hat{F}(u)
        \|_{\calH^2_T}
\le
    \varepsilon_D
    +
    \varepsilon_A
.
\]
Furthermore, the complexity of the neural operator $\hat{F}$ is recorded in Table~\ref{tab:complexity__lemma}.
\end{lemma}
We provide explicit quantitative parameter estimates in the special cases where the compact set $\kkk_0$ and the target neural operator are compatible.  We consider the following notion of a small compact subset of a Banach space.
\begin{definition}[$(r,f)$-Exponentially Ellipsoidal]
\label{defn:exp_width}
Let $r>0$, $f:\calH_T^2\to\calH_T^2$, and fix an orthonormal basis $\{u_i\}_{i=0}^{\infty}$ of $\mathcal{H}_T^2$.  A subset $\kkk \subseteq \calH_T^2$ is of $(r,f)$-exponentially width if the following holds for each $u\in \kkk$:  
\begin{itemize}
    \item[(i)] $u=\sum_{i=1}^{\infty}\, \beta_iu_i$ and $|\beta_i|\lesssim e^{-ri}$,
    \item[(ii)] $f(u)=\sum_{i=1}^{\infty}\, c_iu_i$ and $|c_i|\lesssim e^{-ri}$.
\end{itemize}
\end{definition}

Exponentially ellipsoidal compact sets can be efficiently approximated by low-dimensional representations arising from projections onto the relevant basis.  However, they may still be large in metric entropy (i.e.\ they may be difficult to cover by a few small metric balls).  This is not the case if there is something akin to a latent ``low-dimensional submanifold'' on which the data/approximation is focused.  The following definition makes this rigorous for our infinite-dimensional setting.
\begin{definition}[$(r,f,d)$-Exponential Manifold]
\label{defn:exp_manifold}
Let $r>0$, $f:\calH_T^2\to\calH_T^2$, fix an orthonormal basis $\{u_i\}_{i=0}^{\infty}$ of $\mathcal{H}_T^2$, and let $\widetilde{\kkk}$ be an $(r,f)$-exponentially ellipsoidal subset of $\mathcal{H}_T^2$.  A compact subset $\kkk \subseteq \widetilde{\kkk}$ is said to be an $(r,f,d)$-Exponential Manifold if there exists a $1$-Lipschitz ``latent parameterization'' map $\pi:\mathbb{R}^d\to \mathcal{H}_T^2$ and $\kkk=\pi(\{z\in \mathbb{R}^d:\,\|z\|\le 1\})$.
\end{definition}
Notably, the map $\pi$ need not be known nor be injective (as in~\cite{KratsiosTakashiLassas_CoDOperatorLearning_2024}), nor does it need to be inverted by the deep learning model (either explicitly or implicitly during the approximation theorem). Instead, it simply encodes (in a possibly non-linear way) a low-dimensional stricture into the compact set of controls, allowing for an efficient approximation by controlling the entropy number, see e.g.~\cite{CarlAmazingPaper,LorentzEntropApprox,petrova2023lipschitz}, of the compact set on which the approximation is performed number.

\begin{table}[H]%[!htbp]%
    \centering
    \begin{tabular}{@{}lll@{}}
    \toprule
    \textbf{Param.} & \textbf{Example~\ref{ex:Activation_Superexpressive}} & 
    % \textbf{Poly. W.} & 
\textbf{Example~\ref{ex:Activation_Standard}} \\
    \midrule
    No.\ Param & 
    % \multicolumn{3}{c}{$
    % \mathcal{O}\big(
    %     N(N^2d^4 + Q)
    % \big)
    % $}
    $
    \mathcal{O}\Big(
        \varepsilon_D^{-3\ln(C)/r}
        \,
        \ln\big(
            \varepsilon_D^{-1/r}
        \big)^4
    \Big)
    $
    % & 
    % $
    % \mathcal{O}\Big(
    %     e^{3rC \, \varepsilon_D^{r/(1-r)}}
    %     \,
    %     \varepsilon_D^{4/(1-r)}
    % \Big)
    % $
    & 
    Finite
    \\
    % \midrule
    Depth & 
    % \multicolumn{3}{c}{$
    %     \mathcal{O}\big(
    %         d
    %         N
    % \big)$}
    $
    \mathcal{O}\Big(
        \ln(\varepsilon_D^{-1/r})
        \,
        \varepsilon_D^{-\ln(C)/r}
    \Big)
    $
    % &
    % $\mathcal{O}\Big(
    %     e^{rC\,\varepsilon_D^{r/(1-r)}}
    %     \,
    %     \varepsilon_D^{1/(1-r)}
    % \Big)$
    &
    Finite
    \\
    Width & 
    % \multicolumn{3}{c}{$
    % \mathcal{O}\big(
    %             N
    % \big)$} 
    $
            \mathcal{O}\big(
                \varepsilon_D^{
                % 1
                -\ln(C)/r}
            \big)
    $
    % & 
    % $
    % \mathcal{O}\big(
    %     e^{rC \varepsilon_D^{r/(1-r)}}
    % \big)
    % $ 
    &  
    $
    \mathcal{O}\big(
                d^{d+1} N^{2d+2} \varepsilon_D^{-3d-3}
                % \operatorname{diam}(K)^{d+1}
            \big)
    $
    \\
    \midrule
    Decoding Dim.\ ($Q$) & $\mathcal{O}(\ln(\varepsilon_D^{-1/r}))$ 
    % & $\mathcal{O}(\varepsilon_D^{1/(1-r)}) $
    & Finite\\
    Encoding Dim.\ ($d$) 
    % & $\mathcal{O}(\ln(\varepsilon_D^{-1/r}))$ 
    & $\mathcal{O}(\varepsilon_D^{1/(1-r)}) $ & Finite\\
    % \midrule
    Att. Complexity ($N$) & 
    % $
    %         % \mathcal{O}\big(
    %             %%% OLD 
    %             % \varepsilon_D^{
    %             % % 1
    %             % -\ln(C)/r}
    %         % \big)
    %             %%% FIXED
    %             (r\,(1+(C/\varepsilon_A)^2)^{1/2})^{\mathcal{O}(
    %             \ln(\varepsilon_D^{-1/r})^2
    %             )}
    % $
    $
    \big(\tilde{c}\,\varepsilon_A^{-1}\,\ln(\varepsilon_D^{-1/r})^{1/2} \big)^{c(\ln(\varepsilon_D^{-1/r})}
    $
    % & $
    % \mathcal{O}\big(
    %     e^{rC \varepsilon_D^{r/(1-r)}}
    % \big)
    % $ 
    &  
    Finite
    % $
    % \mathcal{O}\big(
    %             d^{d+1} N^{2d+2} \varepsilon_D^{-3d-3}
    %             % \operatorname{diam}(K)^{d+1}
    %         \big)
    % $
    \\
    \bottomrule
    \end{tabular}
    \caption{\textbf{Complexity of the neural operator.} 
    \textit{Case 1:} $\kkk$ is an $(r,f,d)$-exponential manifold in controls in $\mathcal{H}_T^2$ and $\sigma$ is the super-expressive activation function with neuron-specific skip-connections in~\eqref{ex:Activation_Superexpressive}; $c,\tilde{c}>0$ are absolute constants.  \\
    \textit{Case 2:} $\kkk$ is an arbitrary compact subset of controls in $\mathcal{H}_T^2$ and $\sigma$ is the standard non-trainable activation functions of Example~\eqref{ex:Activation_Standard}.}
        \label{tab:complexity__lemma}
\end{table}

The error $\varepsilon_D>0$ in Table~\ref{tab:complexity__lemma} expresses the ``dimension reduction'' error resulting from the encoding ($\mathcal{E}$) and decoding ($\mathcal{D})$ maps used in the definition of our attentional neural operator, in Definition~\ref{defn:NO}.  That is, $\varepsilon_D$ expresses the error made in encoding infinite dimensional objects, namely $\mathbb{F}$-adapted processes, into finite-dimensional objects, namely vectors in some Euclidean spaces.  Once the neural operator has implicitly transformed the approximation problem as an approximation problem between finite-dimensional spaces, it is approximated by the MLP ($f$) between the encoding and decoding layers of the attentional neural operator.  
The error $\varepsilon_A>0$ in Table~\ref{tab:complexity__lemma} expresses the error incurred in this finite-dimensional approximation step.

\begin{proof}[{Proof of Lemma~\ref{lem:UAT_1}}]
Fix respective ``dimension reduction'' and ``approximation'' errors $\varepsilon_D,\bar{\varepsilon}_A>0$. 

\noindent\textbf{Step 1 - Finite Dimensional Encoding}
\hfill\\
Enumerate $\mathcal{S}=\{s_i\}_{i=1}^{\infty}$, where $\mathcal{S}$ is as in Lemma~\ref{lem:Orthonormal}.
For any $d\in \N$ (which we fix retroactively) define the $1$-Lipschitz encoder $\mathcal{E}_d:\mathcal{H}_T^2\to \R^d$ given, for each $u\in\mathcal{H}_T^2$, by
\[
        \mathcal{E}_d(u_{\cdot})
    \eqdef 
        (\langle u,s_j \rangle_{\calH_T^2})_{j=1}^d
    .
\]
Consider its right-inverse $\iota_d:\R^d\to \calH_T^2$ given, for each $x\in \R^d$, by
\[
    \iota_d(x)
    \eqdef 
    \sum_{i=1}^d\,x_i\,s_i
.
\]
Observe that $\iota_d$ is an isometric embedding; we will come back to this point shortly. 

\begin{itemize}
    \item \textbf{$\kkk_0$ is the exponentially ellipsoidal $\kkk$ in Definition~\ref{defn:exp_width}:}
    By the exponential decay condition in Definition~\ref{defn:exp_width}, we have that: for each $u\in \kkk$, with representation $u=\sum_{i=1}^{\infty}\,\beta_i\,s_i$ the following error estimate holds by orthonormality of the $(s_i)_{i=1}^{\infty}$
\begin{equation}
\label{eq:basis_truncation}
\begin{aligned}
    \big\|
        u
        -
        \iota_d\circ \mathcal{E}_d(u)
    \big\|_{\calH^2_T}^2
=& 
    \big\|
            \sum_{i=1}^{\infty}\,\beta_i\,s_i
        -
            \iota_d\circ \mathcal{E}_d(u)
    \big\|_{\calH^2_T}^2
\\
= &
    \sum_{i=1}^{\infty}\,
        |\beta_i|^2\,
        I_{i\le d}\,
        \|s_i\|^2_{\calH_T^2}
\\
= &
    \sum_{i=d+1}^{\infty}\,
        |\beta_i|^2
\\
\le & 
    \frac{C\,e^{-rd}}{1-e^{-r}} \eqdef \tilde{C}_K\,e^{-rd}
\end{aligned}
\end{equation}
where $\tilde{C}_K\eqdef C/(1-e^{-r})$.  
Fix $\varepsilon_0>0$.
We now can retroactively set $d\eqdef 
\ln\Big(
\frac{2^r\tilde{C}^r_K}{(1-e^{-r})^r}\, \varepsilon_0^{-r}
\Big)
=
\ln(C_K\varepsilon_0^{-r})
\in \mathcal{O}(\ln(\varepsilon_0^{-1/r}))
$ where $C_K\eqdef 
\Big(\frac{C}{3L(1-e^{-r})}\Big)^{1/r}>0
$.  
    \item \textbf{Exponential Sub-manifold:} If $\kkk_0$ satisfies Definition~\ref{defn:exp_manifold}, then this case is implied by the previous case (i.e.\ that of exponentially ellipsoidal compacta),
    \item \textbf{General $\kkk_0$:} If $\kkk_0$ is general, then by the $1$-bounded approximation property (e.g.\ see \cite{Szarek_BAPNoBasis_1987}) of Hilbert spaces with orthonormal bases (which are simply Banach spaces with Schauder bases), for every $\epsilon_0>0$ there is some $d\in \N$ for which $\sup_{x\in \kkk_0}\, \|x-\mathcal{E}_d(x)\|\le \epsilon_0$.
\end{itemize}
In each case, we have that
\begin{equation}
\label{eq:encoding_error}
\sup_{u\in \kkk_0}\, \|u-\iota_d\circ \mathcal{E}_d(u)\| \le \varepsilon_0
.
\end{equation}

\noindent Let $L\ge 0$ denote the optimal Lipschitz constant of $F$.
We note that the map 
\[
    f^{(1)}\eqdef F\circ \iota_d:\R^d\to \mathcal{H}_T^2
\]
is $(L,\alpha)$-H\"{o}lder since $F$ is $(L,\alpha)$-H\"{o}lder
and since is $1$-Lipschitz.  In particular, the Lipschitz constant of $f^{(1)}$ is independent of $d$.  

\noindent\textbf{Step 2 - Quantization of The Image and The Domain:}

Since $\kkk_0$ is compact and since $\mathcal{E}_d$ is continuous, then $\mathcal{E}_d(\kkk_0)$ is compact and thus $\mathcal{E}_d(\kkk_0)$ is totally bounded.  
Therefore, for every $\varepsilon_1>0$ there exists a finite subset $\{x_n\}_{n=1}^{N_{\varepsilon_1}}\subseteq \mathcal{E}_d(\kkk_0)$, of minimal cardinality $N\eqdef N_{\varepsilon_1}$ (a so-called minimal $\varepsilon_1$-net) such that:
\begin{equation}
\label{eq:packing_in_K}
\max_{x\in \kkk_0}\,\min_{n=1,\dots,N_{\varepsilon_1}}\, \|x-x_n\|_{2} 
< 
\Big(
\frac1{L}
    \frac{\varepsilon_1}{3}
\Big)^{1/\alpha}
.
\end{equation}
Since, $f^{(1)}$ is an $L$-Lipschitz surjection of $\kkk_0$ onto $f(\kkk_0)$ then~\eqref{eq:packing_in_K} implies that
\begin{equation}
\label{eq:packing_in_fK}
\begin{aligned}
\begin{aligned}
    \max_{x\in \kkk_0}\,\min_{n=1,\dots,N_{\varepsilon_1}}\, 
        \|f^{(1)}(x)-f^{(1)}(x_n)\|_{\calH_T^2} 
\le &
    L\,\max_{x\in \kkk_0}\,\min_{n=1,\dots,N_{\varepsilon_1}}\, \|x-x_2\|
    ^{\alpha}
    _2
\\
<  
&
        L \biggl(
            \Big(
                \frac{\varepsilon_1}{3L} 
            \Big)^{1/\alpha}
        \biggr)^{\alpha}
    = 
        \frac{\varepsilon_1}{3}
.
\end{aligned}
\end{aligned}
\end{equation}
By Lemma~\ref{lem:Orthonormal}, $(s_i)_{i=1}^{\infty}$ is an orthonormal basis of the Hilbert space $\calH_T^2$.  Therefore, it realizes the $1$-bounded approximation property. This means that since $F(\kkk_0)$ is compact then, for each $Q\in \N$
\begin{equation}
\label{eq:proj}
    \max_{y\in F(\kkk_0)}\,
        \big\|
            y
            -
            P_Q(y)
        \big\|_{\calH_T^2}
    \,\overset{N\mapsto \infty}{\to}\,
    0
    \mbox{ and }
    \|P_Q\|_{op}\le 1
\end{equation}
where $P_Q:\calH_T^2\to \calH_T^2$ is the (rank $Q$) orthogonal projection operator of $\calH^2_T$ onto $\operatorname{span}(\{s_i\}_{i=1}^Q)$; that is, $P_{Q}(u)\mapsto \sum_{i=1}^Q\, \langle s_i,u\rangle_{\calH_T^2}$; where $\|\cdot\|_{op}$ denotes the operator norm.  Thus, for each $\varepsilon_2>0$ (to be fixed retroactively) there exists some $Q\eqdef Q_{\varepsilon_2}\in \N$ for which
\begin{equation*}
% \label{eq:proj}
    \max_{y\in F(\kkk_0)}\,
        \big\|
            y
            -
            P_{Q}(y)
        \big\|_{\calH_T^2}
    \le 
        \varepsilon_2
.
\end{equation*}
In the special case where $F(\kkk_0)$ satisfies Definition~\eqref{defn:exp_width} then, by a similar computation to Step $1$~\eqref{eq:basis_truncation}, we obtain the following bounds on $Q\eqdef Q_{\varepsilon_2}$:
\begin{itemize}
    \item \textbf{Exponentially Ellipsoidal $\kkk_0$:} $Q\in \mathcal{O}(\ln(\varepsilon_D^{-1/r} ))$,
    % \item \textbf{Polynomially Ellipsoidal $K$:} $Q\in \mathcal{O}(\varepsilon_2^{1/(1-r)})$,
    \item \textbf{Exponential Sub-manifold:} As before, if $\kkk_0$ satisfies Definition~\ref{defn:exp_manifold}, then this case is implied by the previous case,
    \item \textbf{General $\kkk_0$:} $Q\to \infty$ as $\varepsilon_2\to 0$.
\end{itemize}

% \noindent\textbf{Remark:} 
\noindent
\textit{To summarize this step, the set $\{x_n\}_{n=1}^{N}$ discretized $\mathcal{E}_d(\kkk_0)$ and the set $\{y_n\}_{n=1}^{N}$ discretized the ``finitely parameterized'' image of $\kkk_0$ under $F$.}

\noindent\textbf{Step 3 - Simplicialization of Target Function:}
\hfill\\
\textit{Our next objective is to replace the target function $F$, with a function which maps between $\mathcal{E}_d(\kkk_0)$ to an $N$-simplex and which, informally speaking, is an approximate continuous selection to the nearest neighbour problem
\[
    x\mapsto \underset{n=1,\dots,N}{\operatorname{argmin}}\, \|F(x)-y_n\|
    .
\]
First, we construct an ``projection-like/extremal'' version of this solution, as in Step 4 of the proof of \citep[Theorem 3.8]{acciaio2023designing}.  In the second step, we ``mollify'' that function to make it comparable with the softmax operation.}

Fix $\varepsilon_3>0$.
Let $\mathcal{P}_1(\{x_n\}_{n=1}^{N},\mathcal{W}_1)$ denote the $1$-Wasserstein space over $\{x_n\}_{n=1}^{N}$ with respect to the 
$\alpha$-snowflaked of the
Euclidean distance $\|\cdot\|_2^{\alpha}$ on the inherited finite set $\{x_n\}_{n=1}^{N}$;
; i.e.\ the metric $\mathbb{R}^d\times \mathbb{R}^d\ni (x,\tilde{x})\mapsto \|x-\tilde{x}\|_2^{\alpha}$ {restricted to the set $\{x_n\}_{n=1}^N$.}
% {\color{red}clarify last sentence}. 
By \citep[Theorem 3.2]{BrueDiMarinaStra_JFA_RandProj_2021}, there exists a Lipschitz map (a so-called weak random projection) $\Pi:(\mathcal{E}_d(\kkk_0),\|\cdot\|_2)\to \mathcal{P}_1(\{x_n\}_{n=1}^{N},\mathcal{W}_1)$ with the property that: for each $x\in \mathcal{E}_d(\kkk_0)$ if $x\in \{x_n\}_{n=1}^{N}$ then $\Pi(x)=\delta_x$; where $\delta_x$ is the pointmass on $x$.  Furthermore, the Lipschitz constant $L_{\Pi}$ of $\Pi$ is at-most $c
\log_2(C_{(\{x_n\}_{n=1}^N,\|\cdot\|_2
^{\alpha}
)})$ where $c>0$ is an absolute constant and $C_{(\{x_n\}_{n=1}^N,\|\cdot\|_2
^{\alpha}
)}>0$ is the doubling constant of the set $\{x_n\}_{n=1}^N$ with respect to the $\alpha$-snowflake of the Euclidean distance restricted to $\{x_n\}_{n=1}^N$. By \citep[Lemma 9.3]{Robinson_DimEmbAtr_Book__2011}, since the inclusion of $\{x_n\}_{n=1}^{N}$ into $\mathbb{R}^d$ is an isometric embedding (with respect to the Euclidean distance on $\R^{d}$) then the doubling constant of $(\{x_n\}_{n=1}^{N},\|\cdot\|_2)$ is no larger than that of $(\mathbb{R}^d,\|\cdot\|_2)$.  By \citep[Lemma 9.3]{Robinson_DimEmbAtr_Book__2011}, the doubling constant of $\R^{d}$ in the Euclidean metric is $2^{d+1}$.  
Thus, the first statement in \citep[Lemma 7.1]{acciaio2023designing} implies that doubling 
of $(\{x_n\}_{n=1}^N,\|\cdot\|_2^{\alpha})$ is at-most equal to the doubling constant of $\mathbb{R}^d$ to the power of 
$
\lceil 
    \frac1{\alpha} 
\rceil.
$
Hence, 
the Lipschitz constant $L_{\Pi}$ of $\Pi$ is
\begin{equation}
\label{eq:Lip_const_rand_proj}
        L_{\Pi}
    \le 
        c
        \log_2(C_{(\{x_n\}_{n=1}^N,
        \|\cdot\|_2
        ^{\alpha}
        )})
    \le
        c
        \Big\lceil 
            \frac1{\alpha} 
        \Big\rceil
        \log_2(C_{(\{x_n\}_{n=1}^N,
        \|\cdot\|_2
        )})
    \le
        c
        \Big\lceil 
            \frac1{\alpha} 
        \Big\rceil
        (d+1)
    \le 
        \tilde{c}
        {\Big\lceil 
            \frac1{\alpha} 
        \Big\rceil}
        d
    \eqdef 
        C_{\Pi}
\end{equation}
where $\tilde{c}\eqdef 2\max\{1,c\}>0$.

As shown in \citep[Equations (41)-(46)]{acciaio2023designing}, the map $\iota_N:\mathcal{P}(\{x_n\}_{n=1}^{N},\mathcal{W}_1)\ni \mu=\sum_{n=1}^{N}\,w_n\delta_{x_n}\to (w_n)_{n=1}^{N}\in (\Delta_{N},\|\cdot\|_2)$ is $\frac{2}{\varepsilon_1}$-Lipschitz since the minimal distance between any distinct pairs of points in $\{x_n\}_{n=1}^{N}$ is $\varepsilon_1$.  
Together with the right-hand side of~\eqref{eq:Lip_const_rand_proj}, this shows that the $\alpha$-H\"{o}lder
constant $L_{f^{(2)}}$ of composite map $f^{(2)}:\iota_N\circ \Pi_N:\mathcal{E}_d(\kkk_0)\to \Delta_N$ is bounded-above by
\begin{equation}
\label{eq:Lipschitz_projected_version_of_f1}
        L_{f^{(2)}}
    \le 
        L_{\Pi}
        \,
        \frac{2}{\varepsilon_1}
    \le 
        \tilde{c}
        \,
        \Big\lceil 
            \frac1{\alpha} 
        \Big\rceil
        \,
        N
        \,
        \frac{2}{\varepsilon_1}
    \eqdef 
        \tilde{L}_{f^{(2)}}
    .
\end{equation}
For $n=1,\dots,N$, let $y_n\eqdef P_{Q_{\varepsilon_3}}(f^{(1)}(x_n)) = \sum_{i=1}^{Q_{\varepsilon_3}}\,\langle f^{(1)}(u),s_i\rangle_{\calH_T^2}s_i$.  Define the Lipschitz map $\eta_{N}:\Delta_N\to \calH^2_T$ by
\begin{equation}
\label{eq:definition_eta}
    \eta_{N}:
    w\mapsto 
        \sum_{i=1}^{N}\,
        w_i\, y_n
.
\end{equation}
Note that $\eta_{N}:(\Delta_N,\|\cdot\|_2)\to (\calH_T^2,\|\cdot\|_{\calH^2_T})$ is $L^{\eta}_{\varepsilon_1,\varepsilon_2}$-Lipschitz with optimal Lipschitz constant, which we denote by $\operatorname{Lip}(\eta_{\epsilon_{N}})\ge 0$, bounded-above by the constant $L^{\eta}_{\varepsilon_1,\varepsilon_2}>0$ defined by
\allowdisplaybreaks
\begin{align}
        \big|
            \eta_{N}(w)
            -
            \eta_{N}(v)
        \big|
    \le &
        \sum_{i=1}^{N}
        \,
            \|y_i\|_{\calH_T^2}
            \,
            |w_i-v_i|
\\
\nonumber
    \le &
        (\operatorname{diam}(f(\mathcal{E}_d(\kkk_0)))+2\epsilon_2)
        \sum_{i=1}^{N}
        \,
            |w_i-v_i|
\nonumber
    \le 
        (L\operatorname{diam}(\kkk_0)
        ^{\alpha}
        +2\epsilon_2)
        \sum_{i=1}^{N}
        \,
            |w_i-v_i|
\\
% \nonumber
    \le &
        (L\operatorname{diam}(\kkk_0)
        ^{\alpha}
        +2\epsilon_2)
        \sqrt{N}\,
        \|w-v\|_2
\label{eq:ub_OLC}
    \eqdef 
        L^{\eta}_{\varepsilon_1,\varepsilon_2}
        \|w-v\|_2
\end{align} 
where $w,v\in \Delta_{N}$ are arbitrary and we have used the inequality $\|\cdot\|_1\le \sqrt{N}\,\|\cdot\|_2$ (on $\R^{N}$).

\textit{Next, we show that $\eta_{N}\circ f^{(2)}$ approximates $F$ on $\mathcal{E}_d(\kkk_0)$.}

Since $\calH_T^2$ is a QAS space (see \citep[Definition 3.4]{acciaio2023designing} with $p=1$ and $C_{\eta}=1$, as shown in \citep[Example 5.1]{acciaio2023designing}) then Step $4$ of the proof of \citep[Theorem 3.8]{acciaio2023designing} holds unaltered in our setting (with $\mathcal{X}=\mathcal{E}_d(\kkk_0)$, $(\mathcal{Y},d_{\mathcal{Y}}=(\calH^2_T,\|\cdot\|_{\calH^2_T})$, the $\alpha$-H\"{o}lder target function with respect to the {$\alpha$}-H\"{o}lder seminorm $L>0$ of $f^{(1)}$).  

Set%
\footnote{In the notation of \citep[Proof of Theorem 3.8 - Step 3]{acciaio2023designing}, we have set $\bar{\varepsilon}_A\eqdef \varepsilon_Q\eqdef \varepsilon_D/3$.}~%
$\varepsilon_1\eqdef \frac{\varepsilon_D}{3\cdot 3L C_{\Pi}} = \frac{\varepsilon_D}{9L \tilde{c}d}
>0
$ and $\varepsilon_2\eqdef \frac{\varepsilon_D}{9}
>0$.  Arguing identically to the \citep[Proof of Theorem 3.8 - Step 3]{acciaio2023designing}, following \citep[Equation 51]{acciaio2023designing}, we conclude that
\begin{equation}
\label{eq:better_Estimate}
    \sup_{x\in \mathcal{E}_d(\kkk_0)}\,
        \|f^{(1)}(x)-\eta_{N}\circ f^{(2)}(x)\|_{\calH^2_T}
    \le
        \frac{\varepsilon_D}{3}
    .
\end{equation}
where, for us, our improved estimate on $C_{\Pi}$ was given in~\eqref{eq:Lip_const_rand_proj}; as is summarized in~\eqref{eq:better_Estimate}.

\noindent \textit{We now modify the map $f^{(2)}$ so that it takes values in the range of the softmax function; that is, in the relative interior of the $N$-simplex, i.e., in the set $\operatorname{int}(\Delta_N)\eqdef \{w\in (0,1)^N:\,\sum_{n=1}^N\,w_n=1\}$. 
We subsequently associate $f^{(2)}$ to a map taking values $\R^{N-1}$.  Finally, this latter map will be approximated by a neural network in step $3$.}

Fix $0<\varepsilon_3\le 1$, to be determined retroactively.  
Consider the $1$-Lipschitz homotopy $H:[0,1]\times \Delta_N\to \Delta_N$ given by $H(t,w)\eqdef t(w-\bar{\Delta}_N)+\bar{\Delta}_N$, where%
\footnote{I.e.\ {$\bar{\Delta}_N\eqdef (1/N,\dots,1/N)$} is the barycenter of the $N$-simplex $\Delta_N$.}~%
$\bar{\Delta}_N=(1/N,\dots,1/N)\in \Delta_N$ is the barycenter of the $N$-simplex.  Observe that, for each $t\in [0,1)$ we have $H(t,\Delta_N)\subset \operatorname{int}(\Delta_N)$ and for each $0\le \varepsilon_3\le \max_{w\in \Delta_N}\,\|w-\bar{\Delta}_N\|_2
{1-\frac{1}{N}}$
% {\color{red}wrong the max is $1-1/N$.} 
there exists a $t_{\varepsilon_3}\in [0,1)$ satisfying%~
\footnote{For the interested reader: we have just noted that the boundary of $\Delta_N$ is a $\mathcal{Z}$-set, in the sense of \citep[Section 5.1]{van2002infinite}.}
\begin{equation}
\label{eq:solving_for_teps}
        \max_{w\in \Delta_N}
    \,
        \|
            H(t_{\varepsilon_3},w)-w
        \|_2
    \le 
        \max_{w\in \Delta_N}
    \,
        \|
            H(t_{\varepsilon_3},w)-w
        \|_1
        \le \varepsilon_3
.
\end{equation}
The right-hand inequality can be solved explicitly for the largest value of $t_{\varepsilon_3}$ in $[0,1)$; this is because $w^{\star}\in \Delta_N$ given by $w^{\star}_1=1$ and $w^{\star}_j=0$ for $j=2,\dots,N$ is a non-unique maximizer of $\max_{w\in \Delta_N}
    \,
\|
    H(t_{\varepsilon_3},w)-w
\|_1$
.
% {\color{red}both norm has the same optimizer. These are convex functions that you are maximizing. the max is at the corners for both.}. 
We therefore compute
\allowdisplaybreaks
\begin{align}
\nonumber
        \|
                H(t_{\varepsilon_3},w)-w
        \|_1 
    = &
        \underbrace{
            |(t_{\varepsilon_3}(1-1/N)+1/N) - 1|
        }_{\text{$1^{rst}$ component}}
\\
    & +
        (N-1)
        \underbrace{
            |(t_{\varepsilon_3}(0-1/N)+1/N) - 0|
        }_{\text{$j>1^{rst}$ components}}
\\
\nonumber
    = &
        % \underbrace{
            |1-t_{\varepsilon_3}|\,|(1-1/N)|
        % }_{\mbox{$1^{rst}$ component}}
        +
        % \underbrace{
            (N-1)
            |1-t_{\varepsilon_3}|\,|1/N|
        % }_{\mbox{$j>1^{rst}$ components}}
\\
\nonumber
    = &
            |1-t_{\varepsilon_3}|\,(N-1)/N
        +
            |1-t_{\varepsilon_3}|\,(N-1)/N
\\
\label{eq:solve_for_teps2}
    = &
            (1-t_{\varepsilon_3})\,\frac{2(N-1)}{N}
    .
\end{align} 
If $\varepsilon_3>0$ is small enough,%
~\footnote{Namely, one needs that $0<\varepsilon_3 <\frac{2(N-1)}{N}$.}~%
then setting the right-hand side of~\eqref{eq:solve_for_teps2} equal to $\varepsilon_3$ and solving for $t_{\varepsilon_3}$ yields $t_{\varepsilon_3} 
    = 
        1-\frac{N\varepsilon_3}{2(N-1)}$.  For general values of $\varepsilon_3$, we may set
\begin{equation}
\label{eq:solve_for_teps2__solved}
        t_{\varepsilon_3} 
    = 
        1
        -
            \min\biggl\{
                \frac{N\varepsilon_3}{2(N-1)}
            ,
                1/2
            \biggr\}
.
\end{equation}

\noindent 
Consider the map
\begin{equation}
\label{eq:definition_rho_lambda}
\rho\eqdef \operatorname{softmax}_N\circ W:\R^{N-1}\to \operatorname{int}(\Delta_N)
\end{equation}
where $W:\R^{N-1}\ni x\to (x_1,\dots,x_{N-1},1)\in \R^N$.  A right-inverse of the smooth function $R:\operatorname{int}(\Delta_N)\to \R^{N-1}$ given for each $y\in \operatorname{int}(\Delta_N)$ by
\begin{align}
    \label{def:R}R(y)\eqdef 
    \big(
        (\ln(y_i)-\ln(y_N) + 1)
    \big)_{i=1}^N
.
\end{align}

Finally, we define the ``mollified simplicial target function'' $f^{(3)}\eqdef R\circ H(t_{\varepsilon_3},\cdot)\circ f^{(2)}:\R^d\to \R^{N-1}$.  

\noindent We therefore, have the following uniform estimate between $\rho\circ f^{(3)}$ and $f^{(2)}$
\allowdisplaybreaks
\begin{align}
\label{eq:estimate_f2f3__Begin}
    % &
    \max_{x\in \mathcal{E}_d(\kkk_0)}
    \,
        \|
                f^{(2)}(x)
            -
                \rho\circ f^{(3)}(x)
        \|_{\calH^2_T}
% \\
% \nonumber
   = &
% & 
        \max_{x\in \mathcal{E}_d(\kkk)}
        \,
        \|
                f^{(2)}(x)
            -
                \rho\circ \big(R\circ H(t_{\varepsilon_3},\cdot)\circ f^{(2)}(x)  \big)
        \|_{\calH^2_T}
\\
\nonumber
    = & 
        \max_{x\in \mathcal{E}_d(\kkk)}
        \,
        \|
                f^{(2)}(x)
            -
                H(t_{\varepsilon_3},\cdot)\circ f^{(2)}(x)\big)
        \|_{\calH^2_T}
\\
\label{eq:estimate_f2f3__SimplexExtremalEstimate}
    = & 
        \max_{w\in\Delta_{N}}
        \,
            \|
                    w
                -
                    H(t_{\varepsilon_3},w)
            \|_{\calH^2_T}
\\
\label{eq:estimate_f2f3__End}
    \le & 
        \varepsilon_3
,
\end{align}
where we have use the fact that $f^{(2)}$ takes values in the $N$-simplex in deducing~\eqref{eq:estimate_f2f3__SimplexExtremalEstimate}, and~\eqref{eq:estimate_f2f3__End} held by~\eqref{eq:solving_for_teps}.  Combining our estimate in~\eqref{eq:better_Estimate} with those in~\eqref{eq:estimate_f2f3__Begin}-\eqref{eq:estimate_f2f3__End} yields
\allowdisplaybreaks
\begin{align}
\label{eq:better_Estimate__Mollified__Begin}
    % &
    \sup_{x\in \mathcal{E}_d(\kkk_0)}\,
        \|f^{(1)}(x)-\eta_{N}\circ \rho\circ f^{(3)}(x)\|_{\calH^2_T}
% \\
% \nonumber
    \le & 
        \sup_{x\in \mathcal{E}_d(\kkk_0)}\,
            \|f^{(1)}(x)-\eta_{N}\circ f^{(2)}(x)\|_{\calH^2_T}
\\
\nonumber
    + &
        \sup_{x\in \mathcal{E}_d(\kkk_0)}\,
            \|\eta_{N}\circ f^{(2)}(x)-\eta_{N}\circ \rho\circ f^{(3)}(x)\|_{\calH^2_T}
\\
\nonumber
    \le & 
        \sup_{x\in \mathcal{E}_d(\kkk_0)}\,
            \|f^{(1)}(x)-\eta_{N}\circ f^{(2)}(x)\|_{\calH^2_T}
\\
\nonumber
    +&
        \operatorname{Lip}(\eta_{N})
        \,
        \sup_{x\in \mathcal{E}_d(\kkk_0)}\,
            \|f^{(2)}(x)-\rho\circ f^{(3)}(x)\|_{\calH^2_T}
\\
\nonumber
    \le & 
        \frac{\varepsilon_D}{3}
+
        \operatorname{Lip}(\eta_{N})
        \,
        \varepsilon_3
\\
\label{eq:better_Estimate__Mollified__end}
    \le & 
        \frac{\varepsilon_D}{3}
% \\
    +%&
        \big(
            (
            L
            \operatorname{diam}(\kkk_0) 
            ^{\alpha}
            +2\varepsilon_2)
            \,
            \sqrt{N}
        \big)
        \,
        \varepsilon_3
,
\end{align}
where $\operatorname{Lip}(\eta_{N})$ denotes the optimal Lipschitz constant of the map $\eta_{N}$ which we have bounded above by $L^{\eta}_{\varepsilon_1,\varepsilon_2}$, in~\eqref{eq:ub_OLC}. 
Combining the estimates in~\eqref{eq:packing_in_fK} and in~\eqref{eq:encoding_error} with those in~\eqref{eq:better_Estimate__Mollified__Begin}-\eqref{eq:better_Estimate__Mollified__end} yields
\allowdisplaybreaks
\begin{align}
\label{eq:final_estimate_BEGIN}
% &
    \sup_{u\in \kkk_0}\,
        \|
            F(u)
            -
            \eta_{N}\circ \rho\circ f^{(3)}\circ \mathcal{E}_d(u)
        \|_{\calH^2_T}
% \\
% \nonumber
& 
\le 
        \sup_{u\in \kkk_0}\,
            \|
                F(u)
                -
                F\circ\iota_d\circ \mathcal{E}_d(u)
            \|_{\calH^2_T}   
\\ 
\nonumber
&
    +
        \sup_{u\in \kkk_0}\,
            \|
                F\circ\iota_d\circ \mathcal{E}_d(u)
                -
                \eta_{N}\circ \rho\circ f^{(3)}\circ \mathcal{E}_d(u)
            \|_{\calH^2_T}  
\\
\nonumber
& 
\le 
        L
        \sup_{u\in \kkk_0}
        \,
            \|
                u
                -
                \iota_d\circ \mathcal{E}_d(u)
            \|_{\calH^2_T}   
\\ 
\nonumber
&
    +
        \sup_{u\in \kkk_0}\,
            \|
                F\circ\iota_d\circ \mathcal{E}_d(u)
                -
                \eta_{N}\circ \rho\circ f^{(3)}\circ \mathcal{E}_d(u)
            \|_{\calH^2_T}  
\\
\nonumber
& 
=
        L
        \sup_{u\in \kkk_0}
        \,
            \|
                u
                -
                \iota_d\circ \mathcal{E}_d(u)
            \|_{\calH^2_T}   
\\ 
\nonumber
&
    +
        \sup_{u\in \mathcal{E}_d(\kkk_0)}\,
            \|
                f^{(1)}(x)
                -
                \eta_{N}\circ \rho\circ f^{(3)}(x)
            \|_{\calH^2_T}  
\\
\nonumber
& 
\le
        L
        \varepsilon_0
    +
        \sup_{u\in \mathcal{E}_d(\kkk_0)}\,
            \|
                f^{(1)}(x)
                -
                \eta_{N}\circ \rho\circ f^{(3)}(x)
            \|_{\calH^2_T}  
\\
\label{eq:final_estimate_END}
& 
\le
        L
        \varepsilon_0 
        +
        \frac{\varepsilon_D}{3}
    +
        \big(
            (L\operatorname{diam}(\kkk_0)
            ^{\alpha}
            +2\varepsilon_2)
            \,
            \sqrt{N}
        \big)
        \,
        \varepsilon_3
.
\end{align}
Retroactively setting
\begin{equation}
\label{eq:values_eps0_eps3}
        \varepsilon_0 
    \eqdef 
        \frac{\varepsilon_D}{3L}
\mbox{ and }
        \varepsilon_3
    \eqdef 
        \frac{
            \varepsilon_D            
        }{
            \min\big\{1,
                3
                (
                L
                \operatorname{diam}(\kkk_0)
                ^{{\alpha}}
                )\sqrt{N}
            \big\}
        }
,
\end{equation}
implying that $\varepsilon_0\in \mathcal{O}(\varepsilon_D)$ and $\varepsilon_3\in \mathcal{O}(\varepsilon_D/\sqrt{N})$.  Consequentially, 
\begin{align}
\label{eq:final_estimate__dimensionreduced}
    \sup_{u\in \kkk_0}\,
        \|
            F(u)
            -
            \eta_{N}\circ \rho\circ f^{(3)}\circ \mathcal{E}_d(u)
        \|_{\calH^2_T}
\le&
    \,
    \varepsilon_D
.
\end{align}

Next, we will obtain upper-bound the best 
% Lipschitz
{$\alpha$-H\"{o}lder}
constant of $f^{(3)}$ to obtain quantitative parameter estimates on our MLP, which will approximate $f^{(3)}$.

\noindent\textbf{Step 4 - Computing the Regularity of the Surrogate Target Function}
To, apply a \textit{quantitative} universal approximation theorem, we need a handle on the regularity of the target function being approximated.  
We first observe that, on the set $\Delta_N^{\epsilon_2}\eqdef H(t_{\epsilon_2},\Delta_N)$ the map $R$, defined in \eqref{def:R}, is $L_{f^{(3)}\epsilon_2,N}$-Lipschitz with constant given by 
% {\color{red}problem with the first inequality, how do you bound the L2 norm with L1 norm? Also $R$ is a vector valued map, so this is the Jacobian?}
\allowdisplaybreaks
\begin{align}
\label{eq:Lip_constant_Rinv__begin}
        L_{f^{(3)}\varepsilon_2,N}
    = 
% &
        \sup_{w\in \Delta_N^{\varepsilon_3}}
        \,
            \|\nabla R(w)\|_2
% \\
% \nonumber
    \le 
&
        \sup_{w\in \Delta_N^{\varepsilon_3}}
        \,
            \|\nabla R(w)\|_1
\\
\label{eq:TB_Improved_at_end}
    = &
        \sup_{w\in \Delta_N}
        \sum_{i=1}^N\,
        \frac1{
            % \lambda
            % \,
            |t_{\varepsilon_3}(w_i-1/N) + 1/N|
        }
\\
\nonumber
    \le & \sum_{i=1}^N \, \frac1{
        \min_{w\in \Delta_N}\, \lambda | t_{\varepsilon_3}(w_i-1/N) + 1/N|
    }
\\
\nonumber
    = & 
        \frac{N^2}{
        % \lambda 
        (1-t_{\varepsilon_3})
        }
\\
\nonumber
    = & 
        % \frac{
            N
        % }
        % {
        %     \lambda
        % }
        \,
        \max\big\{
                2N
            ,
                \frac{
                    2(N-1)
                }{
                    \varepsilon_3
                }
        \big\}
\\
\label{eq:Lip_constant_Rinv}
    \le & 
        \frac{
            2
            \,
            N^2
        }{
            % \lambda
            % \,
            \min\{\varepsilon_3^{-1},1/2\}
        }
% \\
    \eqdef 
    % &
        \tilde{L}_{f^{(3)}\varepsilon_3,N}
\end{align}
{where $\nabla R$ denotes the Jacobian of $R$ and}
where the inequality~\eqref{eq:Lip_constant_Rinv__begin} holds by the Rademacher-Stephanov theorem, see e.g.~\citep[Theorems 3.1.6-3.1.9]{federer2014geometric}.  
% \AK{Check if the extra $2$ is propagated in~\eqref{eq:Lip_constant_Rinv}}[.]

We thus conclude that, from the estimates in~\eqref{eq:Lipschitz_projected_version_of_f1}, \eqref{eq:Lip_constant_Rinv__begin}-\eqref{eq:Lip_constant_Rinv}, and the observation that $H$ is $1$-Lipschitz, that $f^{(3)}$ is $L_{f^{(3)}}$-Lipschitz; where
\begin{equation}
\label{eq:Lip_const_target_function}
\begin{aligned}
        L_{f^{(3)}}
    \le %&
        L_{f^{(2)}}
        % \,
        % 1
        \,
        L_{f^{(3)}\epsilon_2,N}
    & = 
        \frac{\bar{c}N^3}{
        % \lambda\,
        \varepsilon_1
        \,
        \min\{\min\{1,\epsilon_3\},1/2\}
        }
    \\
    & = 
        \frac{9\bar{c}d
            \,
            N^2
        }{
            % \lambda\,
            \varepsilon_D
        \,
        \min\{\varepsilon_D/3,1/2\}
        }
    \eqdef 
        \tilde{L}
,
\end{aligned}
\end{equation}
where $\bar{c}\eqdef 4\tilde{c}>0$. We now construct our deep-learning approximation.  
% Set $\lambda\eqdef 1$.

\textbf{Step 5 - Neural Approximation of Surrogate Target Function $\hat{f}^{(3)}$:}
\hfill\\
We consider two cases; in the former, the activation function is smooth and in the latter, it is the trainable super-expressive activation (defined in~\eqref{eq:supreexpressive}).
\begin{enumerate}
    \item \textbf{Case 1 - $\sigma$ as in Example~\ref{ex:Activation_Standard}:} 
    By~\citep[Proposition 53]{AKLeonie_2022_JMLR}, there is a MLP $\hat{f}:\R^d\to\R^N$ with activation function 
$\sigma_0 \in C(\mathbb{R})$ (in the notation of Example~\ref{ex:Activation_Standard})
satisfying
\begin{equation}
\label{eq:ReLU}
    \sup_{x\in \mathcal{E}_d(\kkk_0)}\,
        \|
                \hat{f}(x)
            -
                f^{(3)}(x)
        \|_2
    <
        \bar{\varepsilon}_A
\end{equation}
%%%
with depth and width given by:
\begin{itemize}
    \item \textbf{Width:}
        $d+N+2$
    \item \textbf{Depth:} Finite, and if $\sigma$ is non-affine and smooth then: 
            $% \[
            \mathcal{O}\Big(
                N ((1-d/4)N)^{2d/\alpha}
                \,
                (2C)^{2d}
                \,
                \varepsilon^{-2d/\alpha}
            \Big)
            $
\end{itemize}
where we use the fact that the diameter of $\kkk_0$ is at-most $2C$.

\item If $\sigma$ is the trainable super-expressive activation function in~\eqref{eq:supreexpressive}, then: for $i=1,\dots,N$ \citep[Theorem 1]{gao2022achieving} there exists an MLP $\hat{f}_i:\R^d\to\R$ with activation function $\sigma_0$ satisfying
\begin{equation}
\label{eq:superexpressiveapprox}
    \max_{i=1,\dots,d}\,\sup_{x\in \mathcal{E}_d(\kkk_0)}
    \,\|
            \langle f^{(3)}(x),e_i\rangle 
        - 
            \hat{f}_i(x)
    \|_2
    <
        \bar{\varepsilon}_A/N
\end{equation}
where $\{e_i\}_{i=1}^N$ is the standard orthonormal basis of $\R^N$.   
Moreover, by \citep[Theorem 1]{gao2022achieving}, and the remark directly after, the width, depth, and number of non-zero parameters determining each network is exactly
\begin{itemize}
    \item \textbf{Width}: $11$,
    \item \textbf{Depth:} $36\,(2d+1)$,
    \item \textbf{No.\ Params:} $5437\,(d+1)\,(2d+1)$.
\end{itemize}

Since $\sigma_1(x)=x$, for all $x\in \R$, then the trainable activation function $\sigma$ has the $1$-identity requirement (see~\citep[Definition 4]{FlorianHighDimensional2021}) applies \citep[Proposition 5]{FlorianHighDimensional2021} from which we conclude that there exists an MLP $\hat{f}:\R^d\to\R^N$ with activation function $\sigma$ satisfying: for each $x\in \R^d$
\[
    \hat{f}(x)
    =
    \sum_{i=1}^N\,
        f_i(x)\,e_i
,
\]
furthermore, the with, depth, and number of non-zero determining $\hat{f}$ are
\begin{itemize}
    \item \textbf{Width}: $12\,N \in \mathcal{O}(N)$,
    \item \textbf{Depth:} $d(N-1) + 36\,(2d+1) \in 
    \mathcal{O}\big(
            dN
    \big)
    $,
    \item \textbf{No.\ Params:} at-most $
        3738
        \,N^2(d^2-1)
        N
        (d+1)\,(2d+1)
        \in \mathcal{O}\big(
        N^3\,d^4
        \big)
    $.
\end{itemize}
Consequentially,~\eqref{eq:superexpressiveapprox} implies that 
% {\color{red}what is $e_i$ below}
\begin{equation}
\label{eq:superexpressiveapprox_done}
\begin{aligned}
        \sup_{x\in \mathcal{E}_d(\kkk_0)}
        \,\|
                f^{(3)}(x)
            - 
                \hat{f}(x)
        \|_2
    \le 
% & 
    \sum_{i=1}^N\,
    \sup_{x\in \mathcal{E}_d(\kkk_0)}
        \,\|
                \langle f^{(3)}(x),e_i\rangle 
            - 
                \hat{f}_i(x)
        \|_2
% \\
        < 
        % &
            N
            \frac{\bar{\varepsilon}_A}{N}
% \\
        = 
        \bar{\varepsilon}_A
\end{aligned}
\end{equation}
\end{enumerate}
{where $e_i$ is the $i^{th}$ standard basis vector in $\mathbb{R}^N$ with $1$ in the $i^{th}$ coordinate and $0$ otherwise.}
We are now ready to complete the proof by combining the estimates from the previous steps.

\textbf{Step 6 - Putting it All Together:}\hfill\\
Set $\hat{F}\eqdef \eta_{N_{\varepsilon_1}}\circ \rho\circ \hat{f}\circ \mathcal{E}_d:\mathcal{H}_T^2\to \mathcal{H}_T^2$.  The estimates in~\eqref{eq:final_estimate__dimensionreduced} with those in~\eqref{eq:ReLU} (resp.~\eqref{eq:superexpressiveapprox_done}) yield
\allowdisplaybreaks
\begin{align}
% \label{eq:final_estimate__dimensionreduced}
\nonumber
    \sup_{u\in \kkk_0}\,
        \|
            F(u)
            -
            \hat{F}(u)
        \|_{\calH^2_T}
% \\
\le &
    \sup_{u\in \kkk_0}\,
        \|
            F(u)
            -
            \eta_{N_{\varepsilon_1}}\circ \rho\circ f^{(3)}\circ \mathcal{E}_d(u)
        \|_{\calH^2_T}
\\ 
\nonumber
& +
    \sup_{u\in \kkk_0}\,
        \|
            \eta_{N_{\varepsilon_1}}\circ \rho\circ f^{(3)}\circ \mathcal{E}_d(u)
            -
            \hat{F}(u)
        \|_{\calH^2_T}
\\
\nonumber
\le &
    \varepsilon_D
+
    \sup_{u\in \kkk_0}\,
        \|
            \eta_{N_{\varepsilon_1}}\circ \rho\circ f^{(3)}\circ \mathcal{E}_d(u)
            -
            \hat{F}(u)
        \|_{\calH^2_T}
\\
\nonumber
= &
    \varepsilon_D
+
    \sup_{x\in \mathcal{E}_d(\kkk_0)}\,
        \|
            \eta_{N_{\varepsilon_1}}\circ \rho_{1}\circ \hat{f}(x)
            -
            \eta_{N_{\varepsilon_1}}\circ \rho_{1}\circ \hat{f}(x)
        \|_{\calH^2_T}
\\
\nonumber
= &
    \varepsilon_D
+
    \operatorname{Lip}(\eta_{N_{\varepsilon_1}}\circ \rho)
    \,
    \sup_{x\in \mathcal{E}_d(\kkk_0)}\,
        \|
            f^{(3)}(x)
            -
            \hat{f}(x)
        \|_{2}
\\
\label{eq:final_bound}
= &
    \varepsilon_D
+
    \operatorname{Lip}(\eta_{N_{\varepsilon_1}}\circ \rho)
    \,
    \bar{\varepsilon}_A
.
\end{align}
Since $W$ is an isometric embedding and the $\operatorname{softmax}$ function is at-most $1$-Lipschitz then $\rho$ is at-most $1$-Lipschitz and $\operatorname{Lip}(\eta_{N_{\varepsilon_1}}\circ \rho)=\operatorname{Lip}(\eta_{N_{\varepsilon_1}})$, which by~\eqref{eq:ub_OLC} is at-most
\[
        \operatorname{Lip}(\eta_{N_{\varepsilon_1}})
    \le 
        (L
        \operatorname{diam}(\kkk_0) 
        {^{\alpha}}
        +2\epsilon_D/9)
        \sqrt{N}
.
\]
Since this upper-bound on $\operatorname{Lip}(\eta_{N_{\varepsilon_1}}\circ \rho))$ depends only on $\varepsilon_D$ and is independent of $\bar{\varepsilon}_A$.  Consequentially for the right-hand side of~\eqref{eq:final_bound} can be made arbitrarily small by choosing $\varepsilon_D$ and $\bar{\varepsilon}_A$ large enough.

It remains to bound $N$ explicitly.  
There are three cases which we consider here, each of which corresponds to the respective assumptions made on $\kkk_0$ and its relationship to the target (non-linear) operator $f$:
\begin{enumerate}
    \item \textbf{Exponentially Ellipsoidal:} 
    Suppose that $\kkk_0\subset\mathcal{H}_T^2$ is such that: for each $\kkk_0\ni x=\sum_{i=1}^{\infty}\,\beta_i s_i$ we have that $|\beta_i|\le C\, r^i$.  
    Recall that, e.g.\ as noted on~\citep[in Remark 1]{DumerPinskerPrelov_CoveringEucSpace_2004_IEEEInfTrans}, that the $\epsilon_A$-covering number of $p_d(\varepsilon_A^{-1}\cdot \kkk_0)$ (where $\delta \kkk_0$ denotes the $\delta$-thickening of $\kkk_0$ in $\mathbb{R}^d$).  Therefore, for each such $x\in \kkk_0$ we have that
    \[
            \sum_{i=1}^d\, 
                \frac{|\beta_i|^2}{\theta_i^2} 
        \le 
            1
    \]
    where the scaling constants $(\theta_i)_{i=1}^{\infty}$ (independent of $d$) are given by
    \begin{equation}
    \label{eq:covering_lemma}
            \theta_i 
        =
            \biggl(
                r
                \,
                \Big(
                    1 + 
                    \Big(
                        \frac{C}{\varepsilon_A}
                    \Big)^{2}
                \Big)^{1/2}
            \biggr)^i
    .
    \end{equation}
    Therefore, \citep[Theorem 2]{DumerPinskerPrelov_CoveringEucSpace_2004_IEEEInfTrans}, with the description of $o(1)$ given in its proof on~\citep[Equation (40)]{DuchiHazanSinger_2011_JMLR_adaptiveSGOnlineOptim}, by~\eqref{eq:covering_lemma} we find that
    \begin{equation}
    \label{eq:covering_lemma__ellipse}
    \begin{aligned}
            N 
        &\le 
            \exp\big(
                \sum_{i=1}^d\,
                    i\, \log\big(
                        Cr
                        \,
                        (
                            C^{-2} + \varepsilon_A^{-2}
                        )^{1/2}
                    \big)
            \big)
\\
    & =
        \exp\big(
                \frac{d(d+1)}{2}
                \, 
                \log\big(
                    Cr
                    \,
                    (
                        C^{-2} + \varepsilon_A^{-2}
                    )^{1/2}
                \big)
        \big)
\\
    & = 
                    \big(
                        r
                        \,
                        (
                            1 + 
                            (C/\varepsilon_A)^{2}
                        )^{1/2}
                    \big)^{\frac{d(d+1)}{2}}
    \end{aligned}
    \end{equation}
    Since, in this case, $d\in \mathcal{O}(\ln(\epsilon_D^{-1/r}))$ then there exists some $C_1>0$ such that~\eqref{eq:covering_lemma__ellipse} reduces to
    \begin{equation}
    \label{eq:covering_lemma__ellipse___completed}
    \begin{aligned}
            N 
        &\le
            % \mathcal{O}\Big(
            \big(
                r
                \,
                (
                    1 + 
                    (C/\varepsilon_A)^{2}
                )^{1/2}
            \big)^{
            C_1
            \ln(\epsilon_D^{-1/r})^2
            }
        % \Big)
    \end{aligned}
    \end{equation}
    \item \textbf{Exponential Manifold:} Suppose that $\kkk_0$ satisfies Definition~\ref{defn:exp_manifold}. For every $\tilde{\varepsilon_A}>0$ (to be fixed momentarily), \cite[Proposition 15.1.3]{lorentz1996constructive}, implies that $\tilde{\varepsilon_A}$-covering number $\tilde{N}$ of the Euclidean unit ball $B_{d}\eqdef\{x\in \mathbb{R}^d:\, \|x\|\le 1\}$ is bounded above and below by
        \begin{equation}
        \label{eq:tight_covering_bounds}
                2^{-d}\,\big(\sqrt{d}/\tilde{\varepsilon_A})^d
            \leq 
                \tilde{N}
            \leq
                3^d
                \,\big(\sqrt{d}/\tilde{\varepsilon_A} \big)^d
        .
        \end{equation}
    Since the ``latent parameterization'' map $\pi:\mathbb{R}^d\to \mathcal{H}_T^2$ was assumed to be $1$-Lipschitz and maps onto $\kkk_0$ then, the image of every $\tilde{\varepsilon_A}$ of $B_d$ under $\pi$ must be $1\cdot \tilde{\varepsilon_A}$ covering of $\kkk_0$. Set $\tilde{\varepsilon_A}=\varepsilon_A$.  Then,~\eqref{eq:tight_covering_bounds} implies that
    \[
            N
        \le 
            \big(3\sqrt{d}/\tilde{\varepsilon_A} \big)^d
        =
            \big(\varepsilon_A^{-1}\,3(c(\ln(\varepsilon_D^{-1/r}))^{1/2} \big)^{c(\ln(\varepsilon_D^{-1/r})}
    \]
    where we have used the fact that $\kkk_0$ is contained in an $(r,f)$-exponentially ellipsoidal set to deduce that $d\le c\ln(\varepsilon_D^{-1/r})$ for some absolute constant $c>0$.
    \item \textbf{General Case:} In the case of general $\kkk_0$, \citep[Lemma 7.1]{acciaio2023designing} and the upper-bound of $2^{d+1}$ on the doubling constant of $\mathcal{E}_d(\kkk_0)$, just prior to Equation~\eqref{eq:Lipschitz_projected_version_of_f1}, implies that
    \[
    N\le \big(2^{(d+1)}\big)^{\log_2(\operatorname{diam}(\kkk_0)) - 
    \frac1{\alpha}\,
    \log_2(\varepsilon_D/\tilde{L}) 
    + 
    \frac1{\alpha}\,
    \log_2(\tilde{c}d)}
    \]
    for some absolute constant $\tilde{c}>0$.  
\end{enumerate}
Setting $
\varepsilon_A
\eqdef 
\bar{\varepsilon}_A
/\big(
    c
        (L\operatorname{diam}(\kkk_0)
        {^{\alpha}}
        +\varepsilon_D)
        \sqrt{N
        % _{\varepsilon_D}
        }
    \big)
\in 
\mathcal{O}\Big(
    \frac{\bar{\varepsilon}_A}{\epsilon_D \sqrt{N
    % _{\varepsilon_D}
    }}
\Big)
$ yields the conclusion.

\textbf{Step 7 - Elucidating the Model}

Let $V\in \mathbb{R}^{N\times Q}$ be such that, for $n=1,\dots,N$ and $i=1,\dots,Q$, $V_{n,q}\eqdef \langle F\circ \iota_d(x_n),s_i \rangle_{\mathcal{H}_T^2}$.  Then,~\eqref{eq:definition_eta} implies that: for each $w\in \Delta_N$
\begin{equation}
\label{eq:elucidating_eta__1}
\eta(w) = \sum_{n=1}^N\, \big(
\sum_{i=1}^Q\,
\langle F\circ \iota_d(x_n),s_i \rangle_{\mathcal{H}_T^2}
\big)
= 
\sum_{n=1}^N
\sum_{i=1}^Q\,
\, w_n \,V_{n,q}
\end{equation}
For either $\sigma$ is smooth~
% $=\operatorname{ReLU}$
or $\sigma$ as in~\eqref{eq:superexpressiveapprox}, consider the MLP with $\sigma$ activation function $\mathcal{V}:\mathbb{R}^d\to \mathbb{R}^{N\times Q}$ given for each $x\in \mathbb{R}^d$ by
\[
        \mathcal{V}(x)
    \eqdef 
        \mathbf{0}^{(1)}
        \sigma\bullet
        \big(
            \mathbf{0}^{(2)}
            x
            +
            \mathbf{0}^{(3)}
        \big)
        +
        V
\]
where $
\mathbf{0}^{(1)}
$ is the $ND\times 1$ zero matrix, $\mathbf{0}^{(2)}$ is the $1\times d$ zero matrix, and $\mathbf{0}^{(3)}=(0)\in \mathbb{R}$, and where we have identified $\mathbb{R}^{N\times D}$ with $\mathbb{R}^{ND}$.  
In either case, observe that the number of non-zero parameters defining $\mathcal{V}$ are at-most $ND$.

For each $n=1,\dots,N$ and $i=1,\dots,Q$, we $V^{(n,i)}\eqdef s_i$.
By construction: for each $u\in \mathcal{H}_T^2$ and every $w\in \Delta_N$ we have that
\begin{equation}
\label{eq:elucidating_eta__2}
\mathcal{D}(w,u)\eqdef 
\sum_{n=1}^N\, w_n \,[\mathcal{V}\circ \mathcal{E}_d(u)]_n\, V^{(n,q)}
=
\eta_N(w)
.
\end{equation}
Since the map $W$ in the definition of $\rho$, see the line just below~\eqref{eq:definition_rho_lambda}, was affine then our approximation $\hat{F}= \eta_{N_{\varepsilon_1}}\circ \rho\circ \hat{f}\circ \mathcal{E}_d:\mathcal{H}_T^2\to \mathcal{H}_T^2$ is of the form in Definition~\ref{defn:NO}.

The number of non-zero parameters defining the model $\hat{F}$ are, therefore, at-most
\begin{equation}
\label{eq:N_PAR}
        \underbrace{
            \operatorname{Depth}(\hat{f})
            \,
            \operatorname{Width}(\hat{f})^2
        }_{\text{No. Param. $\hat{f}$}}
    +
        \underbrace{
            NQ
        }_{\text{No. Param. $\mathcal{V}$}}
\end{equation}
where the depth and width of $\hat{f}$ were computed in step $4$.  
In particular, if $\sigma$ is the trainable super-expressive activation function in~\eqref{eq:superexpressiveapprox} then the quantity in~\eqref{eq:N_PAR} is $\mathcal{O}\big(
    % N^3\,d^4
    % +
    % NQ
    N(N^2d^4 + Q)
    \big)
$.
\end{proof}

\begin{proof}[{Proof of Theorem~\ref{thrm:UniversalApprox}}]
Fix a non-empty compact subset $\kkk_0 \subset \mathcal{H}_T^2$, a continuous function $f: \kkk_0 \to \mathcal{H}_T^2$, and a $\varepsilon>0$.

\textit{Let $\kkk_0\subset \mathcal{H}_T^2$ be non-empty and compact.  We would like to use~\citep[Theorem 1]{Miculescu_RAE_2002__ApproxLipFunctions} to reduce the problem of approximating $f$ to $\varepsilon$ precision, to the problem of approximating an $\varepsilon/2$ Lipschitz approximation of our target function $f$ to $\varepsilon/2$ precision.  Thus, we will be able to employ our technical approximation theorem for Lipschitz maps between $\mathcal{H}_T^2$ to itself, to deduce our conclusion.
On a technical note, we do not argue on the domain $\mathcal{H}_T^2$ but rather on the compact subspace $\kkk_0$, since all continuous functions are both bounded and uniformly continuous thereon; which then directly allows us to~\citep[Theorem 1]{Miculescu_RAE_2002__ApproxLipFunctions} which only allows us to uniformly approximate bounded continuous functions on compacta.}

\noindent \textbf{Step 1 - Verification of Lipschitz Extension Property}
\hfill\\
In order to apply~\citep[Theorem 1]{Miculescu_RAE_2002__ApproxLipFunctions} we will need to show that the pair $(\kkk_0,\|\cdot\|_{\mathcal{H}_T^2})$ and $\mathcal{H}_T^2$ have the so-called ``Lipschitz extension property'' (as named in~\citep[Theorem 1]{Miculescu_RAE_2002__ApproxLipFunctions}).  
This means that the Lipschitz function from $(B,\|\cdot\|_{\mathcal{H}_T^2})$ to $\mathcal{H}_T^2$, for any subset $B$ of $\kkk_0$, can be extended to a Lipschitz function of all of $\kkk_0$ with roughly the same Lipschitz constant.

Let $B\subseteq \kkk_0$.  Note that, as $\kkk_0 \subset \mathcal{H}_T^2$ then, $B$ is a subset of $\mathcal{H}_T^2$.  
Since $\mathcal{H}_T^2$ is a separable Hilbert space then the extension theorem of~\cite[Theorem 1.12]{BenyaminiLindenstrauss_2000_NonlinearFunctionalAnalysis}: for every $L\ge 0$ and each $L$-Lipschitz (non-linear operator) $g:(B,\|\cdot\|_{\mathcal{H}_T^2})\to \mathcal{H}_T^2$ there exists an $L$-Lipschitz extension $\tilde{G}:\mathcal{H}_T^2\to\mathcal{H}_T^2$; i.e.\ $\tilde{G}$ is $L$-Lipschitz
\begin{equation}
\label{eq:intermediateLipschitzextension_V1}
        \tilde{G}|_B
    =
        g
.
\end{equation}
Since the restriction operator $\iota_{\kkk_0}:\mathcal{H}_T^2\ni \tilde{g}\to \tilde{g}|_K\in \kkk_0$ is $1$-Lipschitz then the composite map $G\eqdef \iota_{\kkk_0} \circ \tilde{G}=\tilde{G}|_{\kkk_0} :( \kkk_0,\|\cdot\|_{\mathcal{H}_T^2})\to \mathcal{H}_T^2$ is $L$-Lipschitz.  Furthermore,~\eqref{eq:intermediateLipschitzextension_V1} and the inclusion of $B$ in $\kkk_0$ imply that
\begin{equation}
\label{eq:intermediateLipschitzextension}
        G|_B
    =
        (\tilde{G}|_B)|_{\kkk_0}
    =
        g
.
\end{equation}
Thus, $G$ is an $L$-Lipschitz extension of $g$ to $(\kkk_0,\|\cdot\|_{\mathcal{H}_T^2})$.  Thus, the pair $(\kkk_0,\|\cdot\|_{\mathcal{H}_T^2})$ and $\mathcal{H}_T^2$ has the Lipschitz extension property; thus~\citep[Theorem 1]{Miculescu_RAE_2002__ApproxLipFunctions} implies that the space of Lipschitz functions from $(\kkk_0,\|\cdot\|_{\mathcal{H}_T^2})$ to $\mathcal{H}_T^2$ is \textit{dense} in the space of uniformly continuous and bounded functions from $(\kkk_0,\|\cdot\|_{\mathcal{H}_T^2})$ to $\mathcal{H}_T^2$ with respect to the uniform norm.

\noindent \textbf{Step 2 - $\varepsilon/2$-Approximation of $f$ by Lipschitz Maps}
\hfill\\
Since $\kkk_0$ is compact and $f$ is continuous on $\kkk_0$ then $f$ is uniformly continuous and bounded thereon. By~\citep[Theorem 1]{Miculescu_RAE_2002__ApproxLipFunctions}, we deduce that there exists a Lipschitz function $\tilde{f}_{\varepsilon}:(\kkk_0,\|\cdot\|_{\mathcal{H}_T^2})\to \mathcal{H}_T^2$ satisfying
\begin{equation}
\label{eq:near_apprx_lip}
        \max_{u\in \kkk_0}\, 
            \|f(u)-\tilde{f}_{\varepsilon}(u)\|_{\mathcal{H}_T^2} 
    \le 
        \varepsilon/2
.
\end{equation}
Again applying~\cite[Theorem 1.12]{BenyaminiLindenstrauss_2000_NonlinearFunctionalAnalysis}, we deduce that $\tilde{f}_{\varepsilon}$ admits a Lipschitz extension $f_{\varepsilon}:\mathcal{H}_T^2\to \mathcal{H}_T^2$, with the same Lipschitz constant. Since $f_{\varepsilon}$ is a Lipschitz extension of $\tilde{f}_{\varepsilon}$, beyond $\kkk_0$, then~\eqref{eq:near_apprx_lip} implies that
\begin{equation}
\label{eq:near_apprx}
        \max_{u\in \kkk_0}\, 
            \|f(u)-f_{\varepsilon}(u)\|_{\mathcal{H}_T^2} 
    =
        \max_{u\in \kkk_0}\, 
            \|f(u)-\tilde{f}_{\varepsilon}(u)\|_{\mathcal{H}_T^2} 
    \le 
        \varepsilon/2
.
\end{equation}
\noindent \textbf{Step 3 - $\varepsilon/2$-Approximation of $f_{\varepsilon}$ by Attentional Neural Operator}
\hfill\\
Since $\kkk_0$ is a compact subset of $\mathcal{H}_T^2$ and $f_{\varepsilon}:\mathcal{H}_T^2\to \mathcal{H}_T^2$ is Lipschitz then Lemma~\ref{lem:UAT_1} applies.  Whence, there exists an attentional neural operator $\hat{F}:\calH_T^2\to\calH_T^2$ satisfying
\begin{equation}
\label{eq:near_approx_attentional}
    \max_{u\in \kkk_0}\,
        \|
            f_{\varepsilon}(u)
            -
            \hat{F}(u)
        \|_{\calH^2_T}
.
\end{equation}
Combining~\eqref{eq:near_apprx} and~\eqref{eq:near_approx_attentional} yield
\begin{equation*}
        \max_{u\in \kkk_0}\, 
            \|f(u)-\hat{F}(u)\|_{\mathcal{H}_T^2} 
    \le
            \max_{u\in \kkk_0}\, 
                \|f(u)-f_{\varepsilon}(u)\|_{\mathcal{H}_T^2} 
        +
            \max_{u\in \kkk_0}\,
                \|
                    f_{\varepsilon}(u)
                    -
                    \hat{F}(u)
                \|_{\calH^2_T}
    \le 
        \varepsilon/2 + \varepsilon/2 = \varepsilon , 
\end{equation*}
which concludes our proof.
\end{proof}

\subsection{Proof of \texorpdfstring{Main Stackelberg Equilibria Results}{Theorems~\ref{thrm:Main__BestResponse} and~\ref{thrm:Objective}}}
\label{s:Proofs_MainStackelbergApproximations}

Theorems~\ref{thrm:Main__BestResponse} and~\ref{thrm:Objective} are two parts of a larger whole.  
As such, their derivation is most naturally merged into a single proof; we now do.

\begin{proof}[{Joint Proof of Theorems~\ref{thrm:Main__BestResponse},~\ref{thrm:Objective}, and~\ref{thrm:Main__BestResponse___goodrates}}]
%%%
%%%
%%%
Let $\hat{U}:\mathcal{H}_T^2\to \mathcal{H}_T^2$ be a map, to be fixed retroactively.
\hfill\\
For any $d\in \mathbb{N}_+$, let $p_d:\mathcal{H}_T^2\to \operatorname{span}\{s_i\}_{i=1}^d$ be the orthogonal projection; i.e.\ $p_d\big(
\sum_{i=1}^{\infty}\,\beta_is_i
\big)= \sum_{i=1}^{d}\,\beta_is_i$ for all $u=\sum_{i=1}^{\infty}\,\beta_i\,s_i\in \mathcal{H}_T^2$.  
We will retroactively adjust $d$.
% \hfill\\
Fix $u^0\in \kkk_0$, denote $\hat{u}^0_d\eqdef p_d(u^0)$ and compute
\allowdisplaybreaks
\begin{align} 
\label{eq:tb_controlled}
        \big|
                J_{0}(u^0,U^{\star}(u^0))
            -
                J_{0}\big(
                    \hat{u}^0_d
                ,
                    \hat{U}(
                    \hat{u}^0_d
                    )
                \big)
        \big|
    \le &
        \underbrace{
            \big|
                    J_{0}(u^0,U^{\star}(u^0))
                -
                    J_{0}\big(
                        \hat{u}^0_d
                    ,
                        U^{\star}(
                        \hat{u}^0_d
                        )
                    \big)
            \big|
        }_{\term{t:input_truncation}}
\\
\nonumber
    +
    &
    \underbrace{
        \big|
                J_{0}\big(
                    \hat{u}^0_d
                ,
                    U^{\star}(
                    \hat{u}^0_d
                    )
                \big)
            -
                J_{0}\big(
                    \hat{u}^0_d
                ,
                    \hat{U}(
                    \hat{u}^0_d
                    )
                \big)
        \big|
    }_{\term{t:operator_approximation}}
.
\end{align} 
\paragraph{{Step 1 - Bounding Term~\Cref{t:input_truncation}}}
By Lemma~\ref{lm:stabilityJ0}
, we can bound~\Cref{t:input_truncation} from above by
\begin{align}
\label{eq:input_truncation__bound_1}
        \eqref{t:input_truncation}
    = &
        \big|
                J_{0}(u^0,U^{\star}(u^0))
            -
                J_{0}\big(
                    \hat{u}^0_d
                ,
                    U^{\star}(
                    \hat{u}^0_d
                    )
                \big)
        \big|
    \le 
    % &
      % KL_h\, 
      \widetilde{\omega}\big(
            \|u^0-\hat{u}^0_d\|_{\mathcal{H}_T^2}
      \big)
\end{align}
where $\widetilde{\omega}(t)=C\max\{|t|,|t|^{1/2}\}$ for each $t\in \mathbb{R}$ and for some constant $C\ge 0$ depending only on $T$.
Note that, $t\mapsto \widetilde{\omega}\big(t\big)$ is continuous, monotonically increasing on $[0,\infty)$ and subjective, thus it is a homomorphism of $[0,\infty)$ to itself with continuous inverse given by
\[
    \bar{\omega}(t)\eqdef 
    \begin{cases}
        (t/C)^2 &\mbox{ if } 0\le t\le 1 \\
        (t/C) &\mbox{ if } 1\le t
.
    \end{cases}
\]
We emphasize that, $\bar{\omega}$ has range $[0,\infty)$.

Since $\mathcal{H}_T^2$ has the $1$-bounded approximation property implemented by the (finite-rank) projection operators $(p_d)_{d\in \mathbb{N}_+}$ then choosing $d$ large enough, we may ensure that $\sup_{u^0\in \kkk_0}\,  \|u^0-\hat{u}^0_d\|_{\mathcal{H}_T^2}<
\bar{\omega}(\varepsilon/2)$.  
The right-hand side of~\eqref{eq:input_truncation__bound_1} can be bounded-above as follows
\begin{align}
\label{eq:input_truncation__bound_2}
        \eqref{t:input_truncation}
    \le 
      % KL_h\, 
    \widetilde{\omega}\big(
      \|u^0-\hat{u}^0_d\|_{\mathcal{H}_T^2}
    \big)
    \le 
      % KL_h
      % \, 
    {
            \widetilde{\omega}\biggl(
              \sup_{u^0\in \kkk_0}\,  \|u^0-\hat{u}^0_d\|_{\mathcal{H}_T^2}
            \biggr)
        \le 
            \widetilde{\omega}\Big(
            % KL_h 
                \bar{\omega}\Big(
                    \frac{\varepsilon}{2}
                \Big)
            \Big)
        = 
            \frac{\varepsilon}{2} 
    }
.
\end{align}
It remains to bound~\Cref{t:operator_approximation}.

\paragraph{{Step 2 - Bounding Term~\Cref{t:operator_approximation}}}
By Lemma~\ref{lm:|Ji(u0u1)-Ji(tu0tu1)|<=|u0-tu0|+|u1-tu1|}, we find that
\begin{align} 
    \eqref{t:operator_approximation}
= 
% &
    \big| 
        J_{0}
        (\hat{u}^0_d, U^\star(\hat{u}^0_d))
    -
        J_{0}
        (\hat{u}^0_d, \hat{U}(\hat{u}^0_d)) 
    \big| 
% \\
\leq 
% &  
    C \cdot 
     \big\| U^\star(\hat{u}^0_d) - \hat{U}(\hat{u}^0_d) \big\|_{\calH^2_T}
    \label{eq:operator_approximation__bound_1b} 
.
\end{align} 
By our universal approximation theorem, in Lemma~\ref{lem:UAT_1}, we have that there exists an attentional neural operator $\hat{U}:\mathcal{H}_T^2\to\mathcal{H}_T^2$, as in Definition~\ref{defn:NO}
\begin{align} 
\label{eq:operator_approximation__bound_2}
 \sup_{v\in p_d(\kkk_0)}\,
    \big\| U^\star(v) - \hat{U}(v) \big\|_{\calH^2_T}
\le
    \frac{\varepsilon}{2C}
\end{align} 
where we have set $
\varepsilon_D \eqdef \varepsilon_A \eqdef \varepsilon/4C$.
Thus, the estimate in~\eqref{eq:operator_approximation__bound_2} implies that the right-hand side of~\eqref{eq:operator_approximation__bound_1b} can be bounded above as follows
\begin{align} 
\label{eq:operator_approximation__bound_3}
    \eqref{t:operator_approximation}
\le
% &
    C  
    \,
        \big\| \hat{U}(\hat{u}^0_d) - U^\star(\hat{u}^0_d)   \big\|_{\mathcal{H}_T^2}  
\le 
% &
    C 
    \sup_{v\in p_d(\kkk_0)}\,
    \big\| U^\star(v) - \hat{U}(v) \big\|_{\calH^2_T}
\le 
% &
    C  
    \frac{\varepsilon}{2C}
= 
% &
    \frac{\varepsilon}{2}
.
\end{align} 
Upon combining the estimates in~\eqref{eq:input_truncation__bound_2} and in~\eqref{eq:operator_approximation__bound_3} we obtain the following upper-bound for the right-hand side of~\eqref{eq:tb_controlled}
\begin{equation}
\label{eq:tb_controlled__done}
        \big|
                J_{0}
                (u^0,U^{\star}(u^0))
            -
                J_{0}
                \big(
                    \hat{u}^0_d
                ,
                    \hat{U}(
                    \hat{u}^0_d
                    )
                \big)
        \big|
    \le 
    % & 
        \eqref{t:input_truncation} + \eqref{t:operator_approximation}
    \le 
        % & 
        \frac{\varepsilon}{2} + \frac{\varepsilon}{2}
    = 
    % &
        \varepsilon
.
\end{equation}

\paragraph{Step 3} 

Suppose additionally that $u^0\in \kkk_0$ is such that $(u^0,U^{\star}(u^0))$ is a $(0-)$Stackelberg equilibrium.  
If $(u^0,U^{\star}(u^0))$ is a Stackelberg equilibrium, see Definition~\ref{def:stackelberg}, then this pair is optimal for $J_1$ and the left-hand side of~\eqref{eq:tb_controlled__done} becomes 
\allowdisplaybreaks
\begin{align}
\label{eq:tb_controlled__done___v2}
        0
    \le 
        \big|
                J_{0}(u^0,U^{\star}(u^0))
            -
                J_{0}\big(
                    \hat{u}^0_d
                ,
                    \hat{U}(
                    \hat{u}^0_d
                    )
                \big)
        \big|
    =
            J_{0}\big(
                \hat{u}^0_d
            ,
                \hat{U}(
                \hat{u}^0_d
                )
            \big)
        -
            J_{0}(u^0,U^{\star}(u^0))
.
\end{align}
Consequentially,~\eqref{eq:tb_controlled__done___v2} can be rearranged yielding
\begin{equation}
\label{eq:stackelberg_subcase__1}
        J_{0}\big(
            \hat{u}^0_d
        ,
            \hat{U}(
            \hat{u}^0_d
            )
        \big)
    \le 
        J_{0}(u^0,U^{\star}(u^0))
    +
        \varepsilon
.
\end{equation}
Since the left-hand side of~\eqref{eq:stackelberg_subcase__1} is optimal, then taking infima overall $v$ in $\mathcal{H}_T^2$ does not reduce it further.  Thus,~\eqref{eq:stackelberg_subcase__1} becomes
\begin{equation}
\label{eq:stackelberg_subcase__2}
        \inf_{u^0\in \mathcal{K}_0}\, 
            J_0(u^0,U^{\star}(u^0))
    =
        J_0(u^0,U^{\star}(u^0))
.
\end{equation}
Combining~\eqref{eq:stackelberg_subcase__1} and~\eqref{eq:stackelberg_subcase__2} yields
\[
    J_{0}\big(
        \hat{u}^0_d
    ,
        \hat{U}(
        \hat{u}^0_d
        )
    \big)
\le 
        J_{0}(u^0,U^{\star}(u^0))
    +
        \varepsilon
    =
        \inf_{u^0\in \mathcal{K}_0}\, 
            J_{0}(u^0,U^{\star}(u^0))   
    +
        \varepsilon
.
\]    
This concludes our proofs of Theorems~\ref{thrm:Main__BestResponse} and~\ref{thrm:Objective}.

To obtain the conclusion of Theorem~\ref{thrm:Main__BestResponse___goodrates}, we first note that the non-linear operator $U^{\star}$ is $1/2$-H\"{o}lder continuous.  Since Example~\ref{ex:perturbations__lin_prob_control__Pt2} showed that $K$ is an exponential manifold (in the sense of Definition~\ref{defn:exp_manifold}) then the complexity estimates for the attentional neural operator $\hat{U}$ selected in~\eqref{eq:operator_approximation__bound_2} must be as in Table~\ref{tab:complexity__lemma}.  Consider the special case where $\varepsilon_D=\varepsilon_A$ completes the proof of Theorem~\ref{thrm:Main__BestResponse___goodrates}.
\end{proof}

\section{Acknowledgments}
A.\ Kratsios acknowledges financial support from an NSERC Discovery Grant No.\ RGPIN-2023-04482 and No.\ DGECR-2023-00230.  A.\ Kratsios also acknowledges that resources used in preparing this research were provided, in part, by the Province of Ontario, the Government of Canada through CIFAR, and companies sponsoring the Vector Institute\footnote{\href{https://vectorinstitute.ai/partnerships/current-partners/}{https://vectorinstitute.ai/partnerships/current-partners/}}.
I.\ Ekren is partially funded by NSF grant DMS-2406240.

\bibliography{Bookkeaping/Refs}

\end{document}